\documentclass[preprint]{imsart} %dvips,

\RequirePackage[OT1]{fontenc}
\RequirePackage{amsthm,amssymb,amsfonts,amsmath}
\RequirePackage[numbers]{natbib}
\usepackage{color}
\usepackage{amsmath}
\usepackage{amssymb}
\usepackage{amsfonts}
\usepackage{amsthm}
\usepackage{rotating}
\usepackage{mathtools}
\usepackage{dsfont}
\usepackage{array}
\usepackage{multirow}
\usepackage{multicol}

\usepackage{natbib}

\usepackage{graphicx}

\usepackage{tabularx}
\usepackage{longtable}
\usepackage{lscape}
\usepackage{float}

 \usepackage{hyperref}
 \usepackage{tikz}%pour faire des dessins avec latex
\usetikzlibrary{arrows}
\newtheorem{theo}{Theorem}[section]

\newtheorem{lemma}{Lemma}[section]
\newtheorem{rem}{Remark}[section]

\newtheorem{prop}{Proposition}[section]

\newcommand{\KL}{\mbox{KL}}
\newcommand{\argmin}{\mbox{argmin}}
\newcommand{\HS}{(\mathbf{H_{\mathcal{S}}})}%\newcommand{\HS}{(\mathbf{H^{{\tiny(3,2)}}_{\mathcal{S}}})}%symbole pour les densités smooths

\newcommand{\HL}{(\mathbf{H_{Lip}})}%symbole pour la condition Lip
\newcommand{\HD}{(\mathbf{H_{D}})}%symbole pour la condition Lip
\newcommand{\Th}{\Theta_{n}(c)}                                   %ancienne version {\Theta_{n}(M,c,\overline{\lambda})}

\newcommand{\Thh}{\Theta_n(m,c)}                                                                  %ancienne version {\Theta_{n}(m,M,c,\overline{\lambda})}
\newcommand{\Thhb}{\Theta_{n}(m,M,c,\overline{\lambda})}

\newcommand{\TThl}{\Theta_{n}(M,(\ell_n)_n,\overline{\lambda})}                                                       %{\widetilde{\Theta}_{n}(M,(\ell_n),\overline{\lambda})}
\newcommand{\MM}{\mathcal{M}_{\Lambda,\mathfrak{M}}}
\newcommand{\MMn}{\mathcal{M}_{\Lambda_n,\mathfrak{M}_n}}
\newcommand{\ii}{\mathfrak{i}}

\newcommand{\PP}{\ensuremath{\mathbb{P}} }
\newcommand{\RR}{\ensuremath{\mathbb{R}} }

\newcommand{\NN}{\ensuremath{\mathbb{N}} }

\newcommand{\ps}{\ensuremath{\bullet} }

\newcommand{\tr}{\ensuremath{\tilde{r}} }
\newcommand{\te}{\ensuremath{\tilde{e}} }

\newcommand{\EE}{\ensuremath{{\mathbb E}}}

\newcommand{\cB}{{\cal B}}
\newcommand{\cC}{{\cal C}}

\newcommand{\cG}{{\cal G}}

\newcommand{\cI}{{\cal I}}
\newcommand{\cJ}{{\cal J}}

\newcommand{\cV}{{\cal V}}

\newcommand{\LL}{\mathbb{L}}
\startlocaldefs
\numberwithin{equation}{section}
\theoremstyle{plain}

\endlocaldefs

\begin{document}

\begin{frontmatter}
\title{Parameter recovery in two-component contamination mixtures: the $L^2$ strategy }    %$\mathbb{L}^2$ ca ne compile pas
\runtitle{Another $\mathbb{L}^2$ look at two-component contamination mixture}

\begin{aug}
\author{\fnms{S\'ebastien} \snm{Gadat}\ead[label=e1]{sebastien.gadat@math.univ-toulouse.fr}} 
\and 
\author{\fnms{Jonas} \snm{Kahn}\ead[label=e4]{jonas.kahn@math.univ-toulouse.fr}} 
\and
\author{\fnms{Cl\'ement} \snm{Marteau}\ead[label=e3]{marteau@math.univ-lyon1.fr}}
\and
\author{\fnms{Cathy} \snm{Maugis-Rabusseau}\ead[label=e2]{cathy.maugis@insa-toulouse.fr}}

\runauthor{Gadat et al.}

\affiliation{}%Toulouse School of Economics, University of Toulouse I Capitole}
\address{Toulouse School of Economics\\ 
 Universit\'e Toulouse 1 - Capitole.\\
21 all\'ees de Brienne,\\ 31000 Toulouse, France.\\
\printead{e1}\\}

\affiliation{}%Institut Math\'ematiques de Toulouse, University Toulouse III Paul Sabatier}
\address{Institut de Math\'ematiques de Toulouse ; UMR5219\\
Universit\'e de Toulouse ; CNRS\\
UPS, F-31062 Toulouse Cedex 9, France\\
\printead{e4}\\}

\affiliation{}%Institut Camille Jordan, University of Lyon 1 Claude Bernard}
\address{Univ Lyon, Universit\'e Claude Bernard Lyon 1\\
CNRS UMR 5208, Institut Camille Jordan\\
F-69622 Villeurbanne, France. \\
\printead{e3}\\}

\affiliation{}%Institut Math\'ematiques de Toulouse, INSA Toulouse}
\address{Institut de Math\'ematiques de Toulouse ; UMR5219\\
Universit\'e de Toulouse ; CNRS\\
INSA, F-31077 Toulouse, France\\
\printead{e2}\\}
\end{aug}

\begin{abstract}
In this paper, we consider a parametric density contamination model. We work with a sample of i.i.d. data with a common density, $f^\star =(1-\lambda^\star) \phi + \lambda^\star \phi(.-\mu^\star)$, where the shape $\phi$ is assumed to be known. We establish the optimal rates of convergence for  the estimation of the mixture parameters $(\lambda^\star,\mu^\star)\in (0,1)\times \mathbb{R}^d$.  In particular, we prove that the classical parametric rate $1/\sqrt{n}$ cannot be reached when at least one of these parameters is allowed to tend to $0$ with $n$. 
\end{abstract}

\begin{keyword}[class=AMS]
\kwd[Primary ]{62G05}
\kwd{62F15}
\kwd[; secondary ]{62G20}
\end{keyword}

\begin{keyword}
\kwd{$\mathbb{L}^2$ contrast, parameter estimation, rate of convergence, two-component contamination mixture model}
\end{keyword}

\end{frontmatter}

\section{Introduction}
Because of their wide range of flexibility, finite mixtures are a popular tool to model the unknown distribution of heterogeneous data. They are found in several domains and have been at the core of several mathematical investigations. For a complete introduction to mixtures, we refer the reader to \cite{McLP_2000} and \cite{FS_2006}. In most cases of interest, a sample $\mathcal{S}_n:=(X_1,\ldots, X_n)$ of i.i.d. data is at our disposal, and each entry admits the probability density $f^\star$ w.r.t. the Lebesgue measure. For a finite mixture model, the density $f^\star$ is assumed to have the following shape:
\begin{equation}
f^\star = \sum_{k=1}^K \lambda_k \phi_k.
\label{eq:densiteK}
\end{equation}
With such a representation, the population of interest can in some sense be decomposed into $K$ different groups where each group $k$ has a proportion $\lambda_k$ and is distributed according to the density $\phi_k$. For practical purposes, parametric models are often considered. In such cases, the densities $\phi_k$ are assumed to be known, at least up to some finite parameters, and the parameter estimation problem is often addressed using an EM-type algorithm \cite{DLR_1977}. In contrast, with the impressive range of applications based on mixtures, theoretical issues related to mixture models are somewhat poorly understood.

Among the available theoretical results for mixtures, some of them are particularly linked to the density estimation problem. The works \cite{GW_2000}, \cite{GV_2001} and \cite{KRV_2010}
develop a nonparametric Bayesian point of view, while exploiting both the approximation capacity of mixtures and their metric entropy size, first with Gaussian distributions and later with exponential power distributions. A Gaussian mixture estimator based on a non asymptotic penalized likelihood criterion is proposed in \cite{MM_2011} and the adaptive properties of this estimator are investigated in \cite{MM_2013}. 
 
In the mixture models, the focus on the parameters themselves has received less theoretical attention because of their great mathematical difficulty despite their natural interest. It is indeed highly informative to obtain the estimation of the mixing distribution, and many applied works use this estimation for descriptive statistics. Among them, the unsupervised clustering with Bayesian interpretation is certainly one of the most widely used applications of mixtures (see, \textit{e.g}, \cite{McLP_2000}). Given a dictionary of densities, \cite{BTWB_2010} propose an estimation procedure based on the minimization of an $\mathbb{L}^2$ empirical criterion with a sparsity constraint, providing an estimation of the parameters of interest when the location parameters $\mu_k^\star$ (here $\phi_k=\phi(.-\mu_k^\star)$) are not too close to each other.  \cite{Chen_95} studied the estimation of the mixing distribution under a strong identifiability condition. As observed in the recent works of \cite{Nguyen_2013}, \cite{ho2016} and  \cite{KH15}, the optimal rate depends on the knowledge of  the number of components. \cite{Ho:2015aa} show that the parameter estimation rates are slower for some weakly identifiable mixtures. Other extensions are available in \cite{ho2016}. Identifiability (and estimation) issues are discussed in \cite{Hunter} under the assumption that the $\phi_k$ can be written as $\phi_k = \phi(.-\mu_k)$ for some sequence $(\mu_k)_{k=1..K}$ and a symmetric probability density $\phi$. 

Finally, the EM algorithm (see, \textit{e.g.}, \cite{DLR_1977}) is a popular alternative for the analysis of the latent structures involved in the mixture models, but the analysis of the convergence rate of the final estimator is somewhat intricate. A first positive result about the \textit{convergence} of this method is given in \cite{W_83} when the density is unimodal and certain smoothness conditions hold. However, when multimodality occurs, the behavior of the EM method remains mysterious and is suspected to fall into local traps of the log-likelihood. Some recent advances in the analysis of this famous method were brought by \cite{BWY_16}, where a general result is given for a convergence of the sample-based EM towards the population one, up to initialization, Lipschitz and concavity conditions. 

In this paper, we focus on the multivariate parameter estimation problem when the density of interest is a two-component contamination mixture:
%\begin{equation}
$$
f^\star = (1-\lambda^\star) \phi + \lambda^\star \phi(.-\mu^\star),
$$
%\label{eq:densite}
%\end{equation}
where the density $\phi$ is \textit{known} and the parameters $(\lambda^\star,\mu^\star) \in (0,1) \times \mathbb{R}^d$ are to be estimated. This model is a particular case of the Huber contamination model (\cite{Huber64}).

The estimation of the couple $(\lambda^\star,\mu^\star)$ has already been considered in the literature. In \cite{BMV_2006}, a slightly different model is considered where $f^\star = (1-\lambda^\star) \phi(.-\mu_1^\star) + \lambda^\star \phi(.-\mu_2^\star)$ and $\phi$ is assumed to be symmetric and unknown. Using a recurrence procedure based on an inversion formula, they propose an estimator for $\theta^\star = (\lambda^\star, \mu_1^\star, \mu_2^\star)$ and the function $\phi$. In particular, the parameter $\lambda^\star$ is estimated at the `classical' parametric rate $1/\sqrt{n}$, while the rate $n^{-1/4}$ is obtained for location parameters $(\mu_1^\star, \mu_2^\star)$. A similar problem is addressed in \cite{BV_2014} where the rate $1/\sqrt{n}$ is reached for the estimation of the whole parameter $\theta^\star$. The estimation procedure is based on a computation of an empirical Fourier transform. More recently, \cite{Patra_Sen} considered the situation where the distribution of one of the component of the mixture is known. In such a case, they provide an estimator of both the mixing parameter and of the distribution of the second component. In the setting considered here (i.e., when $f^\star$ is a two-component contamination mixture),
%defined as in (\ref{eq:densite})
\cite{CJL_2007} proposes an iterative procedure based on the empirical distribution function. In the so-called \textit{sparse} setting where\footnote{All the notation used in this paper are made precise at the end of this section.} $\lambda^\star \ll 1/\sqrt{n}$ and $\mu^\star \sim \sqrt{2r\log(n)}$ for some $r\in (0,1)$ as $n\rightarrow +\infty$, the authors derive rates of convergence for the estimation of $\lambda^\star$. In particular, they prove that the classical parametric rate cannot be attained in such a setting. 

In all the aforementioned contributions except \cite{CJL_2007}, it is implicitly assumed that both location and proportion parameters are fixed with respect to $n$. The main aim of this paper is to fill this gap. We propose a procedure inspired by \cite{BTWB_2010} and derive an estimator $(\hat\lambda_n,\hat\mu_n)$ for the couple $(\lambda^\star,\mu^\star)$. This estimator is based on the minimization of an $\mathbb{L}^2$ contrast instead of a usual maximum likelihood estimator of mixture parameters computed with an EM-type algorithm. Then, given a bound $M$ s.t. $\max_{j=1\dots d}|\mu_j^\star |\leq M$ and under mild assumptions on the shape $\phi$, we prove that:
\begin{equation}
\sup_{(\lambda^\star,\mu^\star)\in (0,1)\times [-M,M]^d} \mathbb{E}_{\lambda^\star,\mu^\star} [ (\lambda^\star)^2 \|\mu^\star\|^2 \|\hat\mu_n-\mu^\star\|^2] \lesssim \frac{\log^2n}{n},
\label{eq:result1}
\end{equation}
and 
\begin{equation}
\sup_{\stackrel{(\lambda^\star,\mu^\star)\in (0,1)\times [-M,M]^d }{\lambda^\star \|\mu^\star\|^2 \gtrsim n^{-1/2}}} \mathbb{E}_{\lambda^\star,\mu^\star} [ \|\mu^\star\|^4 (\hat\lambda_n-\lambda^\star)^2] \lesssim \frac{\log^2n}{n},
\label{eq:result2}
\end{equation}
These results are completed by the corresponding lower bounds that ensure the optimality of \eqref{eq:result1} and \eqref{eq:result2}, up to logarithmic factors. In particular, we can immediately observe that the parametric rate of $1/\sqrt{n}$ is attained when $\lambda^\star$ and $\mu^\star$ are fixed, but is deteriorated as soon as these parameters are allowed to tend to $0$ with $n$. 

Finally, we also obtain an interesting link between the $\mathbb{L}^2$ loss and the Wasserstein loss in our contamination mixture model:
\begin{equation}\label{eq:contriblw}
\|f_{\lambda,\mu}-f_{\lambda',\mu'}\|_2  
\geq c_\phi W_2^{2}(G_{\lambda,\mu},G_{\lambda',\mu'}),
\end{equation}
where the Wasserstein ($L^p$)-transportation distances between two probability measures $m_1$ and $m_2$ on $\Omega $ are defined by
\begin{equation}
\label{defW}
    W_p(m_1, m_2)  := \left[\inf_{\pi \in \Pi(m_1,m_2)} \int d^p(x,y) \mathrm{d}\pi(x,y)\right]^{\frac 1 p},
\end{equation}
$\Pi(m_1,m_2)$ being the set of probability measures on $\Omega \times \Omega$ such that their marginals are $m_1$  and $m_2$; and $G_{\lambda, \mu}  = (1 - \lambda) \delta_0 + \lambda \delta_{\mu}$ is the mixing distribution associated to the density $f_{\lambda,\mu}$, where $\delta_{\theta}$ is the Dirac peak at $\theta$. 
%More generally, for a mixture $f = \sum_i \lambda_i \phi_{\theta_i} $ where the components are parametrized by $\theta \in \Theta$, the mixing distribution $G$ is $ \sum_i  \lambda_i \delta_{\theta_i} $.
This makes even more explicit the hardness of recovering the unknown parameters of the contamination mixture model.

The paper is organized as follows. First, a preliminary oracle inequality for $\mathbb L^2$ density estimation is established in Section~\ref{s:model}. 
On the basis of this result, some rates of convergence for the estimation of $(\lambda^\star,\mu^\star)$ are deduced (see Section~\ref{sub:cvrates}) under some assumptions on $\phi$ presented in Section~\ref{sec:baseline}. 
Some lower bounds are provided in Section~\ref{s:lb}, first in a strong contamination model ($\|\mu^\star\|>m$ with $m$ independent of $n$; see Section~\ref{sub:lb1}); and second, in a weak contamination model ($\|\mu\|$ can tend to $0$ when $n\to +\infty$; see Section~\ref{sub:lb2}).
The main part of the paper ends with a discussion in Section~\ref{sec:discussion} that reveals several insights between Wasserstein distances among mixing distributions and distances between the probability distributions. A few simulations are presented in Section~\ref{sec:example_distribution}.
 Proofs of the upper bounds (resp. lower bounds) are given in Section~\ref{s:proof:ub} (resp. Appendix~\ref{sect:LWB}) while Section \ref{sec:metric_proof} provides the proof of the link between some Wasserstein transportation cost among mixing distributions and the $\mathbb{L}^2$ loss.
Technical results are presented in Appendix~\ref{s:tec}, whereas Appendix ~\ref{s:CS} is devoted to a needed refinement of the Cauchy-Schwarz inequality. 
\\

\textbf{Notation.} Above and below, we use in this paper some specific notation. For any real sequences $(u_n)_{n\in \mathbb{N}}$ and $(v_n)_{n\in \mathbb{N}}$, $u_n \ll v_n$ means that $u_n/v_n \rightarrow 0$ as $n\rightarrow +\infty$. Similarly, $u_n \sim v_n$ (resp. $u_n \lesssim v_n$ and $u_n \gtrsim v_n$) means that there exists $a,b \in \mathbb{R}^+$ such that $a v_n \leq u_n \leq b v_n$ (resp. $u_n \leq b v_n$ and $av_n \leq u_n$) for any $n\in \mathbb{N}$. For any $x \in \mathbb{R}^d$, $\| x \|$ will denote the classical euclidian norm (namely $\| x \|^2 = \sum_{j=1}^d x_j^2$) while $\| f \|_2$ will denote the $\mathbb{L}^2$ norm of any $f \in \mathbb{L}^2(\mathbb{R}^d)$ associated to the corresponding scalar product $\langle .,. \rangle$. Finally, $\mathbb{P}_\theta$ will alternatively (the meaning will be clear following the context) correspond to the measure of a single observation $X_i$ or of the whole sample $(X_1,\dots, X_n)$ associated to any mixture parameter $\theta = (\lambda,\mu)$. The associated expectation will be alternatively denoted by $\mathbb{E}_\theta$, $\mathbb{E}_{\lambda,\mu}$ or $\mathbb{E}$, according to the context.

%%%%%%%%%%%%%%%%%%%%%%%%%%%%%%%%%%%%%%%%%%%%%%%%%%%%%%%%%%
\section{A preliminary result on $\mathbb{L}^2$ density estimation}
\label{s:model}

\subsection{Statistical setting and identifiability}
\label{s:setting}

We recall that we have at our disposal an i.i.d. sample of size $n$ denoted $\mathcal{S}_n:=(X_1,\ldots, X_n)$, where the distribution of each $X_i$ is associated with a two-component contamination mixture model. 
More precisely, we assume that each $X_i$ admits an unknown density $f^\star$ with respect to the Lebesgue measure on $\RR^d$, which is given by:
\begin{equation}\label{eq:fmodel}
f^\star = (1-\lambda^{\star}) \phi+\lambda^{\star} \phi(.-\mu^\star).
\end{equation}
In the following text, $ \theta^\star=(\lambda^{\star},\mu^\star) \in (0,1) \times \mathbb{R}^d$ refers to the \textit{parameters} of the two-component contamination mixture model. We assume that the density $\phi$ is a \textit{known} function and that a real contamination of this baseline density $\phi$ occurs ($\lambda^\star>0$). Finally, we assume that the unknown contamination shift $\mu^{\star}$ belongs to a bounded interval $[-M,M]^d$ where $M>0$ is known.\\

Here and below, for any $\theta=(\lambda,\mu) \in (0,1)\times \mathbb{R}^d$, we write: 
$$ f_\theta = f_{\lambda,\mu} = (1-\lambda) \phi + \lambda\phi_\mu,$$
where $\phi_{\mu}$ is defined according to the standard notation in location models:
$$
\forall \mu \in \RR^d \qquad \phi_{\mu}:x \longmapsto \phi(x-\mu). 
$$
In particular, as a slight abuse of notation, we write $f^\star = f_{\theta^\star}=f_{\lambda^\star,\mu^\star}$ and (when the meaning is clear following the context) $\hat f = f_{\hat \theta} = f_{\hat\lambda,\hat\mu}$ for any estimator $\hat\theta$ of $\theta^\star$. \\

We aim to recover the unknown parameter $\theta^\star$ from the sample $\mathcal{S}_n$. This might be possible according to the next identifiability result, whose proof is given in Appendix~\ref{s:tec}.
\begin{prop}\label{prop:id}
Any two-component contamination mixture model is identifiable: $f_{\theta_1}=f_{\theta_2}$ if and only if $\theta_1=\theta_2$.
\end{prop}
Such an identifiability result is well known in some more general cases up to additional assumptions on the baseline density $\phi$ (see, \textit{e.g.}, \cite{Hunter} or Theorem 2.1 of \cite{BMV_2006} where the symmetry of $\phi$ is added to ensure the identifiability of the general mixture model without contamination). Here, the fact that one of the components of the mixture is constrained to be centered makes it possible to get rid of any additional assumption on $\phi$. In particular, Proposition \ref{prop:id} holds as soon as $\phi$ is non-negative with $\int_{\RR^d} \phi = 1$.

\subsection{Estimation strategy and oracle inequality on the $\LL^2$ norms}

Our estimator will be built according to an optimal $\LL^2$ density estimation constrained to the contamination models. For this purpose, we first  define a grid over the possible values of $\lambda$ and $\mu$ through:
$$
\MM  := \left\{ (\lambda,\mu) \,:  \lambda \in \Lambda=\{\lambda_1,\ldots,\lambda_{p}\} \, \ \text{and} \, \ \mu\in \mathfrak{M} =\{\mu_1,\ldots,\mu_{q}\}\right\},
$$
where $\Lambda,\mathfrak{M}$ will depend on $n$ to obtain good properties both from the statistical and approximation point of view. 
To obtain a good estimation of $f^\star$ and $\theta^\star$, we adopt a SURE  approach (see, \textit{e.g.}, \cite{SURE}) and  choose an estimator that  minimizes
%\footnote{In the following, $\| . \|_2$ denotes the norm associated to the scalar product $\langle g_1,g_2 \rangle = \int_{\mathbb{R}^d} g_1(x) g_2(x) dx$ for all $g_1,g_2 \in \mathbb{L}^2(\mathbb{R}^d)$} 
$ \|f^\star-f_{\lambda,\mu}\|_2^2$ over the grid $\MM$. Observing that:
$$\|f^\star-f_{\lambda,\mu}\|_2^2-\|f^\star\|_2^2 = - 2 \langle f^\star,f_{\lambda,\mu}\rangle + 
\|f_{\lambda,\mu}\|_2^2,$$
and since $\|f^\star\|_2^2$ does not depend on $(\lambda,\mu)$, it is natural to introduce the following contrast function:
$$
\forall (\lambda,\mu)  \in \MM \qquad 
\gamma_n(\lambda,\mu) := -\frac{2}{n} \sum_{i=1}^n f_{\lambda,\mu}(X_i)+\|f_{\lambda,\mu}\|_2^2,
$$
leading to the estimator:
\begin{equation}
(\hat \lambda_n,\hat \mu_n) = \arg\min_{(\lambda,\mu) \in \MM} \gamma_n(\lambda,\mu).
\label{eq:estimator}
\end{equation}
Our first main result, stated below, quantifies the performances of $\hat f_n :=f_{\hat \lambda_n,\hat \mu_n}$. 

\begin{theo}
\label{th:oracle}
Let $(\lambda^\star, \mu^\star)\in (0,1)\times \RR^d$. Let $(\hat\lambda_n, \hat\mu_n)$ be the estimator defined in (\ref{eq:estimator}). Then, a positive constant $\mathcal{C}$ exists such that for all $0<\alpha < 1$:
\begin{equation}\label{eq:oracle}\mathbb{E}\left[ \| \hat f_n - f^\star \|_2^2 \right] \leq \left( \frac{1+\alpha}{1-\alpha} \right) \inf_{(\lambda,\mu) \in \MM} \| f_{\lambda,\mu} - f^\star \|_2^2 +  \frac{\mathcal{C} }{2\alpha}\frac{\log^2(|\MM|)}{n}, \end{equation}
where $|\MM|$ corresponds to the cardinality of the grid $\MM$. 
\end{theo}

It is worth mentioning that the result above is almost assumption-free on the two-component contamination mixture model. Nevertheless, this result implicitly requires that the approximation term $\inf_{(\lambda,\mu) \in \MM} \| f_{\lambda,\mu} - f^\star \|_2^2$ is comparable to the residual. In practice, this cannot be achieved unless we have an upper bound on the range for possible values of $\mu$ at our disposal.  The proof of Theorem \ref{th:oracle} is given in Section \ref{s:oracle}.

We stress that Theorem \ref{th:oracle} is not the main interest of our work. It is a minimal requirement to further extend our analysis on the parameter estimation of the mixture models themselves.  In particular, the following question now arises: \textit{does the fact that $\hat f_n$ is a ``good" $\mathbb{L}^2$ estimator of $f^\star$ imply that the corresponding $\hat\theta_n$ provides a satisfying estimator of $\theta^\star$?} The positive answer to this question is the main contribution of our work and is described in the next section. In order to establish this result, some mild restrictions on the class of possible densities $\phi$ are required.

\section{Estimation of the parameter $\theta^\star$}

\subsection{Baseline assumptions}\label{sec:baseline}

We now introduce mild and sufficient assumptions for an optimal recovery of $\theta^\star$ from the oracle inequality (\ref{eq:oracle}) (in terms of convergence rates). In the following, we denote by $\mathcal{C}^{k}(\RR^d)$ the set of continuous functions that admits  $k$ continuous derivatives.

%\paragraph{Assumption $\HS$} \textit{The density $\phi$ fulfills one of the two assumptions: 
%\begin{itemize}
%\item[$\HSone$] $\phi$ is symmetric and belongs to $\mathcal{C}^3_{p}(\RR^d) \cap \LL^2(\RR)$
%\item[] or
%\item[$\HStwo$] $\phi$ belongs to $\mathcal{C}^3(\RR) \cap \LL^2(\RR)$.
%\end{itemize}

\paragraph{Assumption $\HS$} \textit{The density $\phi$ belongs to $\mathcal{C}^3(\RR^d) \cap \LL^2(\RR^d)$}.

The set of admissible densities considered in Assumption $\HS$ is very large, and contains many possible distributions (Gaussian, Cauchy, Gamma  to name a few).   Note that it is also possible to relax the smoothness assumption and handle piecewise differentiable densities with an additional symmetry assumption (see  Appendix~\ref{s:tec}). Note that since the density $\phi$ is continuous and in $\LL^2(\RR^d)$, this density is necessarily bounded on $\RR^d$.

Our second important assumption is concerned with a tight link that may exist between $\phi-\phi_{\mu}$ and $\mu$ itself. It requires a type of Lipschitz upper bound in the translation model.

\paragraph{Assumption $\HL$} \textit{The density $\phi$ satisfies:}
\begin{equation}\label{eq:glip}
\exists \, g \in \LL^2(\RR^d) \quad \forall x \in \RR^d \quad \forall \mu \in [-M,M]^d \qquad |\phi(x)-\phi_{\mu}(x)| \leq \|\mu\| g(x),
\end{equation}
\textit{and $g$ satisfies the integrability condition:}
$$
\mathcal{J}:= \int_{\RR^d} g^{2}(x) \phi^{-1}(x) dx< + \infty.
$$
This assumption will be of primary importance to obtain estimation results on the parameters of the mixture themselves. In particular, it will make it possible to derive a relationship between  the $\LL^2$ norm  of $\phi-\phi_{\mu}$ and the size of $\|\mu\|$. Hence, under Assumption $\HL$, a good estimation of the density $f^{\star}$ for the $\LL^2$ norm is assumed to yield a good estimation of the mixture parameters.

\begin{rem}
\label{rem:1}
 Instead of listing all the possible densities that both meet Assumptions $\HS$, $\HL$  (and later $\HD$ introduced in Section~\ref{sub:lb2} %Appendix \ref{sect:LWB} 
 for our lower bound results), we will show that \emph{any log-concave} distribution $\phi$ written as: 
$$
 \phi(.)=e^{-u(.)}  \quad \text{with} \; u \text{ convex such that} \quad  \|\nabla u\|+\|D^2 u\|=o_{ \infty}(u) ,
$$ 
satisfies these three  conditions\footnote{Hereafter $o_{\infty}(u)$ denotes a quantity negligible compared to $u(x)$ as $\| x\| \rightarrow +\infty$}. The relationships between $\HS,\HL$, $\HD$ and the log-concave distributions are given in Appendix \ref{s:log_concave}.  
\end{rem}

\begin{rem}
An easy consequence of Remark \ref{rem:1} (see also Proposition \ref{prop:lc}) is that the log-concave Gaussian   distributions satisfy assumptions $\HS$ and  $\HL$  so that all the results displayed below apply to these situations. It may be shown as well that our results apply for the Laplace distribution since the smoothness assumption $\HS$ may be replaced by a symmetry property (see  Appendix~\ref{s:tec}).

In the $1$-dimensional Cauchy distribution case, we can compute $\phi-\phi_{\mu}$:
$$
\left| \phi(x)-\phi_{\mu}(x)\right| = |\mu| \frac{|2x-\mu |}{\pi [1+(x-\mu)^2][1+x^2]} \leq C \phi(x) |\mu|,
$$
for a large enough constant $C$. Hence, the assumptions  $\HS$ and $\HL$ are satisfied with $g = C \phi$ for the Cauchy distribution.

The skew Gaussian density\footnote{It is defined as $\phi(.) = 2\psi(.)\Psi(\alpha .)$ where $\psi$ and $\Psi$ denote respectively the density and cumulative function of a standard Gaussian distribution, and $\alpha$ an asymmetry parameter.} $\phi$ satisfies:
$$
\left| \phi(x)-\phi_{\mu}(x)\right| \leq 2 \psi(x) \left| \Psi(\alpha x ) - \Psi(\alpha(x-\mu))\right| + 2 \Psi(\alpha (x-\mu))
\left| \psi( x ) - \psi(x-\mu)\right|.
$$
If we define $g$ as $g(x) :=  4 \sup_{[x-M;x+M]} \psi(t) \times \sup_{[x-M;x+M]} \Psi(\alpha t) $, we can check that $\HS$ and $\HL$ are satisfied. In particular, the integrability condition $\HL$ is satisfied for large $x$ because $\Psi(\alpha x) \longrightarrow 1$ when $x \longrightarrow + \infty$. Conversely, if $x \longrightarrow - \infty$,  we have:
\begin{eqnarray*}
g^2(x) \phi^{-1}(x)& \lesssim &\left[ \psi^{-1}(x) \Psi^{-1}(\alpha x) \right]  \sup_{[x-M;x+M]} \psi^2(t) \times \sup_{[x-M;x+M]} \Psi^2 (\alpha t) \\ &\lesssim & \left[ \alpha x  e^{x^2/2} e^{\alpha^2 x^2/2} \right] e^{-(x-M)^2} \times e^{-\alpha^2 (x-M)^2} [\alpha(x-M)]^{-2}\\
& \lesssim & e^{-(x-2M)^2/4}  e^{-\alpha^2 (x-2M)^2/4},
\end{eqnarray*}
which leads to the integrability condition around $-\infty$.
\end{rem}

In the following text, we maintain a formalism that uses the two assumptions of Section \ref{sec:baseline} for the sake of generality.

\subsection{Consistency rates on the parameters $(\lambda^\star,\mu^\star)$}
\label{sub:cvrates}

We now use our assumptions on $\phi$ to deduce some rates of convergence for the estimation of the couple $(\lambda^\star,\mu^\star)$ from the oracle inequality of Theorem \ref{th:oracle}. According to the assumption $\mu^\star \in  [-M,M]^d$ for some given $M>0$, we define the grid $\mathcal{M}_n = \MM$ as:
\begin{eqnarray}
\lefteqn{\mathcal{M}_n = \left\lbrace (\lambda,\mu): \ \lambda=\frac{i}{\sqrt{n}}, \mu= (\mu^{(1)},\ldots,\mu^{(d)})  \quad \text{with} \quad \mu^{(j)} = \pm \frac{k_j}{\sqrt{n}} \right.}\nonumber\\
& & \hspace{0.5cm}  \left. \mathrm{where} \quad i \in \lbrace 1,\dots, \sqrt{n} \rbrace, \, j \in \lbrace 1,\dots, d \rbrace, \, k_j\in \lbrace  1,\dots, M\sqrt{n} \rbrace  \large\right\}, \label{eq:grille}
 \end{eqnarray}
so that the approximation term $\inf_{(\lambda,\mu)\in \mathcal{M}_n} \|f_{\lambda,\mu}-f^\star\|_2^2$ in Equation \eqref{eq:oracle} can be made lower than $n^{-1}$, while keeping the size of $\log(|\mathcal{M}_n|)$ reasonable and of order $d \log(n)$.
The next result, whose proof is given in Section \ref{s:preuvemu}, explicitly gives a non-asymptotic consistency rate of the estimation of $\mu^{\star}$ in terms of the sample size $n$, of the amount of  contamination $\mu^\star$, and of the probability  $\lambda^\star$ of this contamination itself.

\begin{theo}
\label{thm:vitesse_mu}  Let $(\hat \lambda_n,\hat \mu_n)$ be the estimator defined in \eqref{eq:estimator} with $\mathcal{M}_n$ given in (\ref{eq:grille}). If $\phi$ satisfies Assumptions $\HS$ and $\HL$, a positive constant $C_1$ exists such that:
$$\forall n \in \NN \qquad  \sup_{(\lambda^\star,\mu^\star) \in (0,1) \times [-M,M]^d} \EE_{\lambda^\star,\mu^\star} \left[(\lambda^\star \| \mu^\star\| )^2 \| \hat \mu_n - \mu^\star\| ^2 \right] \leq \frac{C_1 \log^{2}n}{n}.$$
\end{theo}

In the 1-dimensional case ($d=1$), an immediate consequence of Theorem \ref{thm:vitesse_mu} is that for a \textit{fixed} couple $(\lambda^\star,\mu^\star) \in ]0,1[ \times \mathbb{R}\setminus\lbrace 0\rbrace$: 
$$
 \EE_{\lambda^\star,\mu^\star} \left[\left(\frac{\hat \mu_n}{\mu^\star}-1 \right)^2 \right] \leq \frac{C_1 \log^{2}n}{n\{\lambda^\star\}^2 \{\mu^\star\}^4}.
$$
In particular, since $\mu^\star$ is allowed to tend to $0$ with $n$, the estimator $\hat \mu_n$ will be consistent as soon as $\sqrt{n}\lambda^\star \{\mu^\star\}^2 \rightarrow +\infty$ as $n\rightarrow +\infty$. In a detection context,  a two-component mixture distribution can be distinguished from that of a single component as soon as $\sqrt{n} \lambda^\star |\mu^\star| > \mathcal{C}$ for some positive constant $\mathcal{C}$ (see, e.g., \cite{CJJ_2011} or \cite{LMM_2016}).  Naturally, detection is ``easier" than estimation in the sense that the first task requires weaker conditions on the parameters of interest than the second. Since the contamination level $\mu^\star$ is assumed to be upper bounded, it is worth observing that we implicitly require that $\lambda^\star \gg 1/\sqrt{n}$ as $n\rightarrow +\infty$.   \\

Before checking the optimality of this result (see Section \ref{s:lb}), we investigate the estimation of the contamination proportion  $\lambda^{\star}$. According to the previous discussion, we will assume that $\lambda^\star \|\mu^\star\|^2$ is significantly larger than $n^{-1/2} \log^2 n $. This ensures that the contamination level $\mu^\star$ is consistently estimated. For this purpose, we introduce the set $\TThl$ indexed by a sequence $(\ell_n)_n$:
$$
\TThl:= \left\{ \theta=(\lambda,\mu) \, :  \frac{\ell_n}{\|\mu\|^2 \sqrt{n}} \leq \lambda \leq \overline{\lambda}, \|\mu\|_{\infty}\leq M \right\},
$$
for some $\bar\lambda \in (0,1)$.

\begin{theo}
\label{thm:vitesse_lambda}If $\phi$ satisfies Assumptions $\HS$ and $\HL$ and the sequence $(\ell_n)_n$ is such that $ \lim_{n \rightarrow + \infty} \frac{\ell_n}{\log n} =+\infty$,  then a positive constant $C_2$ exists such that:
$$ \sup_{(\lambda^\star,\mu^\star) \in \TThl} \EE_{\lambda^\star,\mu^\star} \left[\|\mu^\star\|^4(\hat \lambda_n - \lambda^\star)^2 \right] \leq \frac{C_2 \log^2 n}{n}.$$
\end{theo}

The proof is given in Section \ref{s:preuve_l}. Once again, we can immediately deduce from this bound that: 
$$ \EE_{\lambda^\star,\mu^\star} \left[\left(\frac{\hat \lambda_n}{\lambda^\star}-1 \right)^2 \right] \leq \frac{C_2 \log^{2}n}{n\{\lambda^\star\}^2 \|\mu^\star\|^4},$$
which only makes sense when $\sqrt{n}\lambda^\star \|\mu^\star\|^2 \rightarrow +\infty$ as $n\rightarrow +\infty$. We stress that in the particular case of fixed $\lambda^\star$ and $\mu^\star$ (w.r.t. $n$), these quantities can be estimated at the classical parametric rate of $1/\sqrt{n}$ (up to a logarithmic term). 

\begin{rem}
The upper bounds displayed in Theorems \ref{thm:vitesse_mu} and \ref{thm:vitesse_lambda} both involve a $(\log(n))^2$ term. This logarithmic term comes from the oracle inequality in Theorem 2.1 and is related to the complexity of the set, namely $\MM$, over which our contrast is minimized. As we will see in the next section, such a term is missing from our lower bound. Up to our knowledge, a logarithmic gap between lower and upper bounds is a classical outcome when dealing with contrast minimization estimators.  
\end{rem}

\section{Lower bounds}\label{s:lb}

We now derive some lower bounds on the estimation of $\lambda^\star$ and $\mu^\star$ and show that our previous results are \textit{minimax optimal} with respect to the values of $n$, $\lambda^\star$ and $\mu^\star$ up to some $\log^2 n$ terms. 

\subsection{Strong contamination model}
\label{sub:lb1}

For this purpose, we split our study into two cases and first consider the ``standard" situation of a strong contamination, meaning that $\|\mu^\star\|$ is bounded from below by a constant independent on $n$: it translates the fact that the contamination is not negligible when $n \longrightarrow +\infty$. 
Let $m$ and $c$ be two positive constants, and:
$$
\Thh:= \left\{ \theta= (\lambda,\mu) : \frac{c}{\|\mu\|^2 \sqrt{n} } \leq \lambda,\  m \leq \|\mu\|    \right\}.  
$$
%%
%Let $0<m<M$, $c>0$ and $\bar\lambda\in(0,1)$, 
%$$
%\Thh := \left\{ \theta= (\lambda,\mu) : \frac{c}{\mu^2 \sqrt{n} } \leq \lambda \leq \bar{\lambda},\  m \leq |\mu| \leq  M   \right\}.    
%$$

Note that this still allows a weak effect of contamination since $\lambda^\star$ can be on the order of $n^{-1/2}$. In this case, we obtain the lower bounds that matches (up to a log term) the upper bounds obtained in Theorems  \ref{thm:vitesse_mu} and \ref{thm:vitesse_lambda}. 

\begin{theo}
\label{th:lb:R1}  Consider two positive constants $m$ and $c$ such that $0<\frac{c}{m^2 \sqrt{n}}<1$ so that $\Thh$ is non empty. 
A density $\phi$ that satisfies $\HS$ and $\HL$ exists such that:
%Consider $0<m<M$ and  $(c,\bar{\lambda})$ such that
%$\bar \lambda > \frac{c}{m^2 \sqrt{n}}>0$ so that $\Theta_n(m,M,c,\bar{\lambda})$ is non empty. 
\begin{itemize}
\item[$(i)$] a positive constant $C_1$ exists such that: 
\begin{equation}
\label{eq:lbmu:R1}
\underset{(\hat\lambda,\hat\mu)}{\inf}\  \underset{(\lambda,\mu)\in\Thh }{\sup}\ \EE[\lambda^2 \|\hat \mu - \mu\|^2] \geq \frac{C_1}{n},
\end{equation}
\item[$(ii)$] a positive constant $C_2$ exists such that: 
\begin{equation}
\label{eq:lblambda:R1}
\underset{(\hat\lambda,\hat\mu)}{\inf}\  \underset{(\lambda,\mu)\in\Thh }{\sup}\ \EE[(\hat \lambda - \lambda)^2] \geq \frac{C_2}{n},
\end{equation}
\end{itemize}
where the infimum is taken over all estimators $\hat \theta=(\hat\lambda,\hat\mu)$ in  \eqref{eq:lbmu:R1} and \eqref{eq:lblambda:R1}. The constants $C_1$ and $C_2$ depend on $c$, $m$ and $\mathcal J$ (defined in $\HL$).
\end{theo}

Even though the proof relies on a Le Cam argument and leads to a $n^{-1}$ rate, it  clearly deserves a careful study for at least two reasons: the loss is asymmetric in $(\lambda,\mu)$ in $i)$ and the balance between $\lambda,\mu$ and $n$ is unclear. We give the proof of this result in Appendix \ref{sub:LWBR1}.

\subsection{Weak contamination model}
 \label{sub:lb2}
We now study the situation when the contamination $\|\mu\|$ is not yet bounded from below and can therefore tend to $0$ as $n \longrightarrow + \infty$. 
Let $c>0$, and:
$$
\Th := \left\{ \theta= (\lambda,\mu) : \frac{c}{\|\mu\|^2 \sqrt{n} } \leq \lambda  \right\} .    
$$

%Let $0<M$, $c>0$ and $\bar\lambda\in(0,1)$:
%$$
%\Th := \left\{ \theta= (\lambda,\mu) : \frac{c}{\mu^2 \sqrt{n} } \leq \lambda \leq \bar{\lambda},\  0< |\mu| \leq  M   \right\}.    
%$$

We introduce a sub-class of densities $\phi$ that satisfy the following assumption:
\paragraph{Assumption $\HD$} \textit{The density $\phi$ satisfies:}
\begin{equation}
\label{IHD}
\cI_\phi := \sup_{1 \leq j \leq d} \int \{d_{j,j}\phi(x)\}^2 \phi^{-1}(x) dx < +\infty,
\end{equation}
where $d_{j,j}$ refers to the second derivative of $\phi$ with respect to the variable $j$.
Note that Assumption $\HD$ is needed for our lower bound results but is not necessary to obtain good estimation properties. However, this assumption is very mild and is again satisfied for many probability distributions as pointed out in Remark \ref{rem:1}. Moreover, from the minimax paradigm, it is enough to obtain our lower bound results with a restricted subset of densities $\phi$.

\begin{theo}
\label{th:lb:R2}
An integer $N>0$ and a function $\phi$ that satisfies $\HS$ and $\HD$ exists such that, for all $n>N$:
\begin{itemize} 
\item[$(i)$]  a positive constant $C_1$ exists such that:
\begin{equation}
\label{eq:lbmu:R2}
\underset{(\hat\lambda,\hat\mu)}{\inf}\  \underset{(\lambda,\mu)\in\Th}{\sup}\ \EE[\|\mu\|^4 (\lambda-\hat\lambda)^2] \geq \frac{C_1}{n},
\end{equation}
\item[$(ii)$] a positive constant $C_2$ exists such that: 
\begin{equation}
\label{eq:lblambda:R2}
\underset{(\hat\lambda,\hat\mu)}{\inf}\  \underset{(\lambda,\mu)\in\Th}{\sup}\ \EE[\lambda^2 \|\mu\|^2 \|\mu-\hat\mu\|^2] \geq \frac{C_2}{n},
\end{equation}
\end{itemize}
where the infimum is taken over all estimators $\hat \theta=(\hat\lambda,\hat\mu)$ in  \eqref{eq:lbmu:R2} and \eqref{eq:lblambda:R2}.
The constant $C_1$ and $C_2$ depend on $c$ and $\mathcal I_\phi$ (defined in $\HD$).
\end{theo}
Finally, we should also remark that estimating $\mu$ when $\lambda$ becomes negligible comparing to $n^{-1/2}$ appears to be impossible as pointed out in $(ii)$ of Theorem \ref{th:lb:R2}.

\section{Discussion}\label{sec:discussion}

\subsection{Related works on distances inequalities and mixture models}

In this paragraph, we provide some additional remarks on the links between several metrics used to describe mixture models in the particular situation of our two-component contamination model.
As pointed out in \cite{ho2016} and \cite{KH15}, relating distances between probability distributions on the observations, and Wasserstein distances (defined in \eqref{defW}) on the space of mixture measures is a popular subject of investigation. Of course, it makes sense when we handle some strong-identifiable models as remarked in the cited previous works.
We will rely the rates for estimating contamination mixtures to rates for general mixtures. The latter are usually stated in terms of transportation distance between the mixing distributions $G$. For a contamination mixture, it reads:

\begin{align}
    \label{eq:Gcontam}
    G_{\lambda, \mu} & = (1 - \lambda) \delta_0 + \lambda \delta_{\mu},
\end{align}

where $\delta_{\theta}$ is the Dirac peak at $\theta $. 

%More generally, for a mixture $f = \sum_i \lambda_i \phi_{\theta_i} $ where the components are parametrized by $\theta \in \Theta$, the mixing distribution $G$ is $ \sum_i  \lambda_i \delta_{\theta_i} $.
%The Wasserstein ($L^p$)-transportation distances between two probability measures $m_1$ and $m_2$ on $\Omega $ are defined by
%\begin{align*}
%    W_p^p(m_1, m_2) & := \inf_{\pi \in \Pi(m_1,m_2)} \int d^p(x,y) \mathrm{d}\pi(x,y),
%\end{align*}
%where $\Pi(m_1,m_2)$ is the set of probability measures on $\Omega \times \Omega $ such that their marginals are $m_1$  and $m_2$. 

In \cite{ho2016}, it is shown that the Total Variation distance denoted $V(f_{\lambda,\mu},f_{\lambda^\star,\mu^\star})$ between the probability distributions dominates the Wasserstein distance $W_1(G_{\lambda, \mu}, G_{\lambda^\star,\mu^\star})$ when the number of components is known. When it is unknown, but we are only interested in the distance of the estimator to the true distribution, the rate deteriorates to $V(f_{\lambda,\mu},f_{\lambda^\star,\mu^\star}) \gtrsim W_2^2(G_{\lambda, \mu}, G_{\lambda^\star,\mu^\star})$, under appropriate identifiability conditions. 

When we are interested in local minimax rates of convergences, the situation worsens, as proved in \cite{KH15}. It is shown that the supremum norm between the probability distributions $\|.\|_{\infty}$ dominates the Wasserstein distance $W_{2m-1}^{2m-1}$ where essentially $2m-1$ is the number of unknown positions to be estimated in the mixture model (the $m$ possible locations and the $m-1$ dimensional weights distribution):
$$
\left\lVert f_{\lambda,\mu} - f_{\lambda^\star,\mu^\star}  \right\rVert_{\infty}  \gtrsim W_{2m-1}^{2m-1}(G_{\lambda, \mu}, G_{\lambda^\star,\mu^\star}). 
$$
The Dvoretzky-Kiefer-Wolfowitz inequality then allows \cite{KH15}  to deduce a $n^{-1/(4m-2)}$ rate of convergence on the parameters. 

Notice that for two components, the above speed is in $n^{-1/6}$, whereas our
speeds here are in $n^{-1/4}$. This is because the bound by  \cite{KH15} is for
generic mixture models, while in this work, we deal with a specific
two-component contaminated model. Specifically,  in typical cases, the minimax
speed for estimating the parameters of mixture models is $n^{-1/2d}$ where $d$
is the number of parameters. The generic two-component model has three
parameters, whereas our contamination model has only two.

\subsection{Comparing $W_2$ and $\|.\|_2$ in a two-component contamination model}

In this work, we have chosen to handle the $\mathbb L^2$ distance on probability distributions, instead of  $V$ or $\|.\|_{\infty}$, nevertheless a relationship between $\|.\|_2$ and $W_p$ should exist. The next result essentially states this dependency.

\begin{theo}\label{theo:L2W2}
For any density $\phi$ that satisfies $\HS$ and $\HL$, a constant $c_\phi>0$ exists such that:
$$\forall (\lambda,\lambda') \in (0,1)^2 \quad \forall (\mu,\mu')\in [-M,M]^d \qquad 
\|f_{\lambda,\mu}-f_{\lambda',\mu'}\|_2  \geq c_\phi W_2^{2}(G_{\lambda,\mu},G_{\lambda',\mu'}).
$$
%where $G_{\lambda,\mu}=(1-\lambda) \delta_{0} + \lambda \delta_{\mu}.$
\end{theo}
%Again, following the arguments of \cite{KH15}, a such inequality is not surprising  in our 2-components contamination model, $\|.\|_2 \geq W^2$ because the number of unknown parameters in our contamination model is $2$ (and then replacing the power $2m-1$ discussed above).
Hence, $\hat{f}_n := f_{\hat{\lambda}_n,\hat{\mu}_n}$ defined by \eqref{eq:estimator} satisfies 
$$
\mathbb{E}_{\lambda^\star,\mu^\star} \left[  W_2^4(G_{\hat\lambda_n,\hat\mu_n},G_{\lambda^\star,\mu^\star}) \right] \lesssim \mathbb{E} \left[\|\hat f_n - f_{\lambda^{\star},\mu^{\star}}\|_2^2\right] \lesssim \frac{(\log n)^2}{n}.
$$
In other words, the $\mathbb{L}^2$ strategy investigated in this paper allows in fact to control the Wasserstein distance between the estimated mixture distribution $G_{\hat\lambda_n,\hat\mu_n}$ and the target $G_{\lambda^\star,\mu^\star}$.
On the other hand, a lower bound on the minimax rate of convergence in term of the Wasserstein distance may not be directly deduced from our results displayed in Theorems \ref{th:lb:R1} or \ref{th:lb:R2} because of the
lack of symmetry in $(\lambda, \mu)$ with respect to $(\hat{\lambda},
\hat{\mu})$.

\section{Simulation study}\label{sec:example_distribution}

\subsection*{Distributions}
In this section, we assess the performance of the $\LL^2$-estimator given in \eqref{eq:estimator} on four particular cases  ($d=1$) of baseline density $\phi$. We study the following features:
\begin{itemize}
\item Standard Gaussian case with $\phi(x)=\frac{1}{\sqrt{2\pi}} e^{-\frac 1 2 x^2}.$
\item Non-smooth distribution with the Laplace density $\phi(x)=\frac 1 2 e^{-|x|}.$
\item Heavy tailed distribution with the Cauchy density:
$
\phi(x)=\frac{1}{\pi (1+x^2)}.$
\item Asymmetry with the skew Gaussian density:
$ \phi(x)= 2 \psi(x) \Psi(\alpha x),$
where $\psi$ and $\Psi$, respectively, denote the density and the cumulative function of the standard Gaussian distribution and where $\alpha$ is the asymmetry  parameter different from $0$ (in the simulations, 
we fix $\alpha=10$). This example of asymmetric distributions has been introduced by \cite{azzalini85}.
\end{itemize}

Our estimator requires the calculation of the contrast $\gamma_n$ and, in particular, the value of the $\LL^2$ norm: 
$$\|f_{\lambda,\mu}\|_2^2 = \left[\lambda^2+(1-\lambda)^2\right] \|\phi\|_2^2+2 \lambda(1-\lambda)\langle \phi,\phi_{\mu}\rangle,$$ that involves the value of inner product $\langle\phi, \phi_\mu\rangle$ for any value of the location parameter $\mu \in [-M,M]$. In the first three examples of distributions, a closed formula exists:
 \begin{itemize}
\item Gaussian density: $\langle\phi, \phi_\mu\rangle = (4\pi)^{-\frac 1 2} \exp\left[-\frac 1 4 \mu^2\right]$
\item Laplace density: $\langle\phi, \phi_\mu\rangle = \frac 1 4 e^{-|\mu|} (1+|\mu|)$
\item Cauchy density:  $\langle\phi, \phi_\mu\rangle = \frac{2}{\pi(4+\mu^2)}$
\end{itemize}
Unfortunately, such a formula is not available (to our knowledge) for the skew Gaussian density: there is no analytical expression of $\langle\phi, \phi_\mu\rangle$.  Instead, we used a Monte-Carlo procedure to evaluate this quantity for each value of $\mu$ in our grid $\mathcal{M}_{n}$ given in \eqref{eq:grille}. To obtain a sufficient approximation of these inner products, we used a number of Monte-Carlo iterations $T_{MC}$ each time of the order $T_{MC} \propto n^2$ (where $n$ will be the sample size used for our estimation problem).

\subsection*{Statistical setting} We have worked in 1-D with a fixed value of $\lambda^\star = \frac 1 4$ while $\mu^\star$ is allowed to vary with $n$. Below, we used the following relationship between $\mu^\star$ and $n$:
$$\mu^{\star} = \sqrt{\frac{1}{\lambda^\star  n^\nu}} \qquad \text{with} \quad \nu= \frac{\alpha}{24},\ \alpha\in\left\{ 1,\dots, 24\right\}.$$ 
For each value of the parameter $\mu^{\star}$, we used $10^3$ Monte-Carlo simulations to obtain reliable results, while the grid size is determined by fixing the maximal value of the unknown $|\mu^{\star}|$ as $M=10$. Finally, we sampled a set of $n=5000$ observations each time.

\begin{figure}[h!]
\begin{minipage}[b]{0.4\linewidth}
\centerline{\includegraphics[width=7cm]{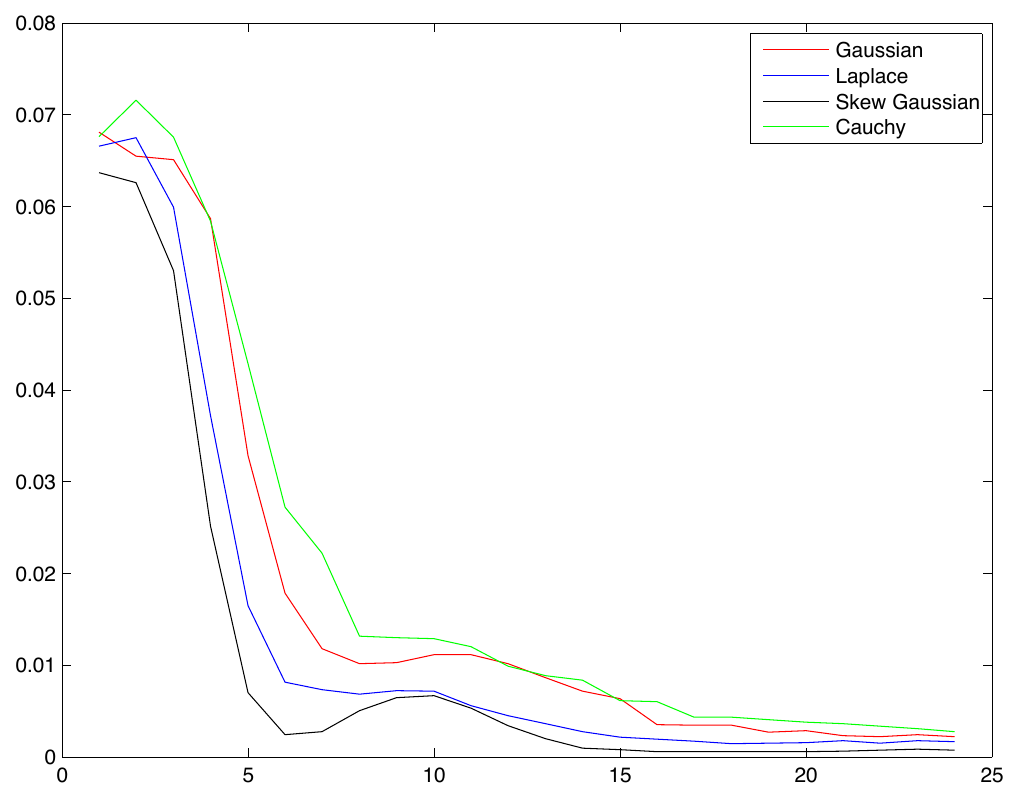}}
\end{minipage}\hfill
\begin{minipage}[b]{0.4\linewidth}
\centerline{\includegraphics[width=7cm]{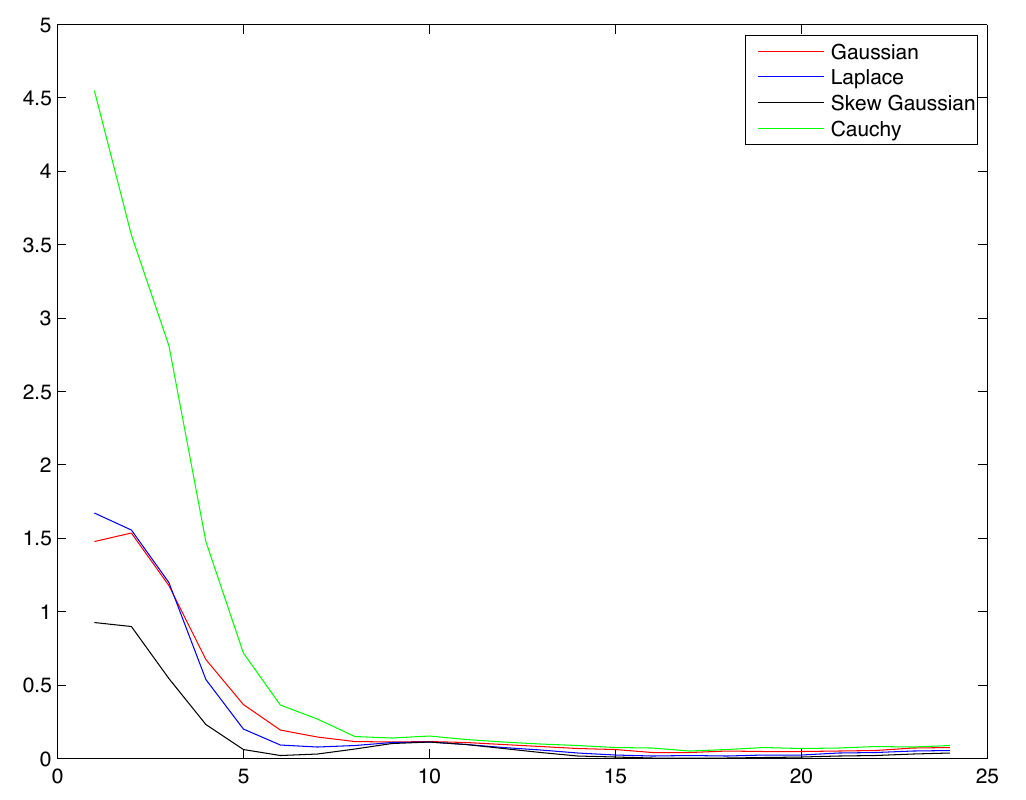}}
\end{minipage}
%\caption{Mean square error of estimating $\lambda^\star$ (left) and $\mu^\star$ (right) when $1/\nu \in [1,12]$.\label{fig:mse}}
\caption{Mean square error of estimating $\lambda^\star$ (left) and $\mu^\star$ (right) for the 24 values of $\nu$ in descending order.\label{fig:mse}}
\end{figure}

In Fig. \ref{fig:mse}, for each case of the mixture model, we represent the evolution of the mean square error for the estimation of $\lambda^\star$ and of $\mu^\star$ when $\nu$ varies between $1/24$ and $1$:
$$\nu \longmapsto \mbox{MSE}(\lambda) = \frac{1}{10^3}\underset{j=1}{\stackrel{10^3}{\sum}}  (\hat \lambda_j - \lambda^\star)^2$$ 
and  
$$\nu \longmapsto \mbox{MSE}(\mu) = \frac{1}{10^3}\underset{j=1}{\stackrel{10^3}{\sum}}  (\hat \mu_j - \mu^\star)^2.$$
As pointed out in Fig. \ref{fig:mse}, the estimation of $\lambda^\star$ and $\mu^\star$ performs quite well as soon as $\nu$ is lower than $1/2$ but becomes completely inconsistent when $\nu>1/2$, even if we use a sample size of $5000$ observations.

We also represent the violin plot of these estimations indicating the same behavior in each particular case (Gaussian and Laplace in Fig. \ref{fig:gauss-lap}; Cauchy and skew Gaussian in Fig. \ref{fig:cauchy-skew}).

\begin{figure}[h!]
\begin{minipage}[c]{0.4\linewidth}
\centerline{\includegraphics[width=7cm]{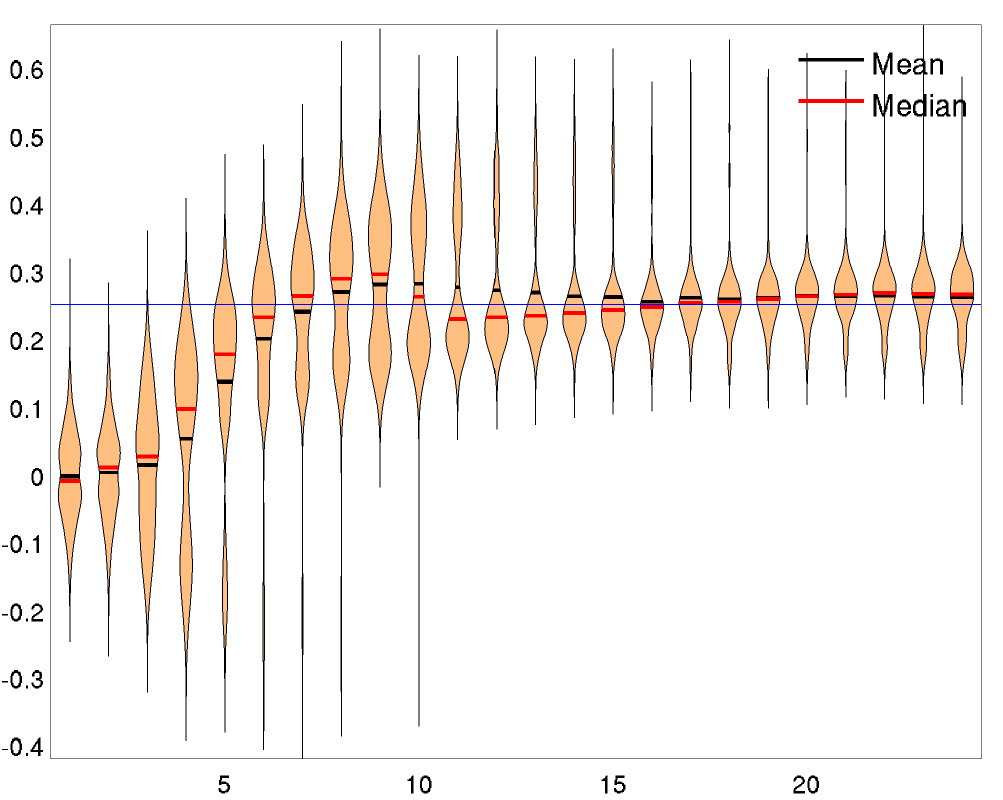}}
\end{minipage}\hfill
\begin{minipage}[c]{0.4\linewidth}
\centerline{\includegraphics[width=7cm]{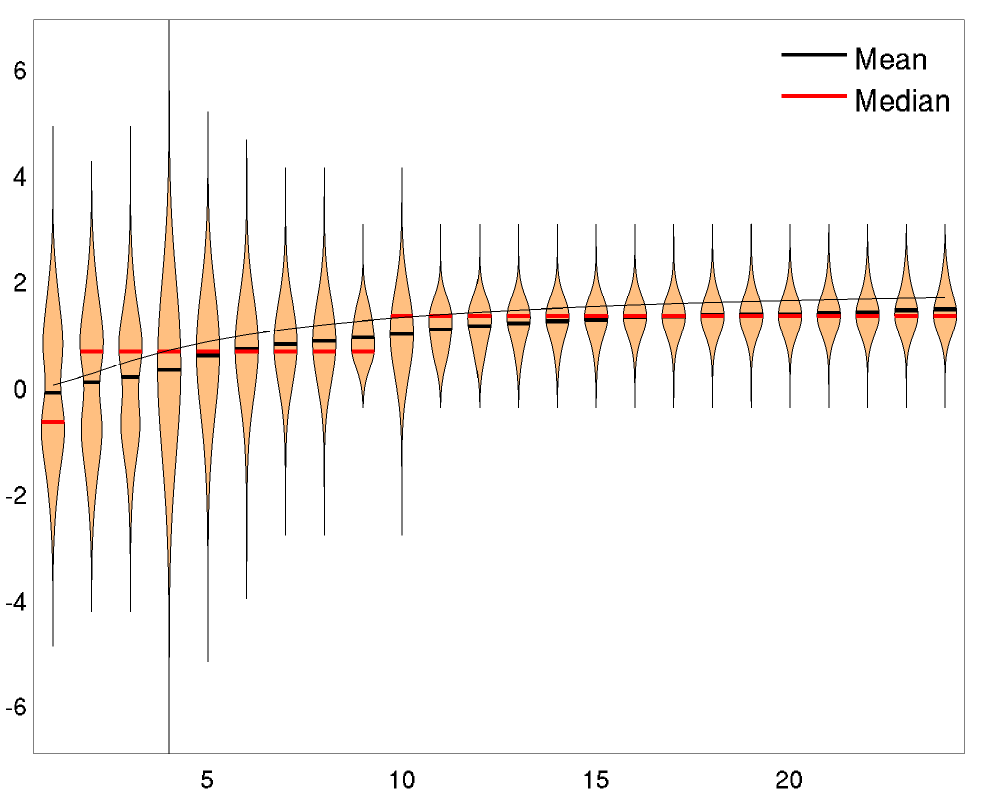}}
\end{minipage}
\begin{minipage}[c]{0.4\linewidth}
\centerline{\includegraphics[width=7cm]{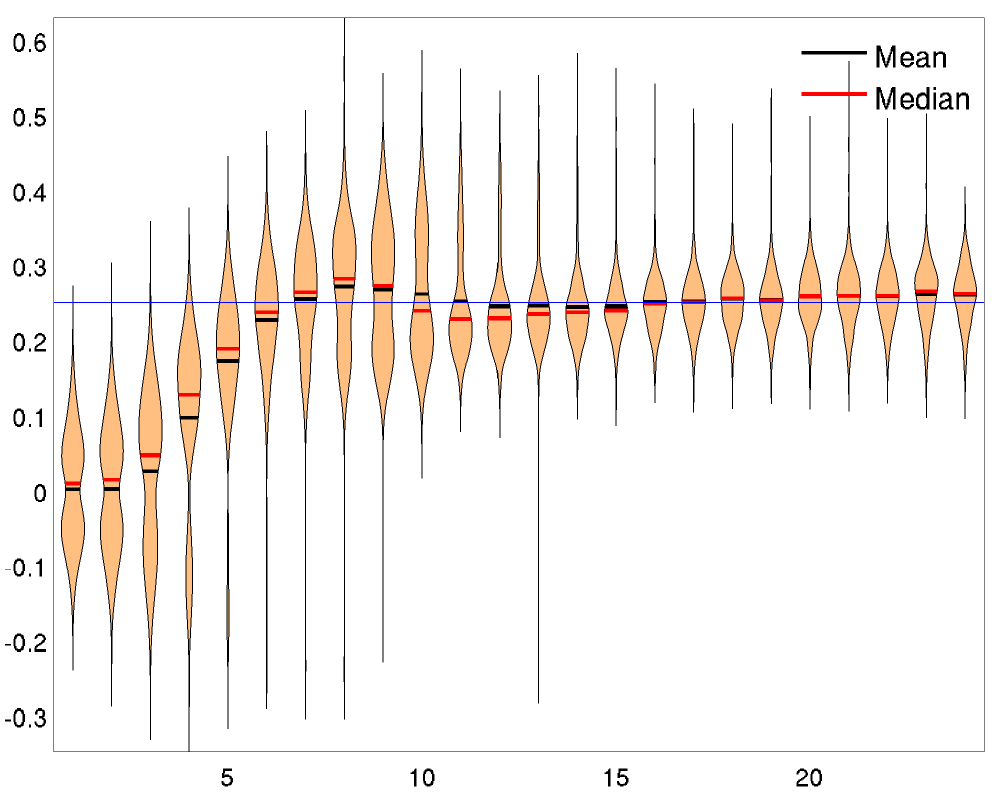}}
\end{minipage}\hfill
\begin{minipage}[c]{0.4\linewidth}
\centerline{\includegraphics[width=7cm]{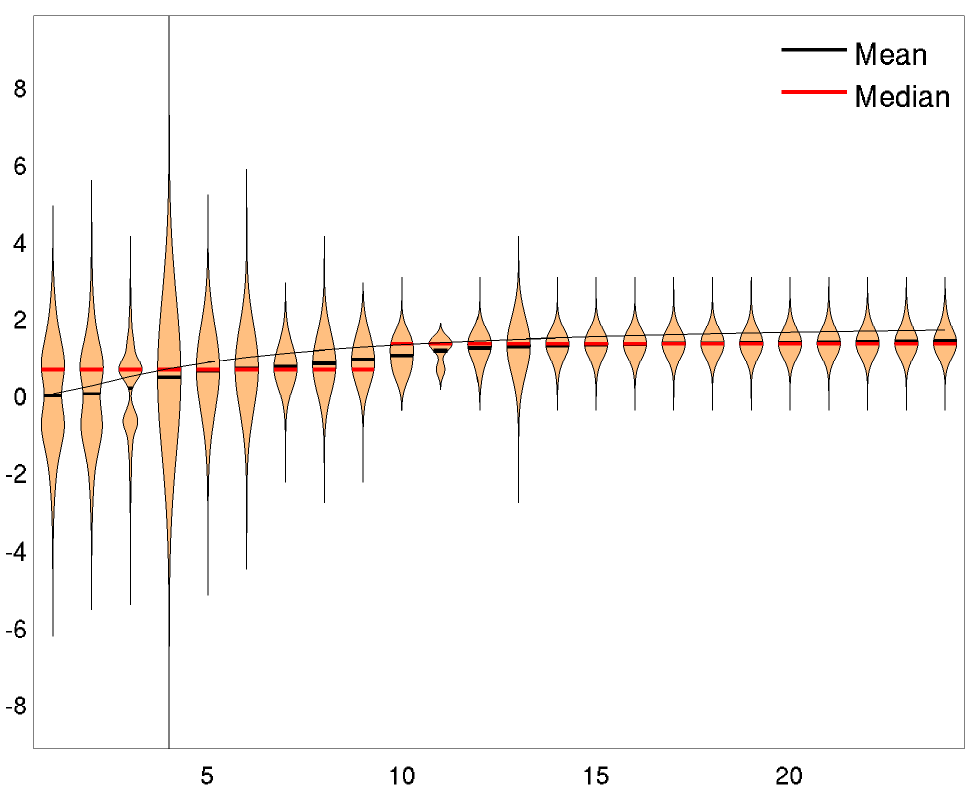}}
\end{minipage}
\caption{Evaluation of $\lambda^\star$ (on the left) and $\mu^\star$ (on the right) for our estimators when \textbf{Gaussian} mixtures (top) and \textbf{Laplace} mixtures (bottom) are considered, for the 24 values of $\nu$ in descending order.}\label{fig:gauss-lap}
\end{figure}

\begin{figure}[h!]
\begin{minipage}[b]{0.4\linewidth}
\centerline{\includegraphics[width=7cm]{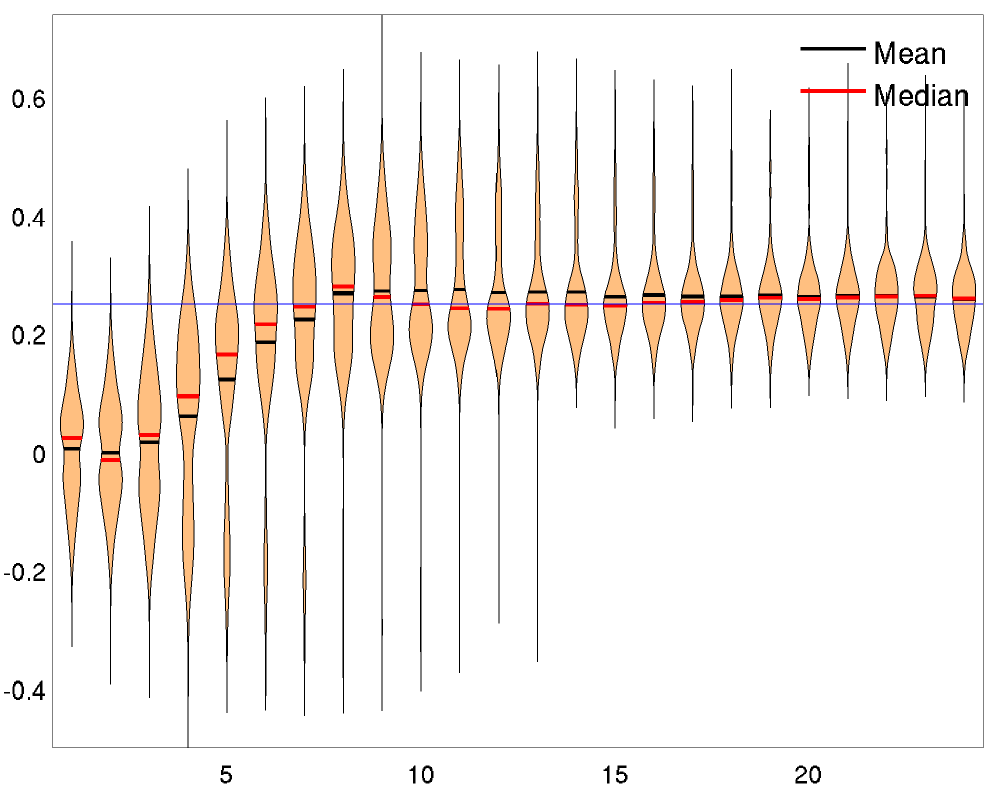}}
\end{minipage}\hfill
\begin{minipage}[b]{0.4\linewidth}
\centerline{\includegraphics[width=7cm]{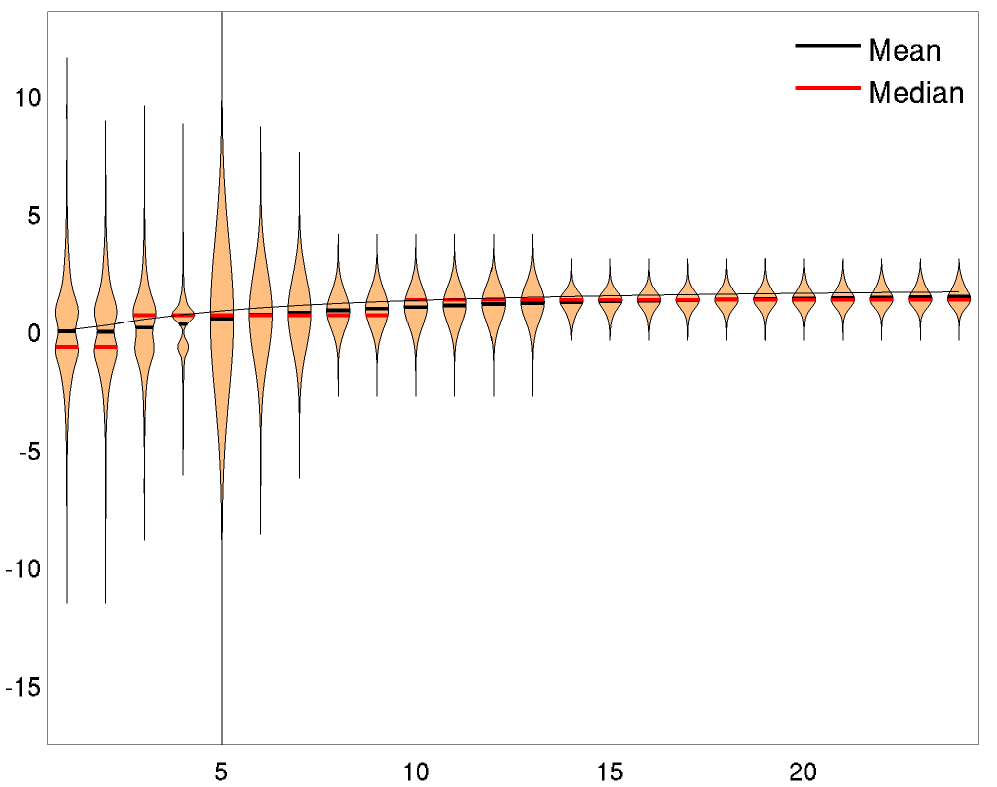}}
\end{minipage}
\begin{minipage}[b]{0.4\linewidth}
\centerline{\includegraphics[width=7cm]{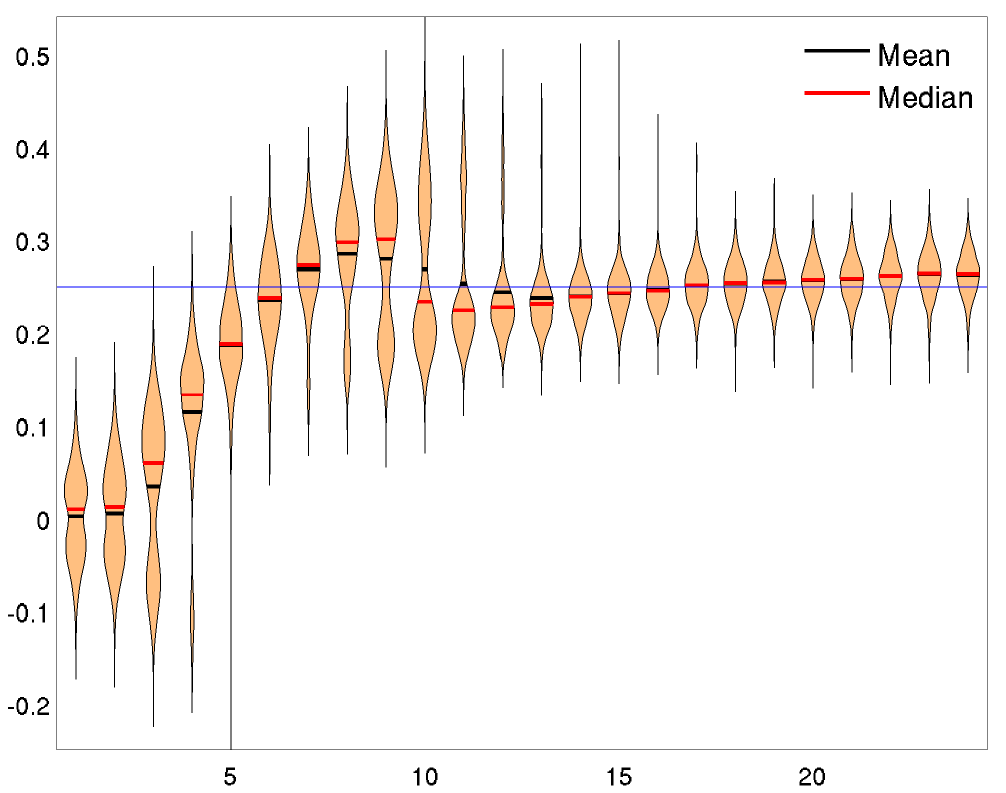}}
\end{minipage}\hfill
\begin{minipage}[b]{0.4\linewidth}
\centerline{\includegraphics[width=7cm]{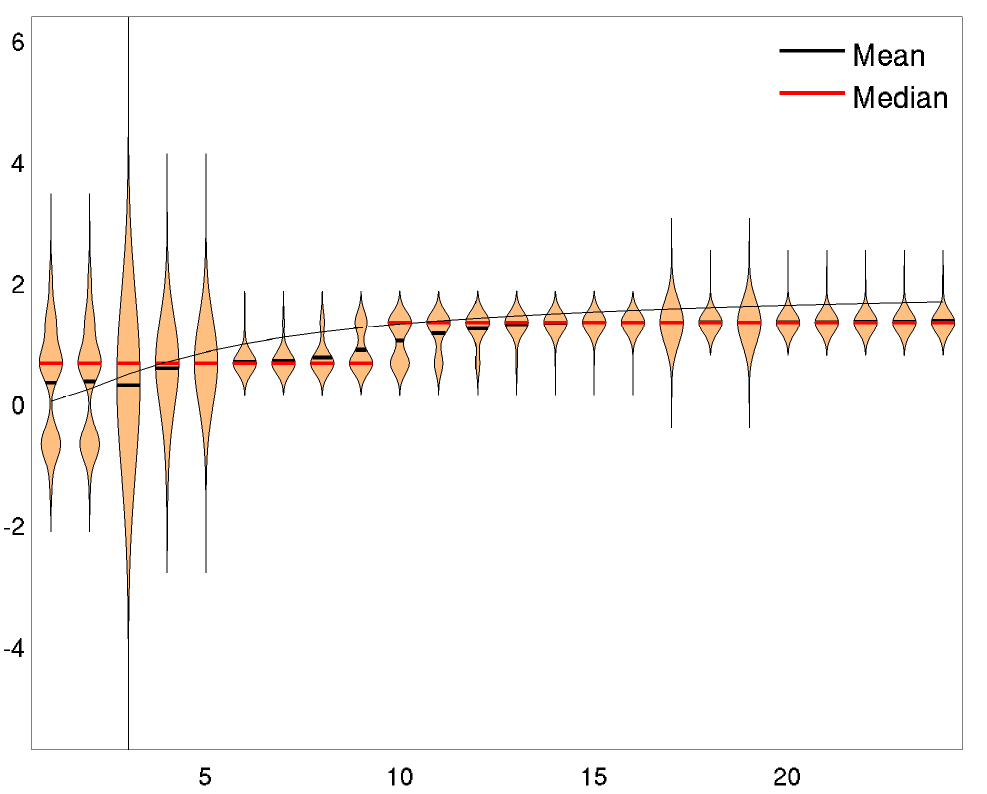}}
\end{minipage}
\caption{Evaluation of $\lambda^\star$ (on the left) and $\mu^\star$ (on the right) for our estimators when \textbf{Cauchy} mixtures (top) and \textbf{skew Gaussian} mixtures (bottom) are considered, for the 24 values of $\nu$ in descending order.}\label{fig:cauchy-skew}
\end{figure}

Again, a similar conclusion holds: the estimators derived from \eqref{eq:estimator} exhibit a low bias and variance when $\nu$ is chosen small enough (lower than $1/2$, which corresponds to values greater than 12 in the horizontal axes of Figs. \ref{fig:gauss-lap}-\ref{fig:cauchy-skew}). In contrast, the estimation is seriously damaged for values of $\nu$ greater than $1/2$ (which corresponds to values lower than 11 in the horizontal axes of Figs. \ref{fig:gauss-lap}-\ref{fig:cauchy-skew}). Finally, it should be noted that the shape of the density $\phi$ does not seem to have a big influence on the estimation ability, even though the Cauchy distribution settings may be seen as the most difficult problem (as represented by the green MSE in Fig. \ref{fig:mse}).

\section{Proofs of the upper bounds}
 \label{s:proof:ub}

\subsection{Preliminary oracle inequality}\label{s:oracle} 
We first establish a technical proposition that will be used to derive the proof of Theorem \ref{th:oracle}. For a given grid $\MM$, we first introduce the theoretical minimizer of the $\LL^2$-norm on this grid:
\begin{equation}\label{eq:approximation}
(\lambda_0,\mu_0) = \arg \min_{(\lambda,\mu)\in \MM} \| f_{\lambda,\mu} - f^\star \|_2^2.
\end{equation}
We then define $\mathcal{E}_n( \lambda, \mu)$ the empirical process indexed by $(\lambda,\mu)\in\MM$  as:
$$
\mathcal{E}_n ( \lambda, \mu) 
 = \frac{2}{n} \sum_{i=1}^n \left\lbrace   f_{\lambda,\mu}(X_i) - f_{\lambda_0,\mu_0}(X_i) - [ \langle f_{\lambda,\mu} - f_{\lambda_0,\mu_0},f^\star \rangle  ] \right\rbrace.
$$
For all $(\lambda,\mu) \in \MM$, the term $\mathcal{E}_n(\lambda,\mu)$ can be rewritten as:
\begin{equation}
\mathcal{E}_n(\lambda,\mu) = \frac{1}{n} \sum_{i=1}^n ( Y_i - \mathbb{E}[Y_i]) \quad \mathrm{where} \quad Y_i := 2 [ f_{\lambda,\mu}(X_i) -f_{\lambda_0,\mu_0}(X_i) ].
\label{eq:def_E}
\end{equation}
In particular, $\EE[\mathcal{E}_n(\lambda,\mu) ]=0$ and:
\begin{eqnarray*}
\mathrm{Var}(Y_i) \leq \mathbb{E}[Y_i^2]
& = & 4 \mathbb{E}[ (f_{\lambda,\mu}(X_i) - f_{\lambda_0,\mu_0}(X_i) )^2] , \\
& = & 4 \int_\mathbb{R} [ f_{\lambda,\mu}(x) - f_{\lambda_0,\mu_0}(x) ]^2 f^{\star}(x)dx , \\
& \leq & 4 \| \phi \|_\infty \| f_{\lambda,\mu} - f_{\lambda_0,\mu_0} \|_2^2,
\end{eqnarray*}
since $\| f^\star \|_\infty \leq \| \phi \|_\infty$.  We will use a normalized version of this process below, which naturally leads 
to the introduction of $\mathcal{G}_n(\lambda,\mu)$:
$$
\forall (\lambda,\mu) \in \MM \setminus \{(\lambda_0,\mu_0)\} \qquad 
\mathcal{G}_n ( \lambda, \mu)=\frac{\mathcal{E}_n ( \lambda, \mu)}{\| f_{\lambda,\mu} - f_{\lambda_0,\mu_0}\|_2}.
$$

Our estimator   $(\hat\lambda_n, \hat\mu_n)$ defined in (\ref{eq:estimator}) satisfies  the following useful property.
\begin{lemma}\label{prop:controlG}~\\
\begin{itemize}
\item[$(i)$]
For any $(\lambda,\mu)$ such that $\|f_{\lambda,\mu}-f_{\lambda_0,\mu_0}\|_2 \geq n^{-1/2}$:
\begin{equation}
\label{controlG}
\forall s >0 \qquad 
\PP \left( | \mathcal{G}_n(\lambda,\mu)  | > s \right) 
\leq \exp \left( - \frac{n s^2 }{8 \|\phi\|_\infty \left[ 1 +   \frac{s\sqrt{n}}{3}\right]}  \right).
\end{equation}
\item[$(ii)$]
We can find $C>0$ such that:
\begin{equation}
\label{controlEG2}
\EE\left[\cG^2_n(\hat \lambda_n,\hat \mu_n) \mathds{1}_{\cB^c}\right] \leq \frac{C \log^2(|\MM|)}{n},
\end{equation}
where $\mathcal{B}$ is the event defined as 
$ 	\mathcal{B} = \left\lbrace \| \hat f_n - f_{\lambda_0,\mu_0} \|_2 \leq \frac{1 }{\sqrt{n}}  \right\rbrace.
$ 
\end{itemize}
\end{lemma}

\begin{proof} In this proof, $C$ refers to a constant that is independent of $n$, whose value may change from line to line.\\

\underline{Proof of $(i)$:}
thanks to the Bennett inequality, we obtain for all $s>0$:
\begin{eqnarray*}
\lefteqn{\PP \left( | \mathcal{G}_n(\lambda,\mu)  | > s \right) }\\
& \leq & \exp \left( - \frac{n^2 s^2 \| f_{\lambda,\mu}-f_{\lambda_0,\mu_0} \|_2^2}{8n \|\phi\|_\infty \| f_{\lambda,\mu}-f_{\lambda_0,\mu_0} \|_2^2 + 8n\| \phi\|_\infty s\| f_{\lambda,\mu}-f_{\lambda_0,\mu_0} \|_2/3}  \right), \\
& = & \exp \left( - \frac{n s^2 }{8 \|\phi\|_\infty \left[ 1 +   s\| f_{\lambda,\mu}-f_{\lambda_0,\mu_0} \|_2^{-1}/3\right]}  \right).
\end{eqnarray*}
Using the fact that $\|f_{\lambda,\mu}-f_{\lambda_0,\mu_0}\|_2 \geq n^{-1/2}$, we obtain:
$$
\PP \left( | \mathcal{G}_n(\lambda,\mu)  | > s \right)  \leq  \exp \left( - \frac{n s^2 }{8 \|\phi\|_\infty \left[ 1 +   \frac{s\sqrt{n}}{3}\right]}  \right),
$$
which is the desired Inequality \eqref{controlG}.

\underline{Proof of $(ii)$:}
 observe that for all $t>0$,
\begin{eqnarray}
\mathbb{E}\left[  \mathcal{G}_n^2(\hat\lambda_n,\hat\mu_n) \mathds{1}_{\mathcal{B}^c}\right]
& \leq & t^2 + \mathbb{E} \left[ \mathcal{G}_n^2(\hat\lambda_n,\hat\mu_n)\mathds{1}_{\lbrace | \mathcal{G}_n(\hat\lambda_n,\hat\mu_n)  | > t \rbrace}  \mathds{1}_{\mathcal{B}^c} \right] \nonumber,\\
& \leq & t^2 + \mathbb{E} \left[\sup_{(\lambda,\mu): \|f_{\lambda,\mu}-f_{\lambda_0,\mu_0}\| \geq n^{-1/2}} \left\lbrace \mathcal{G}_n^2(\lambda,\mu) \mathds{1}_{\lbrace | \mathcal{G}_n(\lambda,\mu)  | > t \rbrace} \right\rbrace\right],\nonumber \\
& \leq & t^2+ \sum_{(\lambda,\mu) : \|f_{\lambda,\mu}-f_{\lambda_0,\mu_0}\| \geq n^{-1/2} } \mathbb{E} \left[ \mathcal{G}_n^2(\lambda,\mu) \mathds{1}_{\lbrace | \mathcal{G}_n(\lambda,\mu)  | > t \rbrace} \right] \label{eq:tecG}. 
\end{eqnarray}

Integrating by parts, we can remark that:
$$ \mathbb{E} \left[ \mathcal{G}_n^2(\lambda,\mu) \mathds{1}_{\lbrace | \mathcal{G}_n(\lambda,\mu)  | > t \rbrace} \right]  = t^2  \ \PP( | \mathcal{G}_n(\lambda,\mu)  | > t ) + \int_{t^2}^{+\infty} \PP( | \mathcal{G}_n(\lambda,\mu)  | > \sqrt{x} )dx.$$
Thus, if we choose $t = \left( \frac{16 \| \phi\|_\infty \log (| \MM|)}{3} \vee 3\right) n^{-1/2}$, then $t \sqrt{n}/3 \geq 1$, so that  for any $s \geq t$ and for a fixed  $(\lambda,\mu)$, \eqref{controlG} yields:
\begin{eqnarray*}
\lefteqn{
\mathbb{E} \left[ \mathcal{G}_n^2(\lambda,\mu) \mathds{1}_{\lbrace | \mathcal{G}_n(\lambda,\mu)  | > t \rbrace} \right]
}\\
&\leq & t^2 \exp \left( - \log (| \MM|) \right) + \int_{t^2}^{+\infty}
\exp \left( - \frac{3\sqrt{nx }}{16\|\phi\|_\infty }\right) dx \\
& \leq & C \frac{ \log^2 (| \MM|)}{n  } \times \frac{1}{| \MM|}+ 2 \int_{t}^{+\infty} u \exp \left( -\frac{3\sqrt{n }u }{16\|\phi\|_\infty }\right) du,
\end{eqnarray*}
for large enough $C$, 
where the last line comes from the size of $t$ for the left-hand side, and from the change of variable $u = \sqrt{x}$ in the integral. The remaining integral  may be integrated by parts, which in turn leads to: 
$$
\mathbb{E} \left[ \mathcal{G}_n^2(\lambda,\mu) \mathds{1}_{\lbrace | \mathcal{G}_n(\lambda,\mu)  | > t \rbrace} \right] \leq  C \frac{\log^2 (| \MM|)}{n } \times \frac{1}{| \MM|}.$$

If we plug the above upper bound into \eqref{eq:tecG}, we then obtain that a sufficiently large constant $C$ exists such that: 
\begin{equation*}
\mathbb{E} \left[ \mathcal{G}_n^2(\hat\lambda_n,\hat\mu_n) \mathds{1}_{\mathcal{B}^c} \right]
\leq C \frac{ \log^2(| \MM|)}{n} \times \frac{| \MM| }{| \MM|} | =  C \frac{ \log^2(| \MM|)}{n}.
\end{equation*}
\end{proof}
%\begin{eqnarray*}
%\lefteqn{\mathbb{E} \left[ \mathcal{G}_n^2(\hat\lambda,\hat\mu) \mathds{1}_{\mathcal{B}^c} \right]} \\
%& \leq& C \frac{ \log^2(| \MM|)}{n} \times \frac{1}{| \MM|} \times | \MM| =  C \frac{ \log^2(| \MM|)}{n}.
%\end{eqnarray*}

We are now interested in the proof of the oracle inequality.
\begin{proof}[Proof of Theorem \ref{th:oracle}]
The best approximation term $(\lambda_0,\mu_0)$ over the grid $\MM$ is defined in \eqref{eq:approximation} and the event $\mathcal{B} = \left\lbrace \| \hat f_n - f_{\lambda_0,\mu_0} \|_2 \leq \sqrt{\frac{1 }{n}}  \right\rbrace$ is introduced in Proposition \ref{prop:controlG}.
On the event $\mathcal{B}$, the situation is easy using the Young inequality $2 ab \leq \alpha a^2+ \alpha^{-1} b^2$ so that for all $\alpha>0$,
\begin{eqnarray}
\EE \left[\| \hat f_n - f^\star \|_2^2 \mathds{1}_\mathcal{B}\right] 
& \leq & (1+ \alpha) \| f_{\lambda_0,\mu_0} - f^\star \|_2^2 + (1+\alpha^{-1})\EE \left[\| \hat f_n - f_{\lambda_0,\mu_0} \|_2^2 \mathds{1}_\mathcal{B}\right],  \nonumber\\
& \leq & (1+\alpha) \| f_{\lambda_0,\mu_0} - f^\star \|_2^2 + \frac{1+\alpha^{-1}}{n}.
\label{eq:m1}
\end{eqnarray}

We provide below a similar control on the event $\mathcal{B}^c$. First, observe that according to the definition of $(\hat\lambda_n, \hat\mu_n)$, for all $(\lambda,\mu) \in \MM$, we have:
\begin{eqnarray*}
& & \gamma_n(\hat \lambda_n, \hat \mu_n) + \| f^\star \|_2^2  \leq  \gamma_n(\lambda,\mu) + \| f^\star \|_2^2, \\
& \Leftrightarrow & \| \hat f_n - f^\star \|_2^2 \leq \| f_{\lambda,\mu} - f^\star \|_2^2 + 2 \left[  \frac{1}{n} \sum_{i=1}^n \hat f_n(X_i) - \langle \hat f_n,f^\star \rangle \right] \\ 
& &  \hspace{3cm}- 2 \left[ \frac{1}{n} \sum_{i=1}^n f_{\lambda,\mu}(X_i) - \langle f_{\lambda,\mu} , f^\star \rangle    \right].
\end{eqnarray*}
This inequality being true for $(\lambda,\mu)= (\lambda_0,\mu_0)$, we obtain:
$$
\| \hat f_n - f^\star \|_2^2\mathds{1}_{\mathcal{B}^c}   \leq \| f_{\lambda_0,\mu_0} - f^\star \|_2^2 + \mathcal{E}_n(\hat \lambda_n,\hat\mu_n) \mathds{1}_{\mathcal{B}^c}.
$$

This implies that for all $0<\alpha<1$:
\begin{eqnarray*}
&  & \| \hat f_n - f^\star \|_2^2 \mathds{1}_{\mathcal{B}^c} \leq \| f_{\lambda_0,\mu_0} - f^\star \|_2^2 + \| \hat f_n - f_{\lambda_0,\mu_0} \|_2 \frac{\mathcal{E}_n(\hat \lambda_n,\hat\mu_n)}{\| \hat f_n - f_{\lambda_0,\mu_0} \|_2}\mathds{1}_{\mathcal{B}^c}, \nonumber \\
& \Rightarrow & \| \hat f_n - f^\star \|_2^2 \mathds{1}_{\mathcal{B}^c} \leq \| f_{\lambda_0,\mu_0} - f^\star \|_2^2 + \frac{\alpha}{2} \| \hat f_n - f_{\lambda_0,\mu_0} \|_2^2 \mathds{1}_{\mathcal{B}^c} + \frac{1}{2 \alpha} \cG_n^2(\hat\lambda_n,\hat\mu_n) \mathds{1}_{\mathcal{B}^c}. \nonumber \\
\end{eqnarray*}
Using $\|u+v\|^2 \leq 2 \|u\|^2+2\|v\|^2$, we then deduce that: 
\begin{equation}
 \| \hat f_n - f^\star \|_2^2 \mathds{1}_{\mathcal{B}^c} \leq \frac{(1+\alpha ) }{(1-\alpha) } \| f_{\lambda_0,\mu_0} - f^\star \|_2^2  +  \frac{1}{2\alpha} \cG_n^2(\hat\lambda_n,\hat\mu_n)\mathds{1}_{\mathcal{B}^c}.
\label{eq:inter1}
\end{equation}
We can conclude the proof taking \eqref{controlEG2} in \eqref{eq:inter1}, and \eqref{eq:m1} together.
\end{proof}

\subsection{Proof of Theorem \ref{thm:vitesse_mu}}
\label{s:preuvemu}
We aim to apply the oracle  inequality established in Theorem \ref{th:oracle}. First, we need an upper bound on the approximation term given by $ \| f_{\lambda_0,\mu_0} - f^\star \|_2^2 $ when $(\lambda_0,\mu_0)$ belongs to our grid $\mathcal{M}_n$.
We can observe that for all $(\lambda,\mu) \in (0,1) \times \mathbb{R}^d$,
\begin{eqnarray} 
\| f_{\lambda,\mu} - f^\star \|_2^2 
& = & \| (1-\lambda) \phi + \lambda \phi_{\mu} - (1-\lambda^\star)\phi - \lambda^\star \phi_{\mu^\star} \|_2^2\nonumber \\
& = & \| (\lambda^\star - \lambda) \lbrace \phi - \phi_{\mu} \rbrace + \lambda^\star \lbrace \phi_{\mu} - \phi_{\mu^\star}  \rbrace \|_2^2\label{eq:interm}  \\
& \leq & 2(\lambda^\star - \lambda)^2 \| \phi - \phi_{\mu} \|_2^2 + 2 \{\lambda^\star\}^2 \| \phi_{\mu} - \phi_{\mu^\star} \|_2^2.\nonumber
\end{eqnarray}
Using Proposition \ref{prop:cnorme}, we can find two positive constants $\overline{\kappa}$ and $\underline{\kappa}$ such that:
\begin{equation}
\label{linkphimu}
\forall (\mu,\tilde{\mu})\in\RR^d \times\RR^d   \qquad 
\underline{\kappa} \|\mu-\tilde \mu\|^2 \leq \| \phi_{\mu} - \phi_{\tilde \mu} \|_2^2 \leq \overline{\kappa} \|\mu - \tilde\mu\|^2,
\end{equation}
which in turn implies that:
$$
\| f_{\lambda,\mu} - f^\star \|_2^2 
\leq  8 \|\phi\|_2^2 (\lambda^\star - \lambda)^2 + 2\overline{\kappa} \,  \{\lambda^\star\}^2 \|\mu-\mu^\star\|^2.
$$
In particular, the definition of $\mathcal{M}_n$ given in (\ref{eq:grille}) makes it possible to find a constant $C>0$ such that:
\begin{equation}
 \| f_{\lambda_0,\mu_0} - f^\star \|_2^2=
\inf_{(\lambda,\mu) \in \mathcal{M}_n} \| f_{\lambda,\mu} - f^\star \|_2^2 \leq \frac{C}{n}.
\label{eq:maj}
\end{equation}

At the same time, observe that \eqref{eq:interm} leads to:
\begin{eqnarray} 
\| \hat f_n - f^\star \|_2^2
& = &  (\lambda^\star - \hat\lambda_n)^2 \| \phi - \phi_{\hat\mu_n} \|_2^2   + \{\lambda^\star\}^2 \| \phi_{\hat\mu_n} - \phi_{\mu^\star} \|_2^2 \nonumber\\
& & \hspace{1.5cm} + 2 (\lambda^\star - \hat\lambda_n) \lambda^\star \langle  \phi - \phi_{\hat\mu_n} , \phi_{\hat\mu_n} - \phi_{\mu^\star} \rangle.\nonumber
\end{eqnarray}
Using Proposition \ref{prop:CS} with $a = \hat\mu_n$ and $b=\mu^\star - \hat\mu_n$ and \eqref{linkphimu},  a positive constant $c$ exists such that:
\begin{eqnarray*}
\lefteqn{
\| \hat f_n - f^\star \|_2^2}\nonumber\\
& \geq &  (\lambda^\star - \hat\lambda_n)^2 \| \phi - \phi_{\hat\mu_n} \|_2^2   + \{\lambda^\star\}^2 \| \phi_{\hat\mu_n} - \phi_{\mu^\star} \|_2^2 \nonumber \nonumber\\
& & \hspace{0.1cm} - 2\left| \lambda^\star - \hat\lambda_n \right| \lambda^\star \| \phi - \phi_{\hat\mu_n} \|_2 \| \phi_{\hat\mu_n} - \phi_{\mu^\star} \|_2 \left(1-c \| \phi  - \phi_{\mu^\star} \|_2^2\right) \nonumber \\
& \geq & (\lambda^\star - \hat\lambda_n)^2 \| \phi - \phi_{\hat\mu_n} \|_2^2   + \{\lambda^\star\}^2 \| \phi_{\hat\mu_n} - \phi_{\mu^\star} \|_2^2 \nonumber \nonumber\\
& & \hspace{0.1cm} - \left[ (\lambda^\star - \hat\lambda_n)^2 \| \phi - \phi_{\hat\mu_n} \|_2^2 +  \{ \lambda^\star\}^2  \| \phi_{\hat\mu_n} - \phi_{\mu^\star} \|_2^2 \right] \left(1-c \| \phi  - \phi_{\mu^\star} \|_2^2\right) \nonumber \\
& \geq & c (\lambda^\star - \hat\lambda_n)^2 \| \phi - \phi_{\hat\mu_n} \|_2^2 \| \phi - \phi_{\mu^\star} \|_2^2 
+ c \{\lambda^\star\}^2 \| \phi_{\hat\mu_n} - \phi_{\mu^\star} \|_2^2  \| \phi - \phi_{\mu^\star} \|_2^2.    \nonumber  
\end{eqnarray*}
We then obtained the crucial inequality:
\begin{equation}\label{eq:vit}
\| \hat f_n - f^\star \|_2^2 \geq  c \underline{\kappa}^2 (\lambda^\star - \hat\lambda_n)^2  \|\hat\mu_n\|^2 \|\mu^\star\|^2 + c \underline{\kappa}^2 \{\lambda^\star\}^2\|\mu^\star\|^2 \|\hat\mu_n - \mu^\star\|^2.
 \end{equation}
 
We see here the central role of the refinement of the Cauchy-Schwarz inequality (see Appendix \ref{s:CS}) to obtain a tractable bound that involves the parameters of the mixture themselves, from the bound on the $\LL^2$-norm of $\hat f_n - f^\star$. We now use the oracle inequality on $\|\hat{f}_n-f^\star\|_2^2$ to deduce that a constant $C>0$ exists such that:
\begin{equation} 
\EE \left[   (\lambda^\star - \hat\lambda_n)^2  \|\hat\mu_n\|^2 \|\mu^\star\|^2 +  \{\lambda^\star\}^2\|\mu^\star\|^2 \|\hat\mu_n - \mu^\star\|^2 \right] \leq \frac{C \log^2 n}{n}.
\label{eq:toto}
\end{equation}

In particular, we immediately deduce from \eqref{eq:toto} that: 
$$ \EE  \left[  \{\lambda^\star\}^2\|\mu^\star\|^2 \|\hat\mu_n - \mu^\star\|^2 \right] \leq \frac{C \log^2 n}{n}.$$
This result is uniform in $(\lambda^\star,\mu^\star)$, we obtain the proof of Theorem \ref{thm:vitesse_mu}. \hfill$\square$
% \end{proof}
%\begin{flushright}
%$\Box$
%\end{flushright}

Unfortunately, we cannot directly use a similar approach for the estimation of $\lambda^\star$. Indeed, we have to first ensure that $\hat \mu_n$ is close to $\mu^\star$ with a large enough probability. 

\subsection{Proof of Theorem \ref{thm:vitesse_lambda}}
\label{s:preuve_l}
%\begin{proof}[]
Let $\mathcal{B}$ and $\mathcal{D}$ be the events respectively defined as: 
\begin{equation} 
\mathcal{B} = \left\lbrace \| \hat f_n - f_{\lambda_0,\mu_0} \|_2 \leq \sqrt{\frac{1 }{n}}  \right\rbrace 
\label{eq:2}
\end{equation}
and 
\begin{equation}
\label{eq:2bis}
\mathcal{D}=  \left\lbrace | \mathcal{G}_n(\hat\lambda_n, \hat\mu_n) | \leq  \frac{16 \| \phi \|_\infty  \log ( n| \mathcal{M}_n |)}{3\sqrt{n}}\right\rbrace.
\end{equation}
Below, the control of the quadratic risk of $\hat\mu_n$ will be investigated according to the partition $\mathcal{B}, \mathcal{B}^c \cap \mathcal{D}$ and $\mathcal{B}^c\cap\mathcal{D}^c$.

\paragraph{Control of the risk on $\mathcal{B}$} 
Equation \eqref{eq:m1} together with (\ref{eq:maj}) indicates that:
$$
 \|\hat{f}_n -f^\star\|_2^2\  \mathds{1}_\mathcal{B}  \leq \frac{C}{n}.
$$
Then, Equation \eqref{eq:vit} implies that:
\begin{equation} 
 \left\| \hat\mu_n - \mu^\star  \right\|^2 \mathds{1}_\mathcal{B} \leq \frac{C }{n\{\lambda^\star\}^2\|\mu^\star\|^2} \leq \frac{C \|\mu^\star\|^2}{\ell_n^2}.
\label{eq:1}
\end{equation}

\paragraph{Control of the risk on $\mathcal{B}^c \cap \mathcal{D}$}
On the set $\mathcal{B}^c \cap \mathcal{D}$, we apply Inequality \eqref{eq:inter1}, which yields:
\begin{eqnarray*}
\|\hat{f}_n -f^\star\|_2^2\ \mathds{1}_{\mathcal{B}^c \cap \mathcal{D}}
& \leq& \frac{(1+\alpha)}{(1-\alpha)} \|f_{\lambda_0,\mu_0}-f^\star\|_2^2+\frac{1}{2\alpha} | \mathcal{G}_n(\hat\lambda_n, \hat\mu_n) |^2\ \mathds{1}_{\mathcal{B}^c \cap \mathcal{D}} \\
&\leq& C \frac{\log^2 (n|\mathcal{M}_n|)}{n}
\end{eqnarray*}
for some positive constant $C$. Since the size of $|\MMn|$ is a polynomial of $n$,  we can find a constant $C$ 
such that Equation \eqref{eq:vit} leads to:
\begin{equation}
  \left\|  \hat\mu_n - \mu^\star \right\|^2 \mathds{1}_{\mathcal{B}^c \cap \mathcal{D}} \leq C \frac{\log^2 n}{n\{\lambda^\star\}^2\|\mu^\star\|^2} \leq C \frac{\log^2 n}{\ell_n^2} \|\mu^\star\|^2.
\label{eq:3}
\end{equation}
Since we assume that $(\lambda^\star,\mu^\star) \in \TThl$ with 
$\ell_n / \log n \longrightarrow + \infty$ when $n \longrightarrow + \infty$, Equations \eqref{eq:1} and \eqref{eq:3} imply that for large enough $n$,
\begin{equation*}%\label{eq:tectec}
\left\| \hat\mu_n - \mu^\star\right\|^2 \left[ \mathds{1}_{\mathcal{B}} +  \mathds{1}_{\mathcal{B}^c \cap \mathcal{D}} \right]\leq \frac{\|\mu^{\star}\|^2}{4}.
\end{equation*}
Remark that  for any $x$ and $y$: $\|x-y\| \leq \frac{\|y\|}{2}$ implies that $\|y\| \geq 2 \|y\|-2\|x\|$ (using the triangle inequality), which in turns yields $\|y\| \leq 2 \|x\|$. Applying this simple remark to the former inequality leads to:
\begin{equation}\label{eq:majomu}\|\mu^\star\|^2  \left[ \mathds{1}_{\mathcal{B}} +  \mathds{1}_{\mathcal{B}^c \cap \mathcal{D}} \right]\leq 4 \|\hat\mu_n\|^2 \left[ \mathds{1}_{\mathcal{B}} +  \mathds{1}_{\mathcal{B}^c \cap \mathcal{D}} \right].\end{equation}

\paragraph{Control of the risk on $\mathcal{B}^c \cap \mathcal{D}^c$} 
Applying (\ref{controlG}) we can check that: 
$$
\PP( \mathcal{B}^c \cap \mathcal{D}^c ) \leq \PP( \mathcal{D}^c)  \leq \frac{C}{n} %\quad \mbox{as soon as } \quad n> \frac{1}{|\mathcal{M}|} e^{\frac{9}{16} \| \phi \|_\infty}.
$$
for some positive constant $C$. 

\paragraph{Synthesis}
Using \eqref{eq:majomu}, a large enough $N$ exists  such that for $n \geq N$:
\begin{eqnarray*}
\lefteqn{\EE [ (\hat\lambda_n- \lambda^\star)^2 \|\mu^\star\|^4 ] }\\
& = & \EE [ (\hat\lambda_n- \lambda^\star)^2 \|\mu^\star\|^4 (\mathds{1}_{\mathcal{B}}+\mathds{1}_{\mathcal{B}^c\cap\mathcal{D}})] + \EE [ (\hat\lambda_n - \lambda^\star)^2 \|\mu^\star\|^4 \mathds{1}_{\mathcal{B}^c\cap\mathcal{D}^c}] , \\
& \leq & 4 \EE [ (\hat\lambda_n - \lambda^\star)^2 \|\mu^\star\|^2 \|\hat\mu_n\|^2 ] + d^2 M^4 \PP(\mathcal{D}^c),\\
& \leq & \frac{C \log^2(n)}{n},
\end{eqnarray*}
for some constant $C>0$, according to \eqref{eq:toto}. This result being uniform in $(\lambda^\star,\mu^\star)$, we obtain the proof   of Theorem \ref{thm:vitesse_lambda}. \hfill $\square$
%\end{proof}
%\begin{flushright}
%$\Box$
%\end{flushright}

\section{Link between the $\|.\|_2$ norm and the Wasserstein distance(s) }\label{sec:metric_proof}
%\textcolor{magenta}{S'il faut reduire, on peut sabrer dans cette preuve, en laissant le tout dans une version longue. }
%\textcolor{red}{Extraire des preuves des Th 5.1 et 5.2 les calculs de $W_2$ et $W_1$ et les mettre en Appendix A afin de pr\'eparer une possible disparition de l'appendix A }

\begin{proof}[Proof of Theorem \ref{theo:L2W2}]

Below, we will establish that the following inequality (stated in Theorem \ref{theo:L2W2}) holds:
\begin{equation}
\label{eq:comparaisonWL2}
W_2^4(G_{\lambda,\mu},G_{\lambda',\mu'}) \lesssim \|f_{\lambda,\mu}-f_{\lambda',\mu'}\|_2^2.
\end{equation}

\noindent \textit{Expression of $W_2$:}
below, we make explicit the link between the $\mathbb{L}^2-$loss on the densities $f_{\lambda,\mu}$ and $f_{\lambda',\mu'}$ and the Wasserstein distance between %the mixture distributions 
$
G_{\lambda,\mu}= (1-\lambda) \delta_{0} + \lambda \delta_{\mu}$ and $G_{\lambda',\mu'}= (1-\lambda') \delta_{0} + \lambda' \delta_{\mu'},
$
where $\delta_{a}$ refers to the Dirac mass at point $a$. First, we provide an expression for the term $W_2(G_{\lambda,\mu},G_{\lambda',\mu'})$. Since the role played by $(\lambda,\mu)$ and $(\lambda',\mu')$ is symmetric, in the following, we assume without loss of generality that $\lambda \leq \lambda'$.  First, the quantity $W_2(G_{\lambda,\mu},G_{\lambda',\mu'})$ can be rewritten as
$$ W_2^2(G_{\lambda,\mu},G_{\lambda',\mu'}) = \inf_{q\in \mathcal{Q}} \left[ q_{12} \| \mu'\|^2 + q_{21} \| \mu \|^2 + q_{22} \| \mu - \mu' \|^2   \right],$$
where 
\begin{eqnarray*}
\mathcal{Q} &=& \left\lbrace q=(q_{11},q_{12},q_{21},q_{22}) \in [0,1]^4: \right. \\
& & \hspace*{0.3cm} \left. q_{11}+q_{12} = 1-\lambda \ ; \ q_{21} + q_{22} = \lambda \ ; \ q_{11}+q_{21} = 1-\lambda' \ ; \ q_{12}+q_{22} = \lambda'    \right\rbrace.
\end{eqnarray*}
After some computations, the set $\mathcal{Q}$ can be rewritten as
$$ \mathcal{Q} = \left\lbrace q \in [0,1]^4: \ q_{12} = \lambda' - q_{22} \ ; \ q_{21}  = \lambda - q_{22} \ ; \ q_{11} = 1- \lambda - \lambda' + q_{22} \right\rbrace.$$
Hence, 
$$ 
W_2^2(G_{\lambda,\mu},G_{\lambda',\mu'}) = \inf_{q_{22} \in [ (\lambda+\lambda'-1)\vee 0 , \lambda]} \left[ (\lambda' - q_{22}) \| \mu' \|^2 + (\lambda - q_{22}) \| \mu \|^2 + q_{22} \| \mu - \mu' \|^2   \right].
$$
The last equation yields
\begin{eqnarray} 
\lefteqn{W_2^2(G_{\lambda,\mu},G_{\lambda',\mu'})}  \label{eq:revision1}\\
 & \hspace{-0.3cm}= & \hspace{-0.3cm} \left\lbrace
\begin{array}{lcl}
(\lambda' - \lambda) \|\mu' \|^2 + \lambda \| \mu-\mu' \|^2 & \mathrm{if} & \| \mu \|^2 + \|\mu' \|^2 \geq \| \mu - \mu' \|^2, \\ 
  & & \\
\lambda \| \mu \|^2 + \lambda' \| \mu' \|^2 & \mathrm{if} & \| \mu \|^2 + \|\mu' \|^2 < \| \mu - \mu' \|^2 \ \mathrm{and} \ \lambda +\lambda' \leq 1,\\
 & & \\
(1-\lambda') \| \mu \|^2 + (1-\lambda) \| \mu' \|^2  &  \mathrm{if} & \| \mu \|^2 + \|\mu' \|^2 < \| \mu - \mu' \|^2 \ \mathrm{and} \ \lambda +\lambda' > 1.\\
+ (\lambda + \lambda' - 1 ) \| \mu - \mu' \|^2 & \\
\end{array} \right. \nonumber
\end{eqnarray}

\noindent \textit{Upper bound on $W_2$:}
The previous expression for $W_2(G_{\lambda,\mu},G_{\lambda',\mu'})$ allows to prove that
\begin{equation} 
W_2^2(G_{\lambda,\mu},G_{\lambda',\mu'}) \leq  (\lambda'-\lambda)\|\mu'\|^2+ \lambda\|\mu-\mu'\|^2 .
\label{eq:revision2}
\end{equation}
Indeed, according to (\ref{eq:revision1}), this bound turns to be an equality when $\| \mu \|^2 + \|\mu' \|^2 \geq \| \mu - \mu'\|^2$. When, $\| \mu \|^2 + \|\mu' \|^2 < \| \mu - \mu' \|^2 \ \mathrm{and} \ \lambda +\lambda' \leq 1$, we have 
\begin{eqnarray*} 
W_2^2(G_{\lambda,\mu},G_{\lambda',\mu'}) & = & (\lambda' -\lambda) \|\mu' \|^2 + \lambda \| \mu-\mu' \|^2 + \lambda (\| \mu' \|^2 + \| \mu \|^2 - \| \mu- \mu' \|^2) \\
& \leq &  (\lambda' -\lambda) \|\mu' \|^2 + \lambda \| \mu-\mu' \|^2.
\end{eqnarray*}
In the last case displayed in (\ref{eq:revision1}), namely when $\| \mu \|^2 + \|\mu' \|^2 < \| \mu - \mu' \|^2 \ \mathrm{and} \ \lambda +\lambda' > 1$, we obtain 
\begin{eqnarray*}
W_2^2(G_{\lambda,\mu},G_{\lambda',\mu'})  & = & (1-\lambda') \| \mu \|^2 + (1-\lambda) \| \mu' \|^2 + (\lambda + \lambda' - 1 ) \| \mu - \mu' \|^2 \\
& = & (\lambda' - \lambda) \| \mu' \|^2 + \lambda \| \mu - \mu' \|^2 + (1-\lambda') \left[ \| \mu' \|^2 + \| \mu \|^2 - \| \mu- \mu' \|^2 \right]. \\
& \leq & (\lambda' -\lambda) \|\mu' \|^2 + \lambda \| \mu-\mu' \|^2.
\end{eqnarray*}
This entails (\ref{eq:revision2}). We get from this inequality, still assuming $\lambda \leq \lambda'$
\begin{eqnarray*}
W_2^2(G_{\lambda,\mu},G_{\lambda',\mu'}) 
& \leq  & (\lambda'-\lambda)\|\mu'\|^2+ \lambda\|\mu-\mu'\|^2 \\
& \leq & (\lambda'-\lambda)\|\mu'\|^2+ \lambda( \| \mu\| + \| \mu' \|)\|\mu-\mu'\|, \\
& \leq & (\lambda'-\lambda)\|\mu'\|^2+ ( \lambda\| \mu\| + \lambda'\| \mu' \|)\|\mu-\mu'\|, \\
& \leq & (\lambda'-\lambda)\|\mu'\| \| \mu \| +  (\lambda'-\lambda)\|\mu'\| \| \mu - \mu'\| + ( \lambda\| \mu\| + \lambda'\| \mu' \|)\|\mu-\mu'\|, \\
& \leq & (\lambda'-\lambda)\|\mu'\| \| \mu \|  + 2( \lambda\| \mu\| + \lambda'\| \mu' \|)\|\mu-\mu'\|.
\end{eqnarray*}
From this latter inequality, we obtain 
\begin{equation}
W_2^4(G_{\lambda,\mu},G_{\lambda',\mu'}) \leq 8 \left[ (\lambda'-\lambda)^2 \|\mu'\|^2 \| \mu \|^2 +  ( \lambda\| \mu\| + \lambda'\| \mu' \|)^2 \|\mu-\mu'\|^2 \right].
\label{eq:revision3}
\end{equation}
In the other hand, Inequality (\ref{eq:vit}) indicates that
$$\| f_{\lambda,\mu} - f_{\lambda',\mu'} \|_2^2 \geq  c \underline{\kappa}^2 (\lambda' - \lambda)^2  \|\mu\|^2 \|\mu'\|^2 + c \underline{\kappa}^2 \{\lambda'\}^2\|\mu'\|^2 \|\mu - \mu'\|^2.$$
Since the role played by $(\lambda,\mu)$ and $(\lambda',\mu')$ is symmetric, we obtain in fact 
$$ \| f_{\lambda,\mu} - f_{\lambda',\mu'} \|_2^2 \geq  c \underline{\kappa}^2 (\lambda' - \lambda)^2  \|\mu\|^2 \|\mu'\|^2 + \frac{c \underline{\kappa}^2}{2} \left(  \{\lambda'\}^2\|\mu'\|^2 + \{\lambda\}^2\|\mu\|^2 \right) \|\mu - \mu'\|^2,$$
which together with (\ref{eq:revision3}) implies \eqref{eq:comparaisonWL2}. Using this inequality with $f_{\hat\lambda_n,\hat\mu_n}$ and $f_{\lambda^\star,\mu^\star}$, and according to Theorem~\ref{th:oracle}, we conclude the proof of Theorem \ref{theo:L2W2}.
\end{proof}

%\begin{proof}[Proof of Theorem \ref{theo:W1}].  ---- c'est faux
%The proof is a direct consequence of Theorem \ref{th:lb:R1} and of a lower bound on the Wasserstein distance $W_1$ between two-components mixture distributions. 
%First, assuming without loss of generality that $\lambda' >\lambda$ and as the computation of $W_2$ in the proof of Theorem \ref{theo:L2W2}, we obtain that
%$$
%W_1(G_{\lambda',\mu'},G_{\lambda,\mu})  = \inf_{q_{22} \in [ (\lambda+\lambda' -1)\vee 0 , \lambda]} \left[ (\lambda' - q_{22}) \| \mu' \| + (\lambda - q_{22}) \| \mu \| + q_{22} \| \mu - \mu' \|  \right].
%$$
%The infimum being achieved at $q_{22}=\lambda$, we get that
%$$
%W_1(G_{\lambda',\mu'},G_{\lambda,\mu}) = (\lambda' - \lambda) \|\mu'\|  + \lambda \|\mu -\mu'\|.
%$$
%In particular, we have 
%$$ \mathbb{E} [W_1^2(G_{\hat\lambda,\hat\mu},G_{\lambda,\mu})] \geq \mathbb{E}[  \lambda^2 \|\mu -\hat\mu\|^2 ],$$
%for any estimator $(\hat\lambda,\hat\mu)$. This inequality, together with item (i) of Theorem \ref{th:lb:R1} leads to the desired result. 
%\end{proof}

\newpage
\appendix
\section{Technical results}\label{s:tec}

\subsection{Identifiability result}

\begin{proof}[Proof of Proposition \ref{prop:id}]
We assume that two parameters $\theta_1=(\lambda_1,\mu_1)$ and $\theta_2=(\lambda_2,\mu_2)$ exist such that $f_{\theta_1}=f_{\theta_2}$. In that case, consider the Fourier transform of $X$ whose density is $f_{\theta_1}$. This Fourier transform is given by
$$
\varphi_X(\xi) 
= \mathbb{E}[e^{\ii \xi \ps X}] 
= \left[ (1-\lambda_1) + \lambda_1 e^{\ii \xi \ps \mu_1}\right] \hat{\phi}(\xi),
$$
where $\hat{\phi}$ is the Fourier transform of $\phi$ and $\ii$ is the complex number such that $\ii^2=-1$. Since $f_{\theta_1} = f_{\theta_2}$, we then deduce that:
$$
\forall \xi \in \RR^d \qquad 
\left[ (1-\lambda_1) + \lambda_1 e^{\ii \xi \ps \mu_1}\right] \hat{\phi}(\xi) = \left[ (1-\lambda_2) + \lambda_2 e^{\ii \xi \ps \mu_2}\right] \hat{\phi}(\xi).
$$
Since $\phi \in \LL^1(\RR^d)$, $\hat{\phi}$ is continuous and cannot be zero everywhere. Thus, we can find an open set $I \subset \RR^d$ such that $\hat{\phi}(\xi) \neq 0$ in $I$ and the Lebesgue measure of $I$ is strictly positive. Hence, 
$$
\forall \xi \in I \quad (1-\lambda_1) + \lambda_1 e^{\ii \xi \ps \mu_1} = (1-\lambda_2) + \lambda_2 e^{\ii \xi  \ps \mu_2},
$$
and from the analytical property of the exponential map, we deduce that:
$$
\forall \xi \in I \qquad (1-\lambda_1) + \lambda_1 [ \cos(\xi \ps\mu_1)+ \ii \sin (\xi\ps \mu_1)] = (1-\lambda_2) + \lambda_2 [ \cos(\xi \ps\mu_2)+ \ii \sin (\xi \ps\mu_2)] 
$$
Identifying now the imaginary parts yields: 
$$
\forall \xi \in I \qquad \lambda_1 \sin ( \xi \ps\mu_1)=\lambda_2 \sin ( \xi \ps \mu_2).
$$
If we write $\mu_1=(\mu_1^{(1)},\ldots,\mu_1^{(d)})$ and $\mu_2=(\mu_2^{(1)},\ldots,\mu_2^{(d)})$, we deduce that
\begin{align*}
\forall \xi=(\xi_1,\ldots,\xi_d): \quad  &
\lambda_1 \left[\sin ( \xi_1 \mu_1^{(1)}) \cos (\sum_{j=2}^d \xi_j \mu_1^{(j)}) + \cos (\xi_1 \mu_1^{(1)}) \sin (\sum_{j=2}^d \xi_j \mu_1^{(j)})\right] \\
& = \lambda_2 \left[\sin ( \xi_1 \mu_2^{(1)}) \cos (\sum_{j=2}^d \xi_j \mu_2^{(j)}) + \cos (\xi_1 \mu_2^{(1)}) \sin (\sum_{j=2}^d \xi_j \mu_2^{(j)})\right].
\end{align*}
Considering now the function of the variable $\xi_1$, it is classical that the family of functions $(\xi_1 \mapsto \sin (\alpha_1 \xi_1),\xi_1\mapsto \sin(\alpha_2 \xi_1))$ is linearly independent if and only if $|\alpha_1| \neq |\alpha_2|$. We  can deduce that, necessarily, $\mu_1^{(1)} = \pm \mu_2^{(1)}$ and therefore
$\cos (\xi_1 \mu_1^{(1)}) = \cos (\xi_1 \mu_2^{(1)})$, which shows that
$ \lambda_1 \sin (\sum_{j=2}^d \xi_j \mu_1^{(j)}) = \lambda_2 \sin (\sum_{j=2}^d \xi_j \mu_2^{(j)})$ for all $\xi \in I$. We then end the argument with an easy recursion: we obtain that $\lambda_1 \sin (\xi_d \mu_1^{(d)}) = \lambda_2 \sin (\xi_d \mu_2^{(d)})$ so that $\mu_1^{(d)} = \pm \mu_{2}^{(d)}$. 
Since $\lambda_1$ and $\lambda_2$ are positive,  then $\mu_1^{(d)}=\mu_2^{(d)}$, which in turn implies that $\mu_1^{(j)}=\mu_2^{(j)}$ for all the coordinates $j \in \{1,\ldots,d\}$.
\end{proof}

\subsection{Connection between $\|\phi-\phi_{\mu}\|_2$ and $|\mu|$}

\begin{prop}
\label{prop:cnorme} Let any $M>0$ be given and assume that $\phi$ satisfies $\HS$ and $\HL$, then
two constants $0<\underline{\kappa}<\overline{\kappa}<+\infty$ exist such that:
\begin{equation}\label{eq:norme2_mu}
\forall (\mu,\tilde\mu) \in [-M,M]^d\times[-M,M]^d \qquad 
 \underline{\kappa}\|\mu-\tilde\mu\|^2 \leq \|\phi_\mu - \phi_{\tilde\mu}\|_2^2 \leq \overline{\kappa} \|\mu-\tilde\mu\|^2.
 \end{equation}
\end{prop}

\begin{proof}
We prove the upper and lower bounds separately. According to the shift invariance of the $\LL^2$ norm, we only establish these inequalities when $\tilde{\mu}=0$. Using $\HL$, the upper bound simply derives from:
$$
\|\phi  - \phi_{\mu}\|_2^2 = \int_{\RR^d} \left[\phi(x)-\phi(x-\mu)\right]^2 dx \leq \int_{\RR^d} \|\mu\|^2 g^2(x) dx  = \|\mu\|^2 \|g\|_2^2,
$$
which is the desired inequality if we choose $\overline{\kappa}=\|g\|^2$. Concerning the lower bound, we have:
$$
\frac{\|\phi(.) - \phi(.-\mu)\|_2^2}{\|\mu\|^2} = \int_{\RR^d} \left[\frac{\phi(x)-\phi(x-\mu)}{\|\mu\|}\right]^2 dx.
$$
We write $\mu=\|\mu\| e$ where $e$ is a unit vector of the sphere.
Inequality \eqref{eq:glip} brought by Assumption $\HL$ makes it possible to apply the Lebesgue convergence theorem, which implies:
\begin{eqnarray*} 
\lim_{\|\mu\| \longrightarrow 0} \frac{\|\phi(.) - \phi(.-\mu)\|^2}{\|\mu\|^2} 
& = & \int_{\RR^d} \lim_{\|\mu\|\rightarrow 0}\left[\frac{\phi(x)-\phi(x-\mu)}{\|\mu\|}\right]^2 dx ,\\
& = & \|\nabla \phi \bullet e \|^2 = \|d_e[\phi]\|^2> 0.
\end{eqnarray*}
Indeed, $\phi$ being differentiable ($\phi \in \mathcal{C}^1(\RR^d)$), $
  \frac{\phi(x)-\phi(x-\mu)}{\|\mu\|} \longrightarrow d_e[\phi](x)$ almost surely when $\|\mu\| \longrightarrow 0$.

Now, $\phi$ is continuous and 
$\psi: 
\mu \longrightarrow \frac{\|\phi - \phi_{\mu}\|_2^2}{\|\mu\|^2} \in \mathcal{C}^{0}([-M,M]^d,\RR)
$ from the Lebesgue convergence theorem.  This continuous map $\psi$ attains its lower bound on $[-M,M]^d$ and the identifiability result of Proposition \ref{prop:id} implies that this lower bound is positive. This leads to the existence of $\underline{\kappa}>0$ such that:
$$ \|\phi - \phi_{\mu}\|_2^2 \geq \underline{\kappa}  \|\mu\|^2.$$
\end{proof}

\subsection{Log-concave distributions}\label{s:log_concave}
In this section, we establish that most of the log-concave real distributions satisfy the assumptions $\HS,\HL$ and $\HD$. For this purpose, we introduce the associated class of probability measures:
$$
\mathcal{LC} := \left\{ \phi(.) = e^{-u(.)}: u \text{ is convex}, u \in \mathcal{C}^2(\RR^d) \, \text{and} \, \|\nabla u\|+\|D^2 u\|=o_{ \infty}(u) \right\}. 
$$
The set of possible densities is rich and contains Gaussian or Gamma distributions. However, the set $\mathcal{LC}$ does not capture the situation where $u(x) = e^{|x|}$ or $u(x)=e^{x^2}$ since $u$ exhibits variations that are too great for large values of $x$.

\begin{prop}\label{prop:lc}
Assume that $\mu$ varies in $[-M,M]^d$ and that $\phi \in \mathcal{LC}$. Let $\varepsilon \in (0,M)$. If we set:
$$
	g(x) := g_1(x) \vee g_2(x) \vee g_3(x)$$
	with
{\small
	$$
	g_1(x) := 
	\sqrt{\frac{\sup_{e \in \mathcal{S}^1} \int_{[x-M e,x]} \langle \nabla \phi(t),e\rangle^2 dt}{\varepsilon}}, \, 
	g_2(x) := 
	\sqrt{\frac{\sup_{e \in \mathcal{S}^1} \int_{[x,x+M e ]} \langle \nabla \phi(t),e\rangle^2 dt}{\varepsilon}},$$
	and $$	g_3(x) := \underset{t\in B(x,\varepsilon)}{\sup} \|\nabla \phi(t)\|.$$}
Then, $\HL$ and $\HD$ hold:
\begin{itemize}
\item[$i)$] $\forall \mu \in [-M,M]^d \quad \forall x \in \RR^d \qquad |\phi(x)-\phi_{\mu}(x)| \leq \|\mu\| \,  g(x).$
\item[$ii)$]  $ g \phi^{-1/2}\in \mathbb{L}^2(\RR^d)$
\item[$iii)$] $D^2\phi\ \phi^{-1/2} \in\mathbb{L}^2(\RR^d) $
\end{itemize}
\end{prop}

\begin{proof}
We provide a proof in the case when $\phi \in \mathcal{C}^2$. This proof can be extended when $\phi \in \mathcal{C}^2_p$ according to some small modifications that are left to the reader, it then makes possible to extend our results to the Laplace distributions for example.\\
\underline{Proof of $(i)$:}
Remark first that $\forall \mu \in [-M,M]^d$, a unit vector $e \in \mathcal{S}^1$ exists such that $\mu = \|\mu\| e$ and in that case
$$
 \forall x \in  \RR^d \qquad 
|\phi(x)-\phi_{\mu}(x)| = \left| \int_{[x-\mu,x]} \langle \nabla \phi(t),e\rangle dt \right| \leq \sqrt{\|\mu\|} \sqrt{\int_{[x-\mu, x]}  \langle \nabla \phi,e\rangle ^2},
$$
where $[x-\mu,x]$ refers to the segment that joins $x-\mu$ to $x$ in $\RR^d$ and the last upper bound comes from the Cauchy-Schwarz inequality. Let $\varepsilon \in (0,M)$. If $\|\mu\| \in [\varepsilon, M]$, 
we obtain that: 
$$
	|\phi(x)-\phi_{\mu}(x)| \leq \|\mu\| \left( g_1(x)  \vee  g_2(x)\right),
$$
where $g_1$ and $g_2$ are defined in the statement of the Proposition.
%$$
%	g_1(x) = \sup_{e \in \mathcal{S}^1} \int_{[x-M,x]} \langle \nabla \phi(u),e\rangle^2 du \qquad \text{and} \qquad g_2(x) = \sup_{e \in \mathcal{S}^1}  \int_{[x,x+M]} \langle \nabla \phi(u),e\rangle^2 du.
%$$ 
Finally, we should remark that if $\|\mu\|\in [0,\varepsilon)$, then
$$
	|\phi(x) - \phi_\mu(x)| \leq \|\mu\|  \underset{t\in B(x,\varepsilon)}{\sup} \|\nabla \phi(t)\| := \|\mu\| g_{3}(x).
$$
It proves that $g=g_1 \vee g_2 \vee g_3$ satisfies the desired inequality.

\underline{Proof of $(ii)$:}
In order to prove that $g \phi^{-1/2} \in \mathbb{L}^2(\RR^d)$, we separately prove that $g_1^2 \phi^{-1}, g_2^2 \phi^{-1}$ and $g_{3}^2 \phi^{-1}$ belong to $\LL^1(\mathbb{R}^d)$.
We should remark that since $g_1, g_2$ and $g_3$ are continuous functions, then we only have to check the integrability when $\|x\| \longrightarrow + \infty$.
$g_1$ and $g_2$ are rather similar and we only handle the integrability of $g_1^2 \phi^{-1}$.\\

We write
\begin{eqnarray*}
g_1^2(x) \phi^{-1}(x)& = &\varepsilon^{-1} e^{u(x)} \sup_{e \in \mathcal{S}^1} \int_{[x-M e,x]} \langle \nabla \phi(t),e\rangle^2 dt \\
& = & \varepsilon^{-1} \sup_{e \in \mathcal{S}^1}  e^{u(x)}  \int_{[x-M e,x]} \langle \nabla \phi(t),e\rangle^2 dt\\
& = &  \varepsilon^{-1} \sup_{e \in \mathcal{S}^1}  \underbrace{e^{u(x)}  \int_{[x-M e,x]} \langle \nabla u(t),e\rangle^2  e^{-2 u(t)} dt}_{:=G_e(x)}.
\end{eqnarray*}
At this stage, we are driven to consider the $1$-dimensional fonction $u_e(t) = u(x+(t-M) e)$, which is a convex function. We then have
$$
G_e(x) = e^{u_e(M)} \int_{0}^M u'_e(s)^2 e^{-2 u_e(s)} ds.
$$
We shall now produce a $1$-dimension argument with the convex function $u_e$. We assume that $u_e(M) \geq u_e(0)$, and know that $u_e'$ is an increasing map and positive: 
\begin{eqnarray*}
G_e(x) & \leq & u_e'(M) e^{u_e(M)} \int_{0}^M u_e'(s) e^{-2u_e(s)}  ds \\
& \leq & \langle \nabla u(x),e \rangle e^{u(x)} \frac{e^{-2u(x-M e)}-e^{-2 u(x)}}{2}\\
& \leq & \frac{\langle \nabla u(x),e \rangle}{2} e^{-2 u(x-M e)+u(x)}.
\end{eqnarray*}
The mean value theorem leads to:
$$ \exists \xi \in [x-Me, x] \qquad 
u(x-M e) = u(x) - M \langle \nabla u(\xi),e\rangle \geq u(x) - M \langle \nabla u(x),e\rangle.
$$

Consequently, we obtain:
$$G_e(x) \leq \frac{\langle u(x),e\rangle}{2} e^{-u(x)+2 M \|\nabla u(x)\|}.
$$
The density $\phi \in \mathcal{LC}$ and we can find $K$ large enough such that:
$$
\forall \|x\| \geq K \quad \forall e \in \mathcal{S}^1 \qquad -u(x) + 2 M \|\nabla u(x)\| \leq -(1-\eta) u(x)
$$
For such an $x$, we have $G_e(x)   \leq \frac{\langle \nabla u(x),e \rangle }{2} e^{-(1-\eta) u(x)} \in \mathbb{L}^1(\RR^d)$.\\
Concerning $g_2(x)\phi(x)^{-1}$, we can produce an almost identical argument left to the reader.
 We now consider $g_{3,\varepsilon}^2 \phi^{-1}$:
$$
g_{3,\varepsilon}^2(x) \phi^{-1}(x) =  \underset{t\in B(x,\varepsilon)}{\sup} \|\nabla u(t)\|^2 e^{-2u(t)+u(x)}.
$$
If $t\in[x-\varepsilon, x]$, the mean value theorem leads to:
\begin{eqnarray*}
u(t) &=& u(x) - \langle (x-t),\nabla u(\xi)\rangle \textrm{ with } \xi\in]t,x[\\
&\geq & u(x) - \varepsilon \sup_{B(x,\epsilon)} \|\nabla u\|.
\end{eqnarray*}

Using the fact that $\|D^2 u\|+\|\nabla u\|=o_{\infty}(u)$, we can find a positive constant $C>0$, a parameter $\eta \in (0,1)$ and
for $K$ large enough such that $\forall  \|x\| \geq K$:
\begin{equation}
\label{eqg3-1}
\|u\|(t)^2 e^{-2u(t) + u(x)} \leq C \|u(x)\| e^{-(1-\eta) u(x)}.
\end{equation}
Thus, \eqref{eqg3-1} imply that $g_{3,\varepsilon}^2\phi^{-1}\in\LL^1(\RR^d)$.
As a maximum of three functions in $\mathbb{L}^1(\RR^d)$, we deduce that $g^2 \phi^{-1} \in \LL^1(\RR^d)$.

\underline{Proof of $(iii)$:}
A direct computation shows that, almost surely:
$$
\{d_{jj}\phi\}^2 \phi^{-1} = [ d_{jj}u-\{d_{j}u\}^2]^{2} e^{-u} \leq 2 \{d_{jj}u\}^{2} e^{-u} + 2 \{d_j u\}^4 e^{-u}.
$$ 
Again, using the fact that $\|D^2 u\|+\|\nabla u\|=o_{\infty}(u)$, we can find a positive constant $C>0$, a parameter $\eta \in (0,1)$ and
a  large enough $K$ such that $\forall \|x\| \geq K$:
\begin{eqnarray*}
\{d_{jj}u\}^{2}(x) e^{-u(x)} &\leq &C  d_{jj}u(x) e^{-(1-\eta) u(x)} \\
& \leq & C d_j(d_j u(x) e^{-(1-\eta) u(x)})+ C (1-\eta) \{d_{j}u(x)\}^2 e^{-(1-\eta) u(x)} \\
& \leq & C d_j(d_j u(x) e^{-(1-\eta) u(x)}) + C^2 (1-\eta) d_j u(x)e^{-(1-\eta)^2 u(x)},
\end{eqnarray*}
which is integrable when $\|x\| \longrightarrow + \infty$. A similar argument leads to $ d_ju^4 e^{-u} \leq C d_ju e^{-(1-\eta)u}$.
We can repeat the same argument when $\|x\| \longrightarrow - \infty$ with an adaptation of the sign of $d_j u(x)$. 
We can conclude that $\{d_{jj}\phi\}^2 \phi^{-1} \in \LL^1(\RR^d)$.
\end{proof}

\section{Refinement of a Cauchy-Schwarz inequality}
\label{s:CS}

%\textcolor{magenta}{Discuter de la reduction de cette partie pour l'article principal, concern\'e tout pour la version sur Hal/Arxiv}

In this section, without loss of generality, we normalize the density $\phi$ to $1$ over $\RR^d$, meaning (with a slight abuse of notation) that:
 $$
\forall \mu \in \RR^d \qquad \|\phi_{\mu}\|_2=1.
$$
In what follows, we assume that $\phi$ satisfies $\HS$ and $\HL$. In particular, these conditions imply the ``asymptotic decorrelation" of the location model.
\begin{prop}\label{prop:decorrelation}
Assume that $\phi$ satisfies $\HS$, then:
$$\lim_{\|a\|\longrightarrow + \infty} \langle \phi,\phi_{a}\rangle = 0.$$
\end{prop}
\begin{proof}
The continuity of $\phi$ implies that $\phi$ is bounded by a constant $K$ on $\RR^d$ and that:
$$
\lim_{\|x\|\longrightarrow + \infty} \phi(x) = 0,
$$
which in turns implies that: 
$$
\lim_{\|a\|\longrightarrow + \infty} \langle \phi,\phi_{a}\rangle = 
\lim_{\|a\|\longrightarrow + \infty} \int \phi(x-a) \phi(x) dx = 0,
$$
from the Lebesgue dominated convergence theorem.
\end{proof}

\subsection{Main inequality}
We are interested in Proposition \ref{prop:CS}, which can be viewed as a refinement of the Cauchy-Schwarz inequality. Its proof relies on somewhat technical lemmas that are given in Appendix~\ref{sec:teclemma}, and on the following ratio:
\begin{equation}\label{eq:defR}
R(a,b) :=  \frac{\left| \langle \phi - \phi_a,\phi_{a+b}-\phi_a \rangle \right|}{\left\|\phi-\phi_a\right\|_2  \left\|\phi_{a+b}-\phi_a\right\|_2}:= \frac{\left| N(a,b)\right|}{D(a,b)}.
\end{equation}
According to Lemma \ref{lemma:Rcontinu}, the function $(a,b) \mapsto R(a,b)$ defines a continuous map as soon as $a\neq 0$ and $b \neq 0$. 

As indicated above, Proposition \ref{prop:CS} is crucial for the proof of Theorems \ref{thm:vitesse_mu} and \ref{thm:vitesse_lambda}. At this stage, a standard Cauchy-Schwarz inequality would then conclude that $R(a,b) \leq 1$. Indeed, such an upper bound is not enough for our purpose and we need to improve it when $R$ becomes close to $1$. To obtain such an improvement, we will take advantage of the fact that each $\phi_a$ belongs to the unit sphere (\textit{i.e.} $\|\phi_a\|_2=1$ for all $a$), of the identifiability of the model, and of the asymptotic decorrelation  when the location is arbitrarily large: $\langle \phi,\phi_a\rangle \longrightarrow 0$ as $\|a\| \longrightarrow + \infty$.

The main ingredients of the proofs will then use some continuity and differentiability arguments associated with  multivariate second- and third-order expansions of the numerator $N(a,b)$ and denominator $D(a,b)$ involved in $R(a,b)$.
It appears that the next inequality will be shown to be ``easy" as soon as $a$ and $b$ are located outside the diagonal, meaning that $a+b$ is quite different from $0$ since in that case $R$ will be shown to be lift away from 1. This behaviour is described in Lemma  \ref{lemma:out} (see also Figure \ref{fig:roadmap}).

The situation when $a$ is close to $-b$ is more involved and the joint behaviour of $\phi-\phi_a$ and $\phi_a-\phi_{a+b}$ will be crucial. To quantify this link, we will need to consider two cases: first when the diagonal $a+b=0$ is itself near the origin $a=b=0$ (Lemma \ref{lemma:0diagonal}), second when the diagonal is far enough from the origin (Lemma \ref{lemma:diagonal}) (see Figure \ref{fig:roadmap}).

\begin{center}
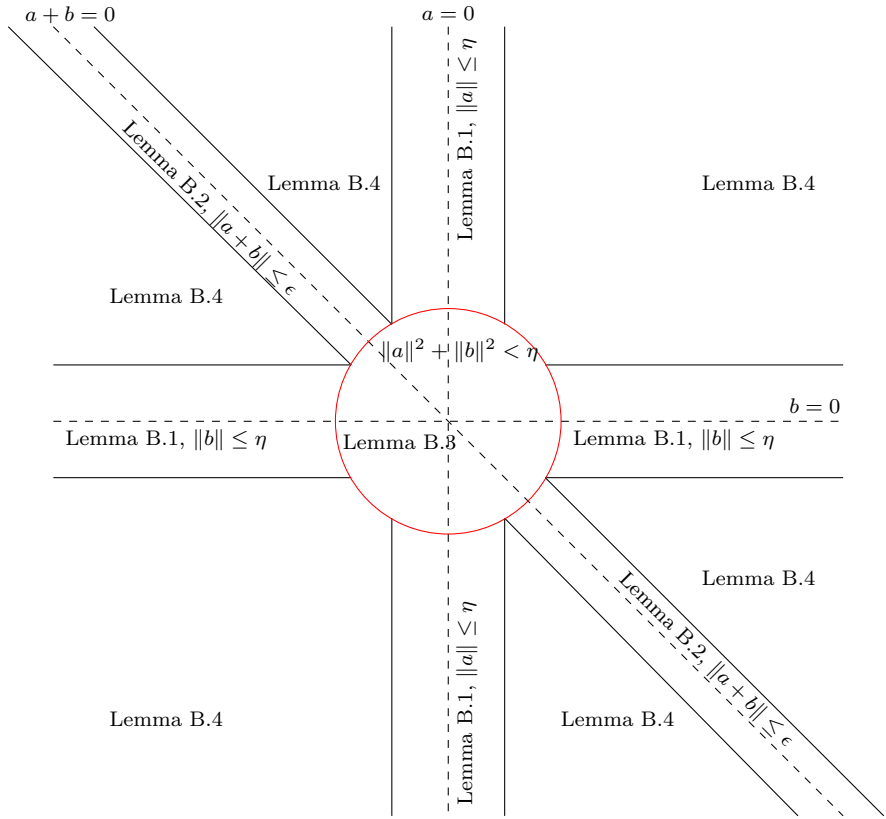
\begin{figure}[htbp]
\begin{tikzpicture}[scale=0.75]
\draw[dashed] (-7,0)--(7,0);
\draw[dashed] (0,7)--(0,-7);
\draw[dashed] (-7,7)--(7,-7);

\draw (-7.8,7)--(-1.72,1);
\draw (6.2,-7)--(1,-1.72);
\draw (-6.28,7)--(-1,1.72);
\draw (1.72,-1)--(7.72,-7);

\draw (-7,1)--(-1.72,1);
\draw (-7,-1)--(-1.72,-1);
\draw (7,-1)--(1.72,-1);
\draw (7,1)--(1.72,1);
\draw (-1,7)--(-1,1.72);
\draw (1,7)--(1,1.72);
\draw (-1,-7)--(-1,-1.72);
\draw (1,-7)--(1,-1.72);
\draw[red] (0,0) circle (2cm);
\draw (4.8,-4) node[below,rotate=-45]{Lemma B.2, $\|a+b\| \leq \epsilon$};
\draw (-4,4) node[below,rotate=-45]{Lemma B.2, $\|a+b\| \leq \epsilon$};

\draw (-6.7,7.5) node[below]{$a+b=0$};
\draw (0,7.5) node[below]{$a=0$};
\draw (6.5,0) node[above]{$b=0$};
\draw (-5,0) node[below]{Lemma B.1, $\|b\| \leq \eta$};
\draw (-5,-5) node[below]{Lemma B.4};
\draw (3,-5) node[below]{Lemma B.4};
\draw (-5,2.5) node[below]{Lemma B.4};
\draw (5.5,-2.5) node[below]{Lemma B.4};
\draw (5.5,4.5) node[below]{Lemma B.4};
\draw (-2.2,4.5) node[below]{Lemma B.4};
\draw (4,0) node[below]{Lemma B.1, $\|b\| \leq \eta$};
\draw (0,5) node[below,rotate=90]{Lemma B.1, $\|a\| \leq \eta$};
\draw (0,-5) node[below,rotate=90]{Lemma B.1, $\|a\| \leq \eta$};
\draw (-0.86,-0.1) node[below]{Lemma B.3};
\draw (0.2,1.6) node[below]{$\|a\|^2+\|b\|^2<\eta$};
%\draw (0.9,0.6) node[below]{$|b|<\eta$};

\end{tikzpicture}
\caption{\label{fig:roadmap} Roadmap of the proof of Proposition \ref{prop:CS} with the associated partition of $\RR^d \times \RR^d$.}
\end{figure}
\end{center}

The main result is stated below and the demonstration follow the sketch of proof described above.
\begin{prop} 
\label{prop:CS}
If $\phi$ satisfies $\HS$ and $\HL$,  then a constant $c>0$ exists such that $\forall (a,b) \in \RR^d \times \RR^d$:
\begin{equation}\label{eq:cs}
\left| \langle \phi - \phi_a,\phi_{a+b}-\phi_a \rangle \right| \leq \left\|\phi-\phi_a\right\|_2  \left\|\phi_{a+b}-\phi_a\right\|_2 \left( 1 - c \left\|\phi-\phi_{a+b}\right\|_2^2\right).
\end{equation}
\end{prop}

\begin{proof} The proof relies on a partition of $ \RR^d \times \RR^d$ that is detailed in Figure \ref{fig:roadmap}.

Note that when $a=0$ or $b=0$, Inequality \eqref{eq:cs} is trivial. We then consider the cases where $a \neq 0$ and $b \neq 0$.

Around the diagonal $a+b=0$, Lemmas \ref{lemma:diagonal} (far from the origin) and \ref{lemma:0diagonal} (near the origin) show that a couple $(\epsilon,c_{\epsilon})$ exists such that:
$$
\|a+b\| \leq \epsilon \Longrightarrow R(a,b) \leq 1 - c_{\epsilon} \|\phi_{a+b}-\phi\|_2^2.
$$
Therefore, Inequality \eqref{eq:cs} is true near the diagonal when $|a+b| \leq \epsilon$.

Now, outside the diagonal, Lemma \ref{lemma:out} shows that a constant for the value of $\epsilon>0$ found above, a constant $\tilde{c}_{\epsilon}$ exists such that:
$$
\|a+b\| \geq \epsilon \Longrightarrow R(a,b) \leq 1 - \tilde{c}_{\epsilon}.
$$
Since $\|\phi_{a+b}-\phi\|_2^2\leq 2$, it also implies that:
$$
\|a+b\| \geq \epsilon \Longrightarrow R(a,b) \leq 1 - \frac{\tilde{c}_{\epsilon}}{2} \|\phi_{a+b}-\phi\|_2^2.
$$
Then, Equation \eqref{eq:cs} holds outside the diagonal, it ends the proof.
\end{proof}

\subsection{Technical lemmas}\label{sec:teclemma}
\subsubsection{Properties of the location model $(\phi_a)_{a\in \RR^d }$}
In the following text, we will have to compute several Taylor's expansions that involve $(\phi_a)_{a \in \RR^d}$ and its successive derivatives. The $d$-dimensional Euclidean scalar product is denoted by:
$$
\forall (x,y) \in \RR^d \times \RR^d \qquad
x \ps y := \sum_{i=1}^d x_i y_i.
$$
This notation should be distinguished from the one of the scalar product among $\mathbb{L}^2$ functions:
$
\langle f,g\rangle = \int f(x) g(x) dx.
$
Finally, note that for any differentiable functions, the derivative of any function $f : \RR^d \longrightarrow \RR$ in any direction $e \in \mathbb{S}^1$  in any position $x \in \RR^d$ is
$$
d_{e}[f](x):=
\lim_{s \longrightarrow 0} \frac{f(x+s e) - f(x) }{s}.$$
Now, some standard arguments of geometry yield
$$
d_{e}[f](x)
 =  \nabla f(x)\ps e.
$$
We also introduce the successive derivation notation applied on  a twice differentiable function $f$:
$$
\forall (u,v) \in \mathbb{S}^1 \times \mathbb{S}^1 \quad
\forall x \in \RR^d  \qquad d_{u,v}[f]:= d_{u}[d_v[f]].
$$
Note that if $f$ is $\mathcal{C}^2(\RR^d)$, the Schwarz equality holds
$
d_{u,v}[f]=  d_{v,u}[f].
$

\begin{prop}\label{prop:phitec}
If the density $\phi$ satisfies $\HL$ and $\HS$, then for any unitary vectors $(u,v) \in \mathbb{S}^1 \times \mathbb{S}^1$:
\begin{itemize}
\item[$(i)$] $\langle \phi,u \bullet \nabla \phi \rangle =\langle \phi,d_u[\phi] \rangle =  0$.
\item[$(ii)$]$\langle d_u[\phi],d_{u,v}[\phi]\rangle = 0$.
\item[$(iii)$] $\langle d_u[\phi],d_{u,v,v}[\phi]\rangle = 
- \langle d_{u,u}[\phi],d_{v,v}[\phi]\rangle$
\item[$(iv)$] For any $a \in (\RR^d)^{\star}$ and $e \in \mathbb{S}^1$,  $\nabla \phi \ps e$ and $\phi-\phi_a$ are not proportional.
\end{itemize}
\end{prop}

\begin{proof}

Item $(i)$
If $\phi$ is $\mathcal{C}^1$, then the conclusion is immediate because
$$
d_e\left[\frac{\phi^2}{2}\right]= \phi\,  e \ps \nabla \phi.
$$
Since $e$ is a unit vector, we can find an orthonormal basis $(e_1=e,e_2,\ldots,e_d)$ and $(i)$ then comes from direct integration over $\RR^d$ of $d_e[\phi^2/2]$ because the Jacobian of the change of basis has value $1$.

%
% Otherwise, $\phi$ is piecewise $\mathcal{C}^3$, and we can find a finite set of points $a_{-p}=-\infty<a_{-(p-1)}<\ldots<a_{-1} \leq a_0=0 \leq a_1<\ldots a_p = +\infty$, with $a_{-j}=-a_j$, such that $\phi$ is $\mathcal{C}^1$ on each segment $[a_j,a_{j+1}]$. Integrating by part on each segment, we have:
%$$
%\langle \phi,\phi' \rangle  = \sum_{j=-p}^{p-1} \int_{a_j}^{a_{j+1}} \phi'(x) \phi(x) dx = \frac{1}{2} \sum_{j = -p}^{p-1} [\phi^2(x)]_{a_j}^{a_{j+1}}.
%$$
% Since $\phi$ is continuous and $\lim_{\pm \infty} \phi=0$, we deduce that $\langle \phi, \phi' \rangle = 0$.

Item $(ii)$ proceeds from the same kind of argument by considering
$$
d_v \left[ \frac{d_u[\phi]^2}{2}\right] = d_u[\phi] d_{u,v}[\phi],
$$
and using a change of coordinate with $v$.
%
%We use the same argument and obtain that:
%$$
%\langle \phi',\phi'' \rangle =  \frac{1}{2} \sum_{j =-p}^{-1} [\phi'^2(x)]_{a_j}^{a_{j+1}} +
%\frac{1}{2} \sum_{j =0}^{p-1} [\phi'^2(x)]_{a_j}^{a_{j+1}}.
%$$
%The function $\phi$ being symmetric, we have: 
%$$\sum_{j =-p}^{-1} [\phi'^2(x)]_{a_j}^{a_{j+1}}  = -  \sum_{j =0}^{p-1} [\phi'^2(x)]_{a_j}^{a_{j+1}},$$ 
%so that $\langle \phi',\phi'' \rangle=0$.

Item $(iii)$: 
this identity is obtained using an integration by parts.

Item $(iv)$: we assume that:
\begin{equation}\label{eq:proportional}
\exists \lambda \in \RR \quad  \forall x \in \RR^d \qquad d_e[\phi](x)= \lambda [\phi(x) -\phi(x-a)]
\end{equation}
If $\lambda \neq 0$,  it implies that $d_e[\phi]$ is continuous everywhere (since $\phi_a$ and $\phi$ are continuous). Considering $x^*\in \arg\max \phi$, we use \eqref{eq:proportional} to obtain:
$$
\nabla \phi(x^*) = 0 \Longrightarrow d_e[\phi](x^*)=0 \Longrightarrow \phi(x^*)=\phi(x^*+a).
$$
In particular, we cannot have $\lim_{\|x\| \longrightarrow + \infty} \phi(x) = 0$, and $\phi \notin \mathbb{L}^2(\RR^d)$. We deduce that, necessarily, $d_e[\phi]=0$ everywhere, meaning that 
$$
\forall x \in \mathbb{R}^d \quad \forall s \in \mathbb{R} \qquad 
\phi(x+s e ) = \phi(x).
$$ This last equality  is  impossible because the location model is identifiable. 
\end{proof}

\subsubsection{Properties of the ratio $R$}

\begin{lemma}\label{lemma:Rcontinu} The function $R$ defined in \eqref{eq:defR} is a continuous function on $\{\RR^d\}^\star \times \{\RR^d\}^\star$ and is bounded from above by $1$. Moreover, we have:
$$
R(a,b)=1 \Longleftrightarrow a+b=0.
$$
Finally, we have
$$
\forall b \in \{\RR^d\}^\star \quad \forall e \in \mathbb{S}^1 \qquad 
\lim_{s \longrightarrow 0} R(s e,b) = \frac{|\langle d_e[\phi],\phi_b-\phi\rangle|}{\|d_e[\phi]\|_2 \|\phi-\phi_b\|_2}<1,
$$
and
$$ 
\forall (e,e') \in \mathbb{S}^1\times \mathbb{S}^1 \quad e \neq e' \Longrightarrow
\lim_{(s,s') \longrightarrow 0} R(s e,s' e') = \frac{|\langle d_e[\phi],d_{e'}[\phi]\rangle|}{\|d_e[\phi]\|_2 \|d_e'[\phi]\|_2}<1.
$$
\end{lemma}
\begin{proof}
The continuity of $R$ when $a \neq 0$ and $b \neq 0$ is clear from the Lebesgue Theorem because $\HL$ implies that
$|\phi(x-a)-\phi(x)| \leq \|a\| g(x)$ with $g \in \mathbb{L}^2(\RR^d)$. We now consider the behaviour of $R$ when $a$ or $b$ are close to $0$.\\

$\bullet$
When $b \neq 0$ is fixed and $a \longrightarrow 0$, the assumption $\HL$ implies that
$|\phi(x-a)-\phi(x)| \leq \|a\| g(x)$ with $g \in \mathbb{L}^2(\RR^d)$. We can apply the Lebesgue Theorem and obtain, when $a =s e \longrightarrow 0$,
\begin{eqnarray*}
N(s e,b)&=&\int [\phi(x)-\phi_{s e}(x)][\phi_{s e+b}(x)-\phi_{s e}(x)] dx\\& \sim &s  \int d_e[\phi](x) [\phi_b(x)-\phi(x)] dx  \quad \mathrm{when} \ a\rightarrow 0.
\end{eqnarray*}
A similar computation shows that, when $a=s e\rightarrow 0$,
$$
D(s e,b) \sim s \sqrt{\int d_e[\phi]^2(x) dx} \sqrt{\int [\phi(x)-\phi_b(x)]^2 dx}.
$$
Hence, $R(s e , b)$ has a limit when $s \longrightarrow 0$ and $b \neq 0$ is fixed. For the sake of convenience, we keep the notation $R(0,b)$ to refer to this limit and the Cauchy-Schwarz inequality shows that:
$$
R_e(0,b) := \lim_{s \longrightarrow 0} R(s e ,b)  = \frac{|\langle e \ps \nabla \phi,\phi_b-\phi\rangle|}{\|e \ps \nabla \phi\|_2 \|\phi_b-\phi\|_2} \leq 1.
$$
For symmetry reasons in $a$ and $b$, the same results hold for $a \longmapsto R_e(a,0)$.\\

$\bullet$ The situation may be dealt with similarly near  $(0,0)$, the  Lebesgue Theorem yields:
$$
\frac{|N(s e,s' e')|}{D(s e,s' e')} \underset{(s,s') \longrightarrow (0,0)}{=}  \frac{|\langle d_e[\phi],d_{e'}[\phi]\rangle|}{\|d_e[\phi]\|_2 \|d_{e'}[\phi]\|_2}.
$$
If  $d_e[\phi]$ and $d_{e'}[\phi]$ were proportional, then $\lambda$ exists such that $d_e[\phi]=\lambda d_{e'}[\phi]$ everywhere, meaning that for all $x$ in $\RR^d$, the function $s \longmapsto \phi(x+s(e- \lambda e'))$ is constant, which is impossible because considering the variation of $\phi$ on the line $x^\star + s (e- \lambda e')$ where $x^\star =  \arg \max \phi$. 
Therefore, the limit is also strictly lower than $1$.
\end{proof}

The next lemma concerns the behavior of $R$ around the diagonal $a+b=0$ when $a$ or $b$ are not close to $0$.
\begin{lemma}\label{lemma:diagonal}
For any $\eta>0$, we can find $\epsilon>0$ such that: 
$$
\forall \|a\| \geq \eta \quad \forall \|h\| \leq \epsilon \qquad 
R(a,-a+h) \leq 1 - c_{\eta} \|\phi_h-\phi\|_2^2.
$$
\end{lemma}
\begin{proof}

To establish the desired inequality, remark that: 	
\begin{eqnarray}
\lefteqn{R(a,-a+h) \leq 1 - c \|\phi_h-\phi\|_2^2}\nonumber\\
 &\Longleftrightarrow & N(a,-a+h) \leq D(a,-a+h) - c  \|\phi_h-\phi\|_2^2 D(a,-a+h)\nonumber \\
&\Longleftrightarrow &D(a,-a+h)-N(a,-a+h) > c  \|\phi_h-\phi\|_2^2 D(a,-a+h).\label{eq:cns}
\end{eqnarray}

\underline{\textit{Point 1): Taylor expansion of $N$ and $D$.}}

We use a Taylor expansion when $h=o(1)$ and compute:
\begin{eqnarray*}
N(a,b) &= & N(a,-a+h) = \langle \phi-\phi_a,\phi_h-\phi_a \rangle\\
& = & \|\phi-\phi_a\|_2^2 -  \langle h \ps \nabla \phi,\phi-\phi_a \rangle  + \frac{1}{2} ^t h \frac{\langle D^2 \phi,\phi-\phi_a \rangle}{2} h + o(\|h\|^2).\\
\end{eqnarray*}
where the $o(\|h\|^2)$ is uniform in $a \in B(0,\eta)^c$.
In the meantime, we have:
\begin{eqnarray*}
\lefteqn{D(a,b) =  D(a,-a+h)}\\
&  = & \|\phi-\phi_a\|_2 \sqrt{\|\phi-\phi_a\|_2^2 - 2   \langle h \ps \nabla \phi,\phi-\phi_a \rangle +  \|h \ps \nabla \phi\|_2^2 + ^t h \langle D^2 \phi,\phi-\phi_a\rangle h + o(\|h\|^2)}\\
&=& \|\phi-\phi_a\|_2^2 \sqrt{1- \frac{ 2 \langle h \ps \phi,\phi-\phi_a \rangle}{\|\phi-\phi_a\|_2^2}+\frac{\|h \bullet \nabla \phi\|_2^2+^t h \langle D^2 \phi,\phi-\phi_a\rangle h}{\|\phi-\phi_a\|_2^2}  + o(\|h\|^2)} \\
& = &  \|\phi-\phi_a\|_2^2-    \langle h \ps \nabla \phi,\phi-\phi_a \rangle \\
&& + \left( \frac{\|h \ps \nabla \phi\|_2^2}{2} + \frac{^t h\langle D^2 \phi,\phi-\phi_a\rangle h}{2} - \frac{ \langle h \ps \nabla \phi,\phi-\phi_a \rangle^2}{2 \|\phi-\phi_a\|_2^2}\right)  +o(\|h\|^2),
\end{eqnarray*}
where the $o(\|h\|^2)$ is uniform in $a \in B(0,\eta)^c$.
Consequently, we obtain:
\begin{eqnarray}\label{eq:D-N}
\lefteqn{D(a,-a+h)-N(a,-a+h)}\nonumber\\
&=&\frac{1}{2 \|\phi-\phi_a\|_2^2}  \left[ \|h \ps \nabla \phi\|_2^2 \|\phi-\phi_a\|_2^2 - \langle h \ps \nabla  \phi,\phi-\phi_a\rangle^2 \right]  +o(\|h\|^2).
\end{eqnarray}
The main term of the right hand side is obviously non negative from the Cauchy-Schwarz inequality. But it requires a deeper inspection to establish Inequality \eqref{eq:cns}.
We introduce the following parametrization: $h = \|h\| e$ where $e \in \mathbb{S}^1$. Equation \eqref{eq:D-N} yields
$$
D(a,-a+h)-N(a,-a+h) = \frac{|h|^2}{2 \|\phi-\phi_a\|_2^2}
\psi(e,a)+ o(\|h\|^2),
$$
where the $o(\|h\|^2)$ is uniform in $a \in B(0,\eta)^c$ with $\psi$ given by
$$\psi(e,a)=\|e \ps \nabla \phi\|_2^2 \|\phi-\phi_a\|_2^2 - \langle e \ps \nabla  \phi,\phi-\phi_a\rangle^2.$$
 We shall prove that
$$
\min_{e \in \mathbb{S}^1} \min_{ a \in B(0,\eta)^c} \psi(e,a)>0.
$$

\underline{\textit{Point 2): $\psi$ is uniformly lower bounded.}}

We remark first that $\psi$ is continuous over $\mathbb{S}^1 \times B(0,\eta)^{c}$ and for any vector $e \in \mathbb{S}^1$ and any $a \in B(0,\eta)^{c}$, we know that $\psi(e,a)>0$ since we have seen in the proof of Lemma \ref{lemma:Rcontinu} that $d_e[\phi]$ and $\phi-\phi_a$ cannot be proportional each other.

We study the behaviour of $\psi$ when $\|a\| \longrightarrow + \infty$ \textit{uniformly in $e \in \mathbb{S}^1$}. A straightforward application of Proposition \ref{prop:decorrelation} shows that
$$
\|\phi-\phi_a\|_2^2  = 2 - 2 \langle \phi,\phi_a\rangle \longrightarrow 2 \quad \text{as} \quad \|a\| \longrightarrow + \infty.
$$
Hence, a large enough $A$ exists such that $\|a\| \geq A \Longrightarrow \|\phi-\phi_a\|_2^2  \geq 3/2$.
In the meantime, we have
$$
[\langle e \ps \nabla \phi,\phi-\phi_a \rangle]^2 = 
[\underbrace{ \langle e \ps \nabla \phi,\phi \rangle}_{:=0 \, \text{from Proposition \ref{prop:phitec}}} -\langle e \ps \nabla \phi_a \rangle]^2 = \langle e \ps \nabla \phi, \phi_a \rangle^2 \leq \|e\ps \nabla \phi \|_2^2.
$$
Therefore, we deduce that
\begin{equation}\label{eq:mino_psi}
\|a\| \geq A \Longrightarrow \forall e \in \mathbb{S}^1 \quad \psi(e,a) \geq \frac{\|e \ps \nabla \phi\|_2^2}{2}.
\end{equation}
Now, $e\in \mathbb{S}^1 \longmapsto \|e \ps \nabla \phi\|_2^2$ is a continuous map that does not vanish on $\mathbb{S}^1$,  otherwise $\phi$ would be constant on each line parallel to a direction of $\mathbb{S}^1$, and in particular would be constant on a line passing through $x^\star$. The compactness of $\mathbb{S}^1$ implies that 
$$
m := \inf_{e \in \mathbb{S}^1}  \|e \ps \nabla \phi\|_2^2  >0.
$$
This last bound used in Equation \eqref{eq:mino_psi} yields
$
\|a\| \geq A \Rightarrow \inf_{ e \in \mathbb{S}^1}  \psi(e,a) \geq \frac{m}{2}.
$
Consequently, $\psi$ is uniformly lower bounded by $\tilde{m}_{\eta}>0$ over $\mathbb{S}^1 \times B(0,\eta)^c$.

\underline{\textit{Point 3): Final inequality}}
We can gather the conclusions of Point 1) and Point 2) and obtain that for any $\eta>0$, a small enough $\epsilon$ exists such that 
$$
\forall a \in B(0,\eta)^c \quad \forall \|h\| \leq \epsilon\qquad 
D(a,-a+h)-N(a,-a+h) \geq \frac{\|h\|^2}{4 \|\phi-\phi_a\|_2^2} \tilde{m}_{\eta}.
$$
%
%Note that $\psi$ is a continuous function of $a$ from the Lebesgue Theorem and as pointed out above, it is impossible for $\phi'$ to be proportional to $\phi-\phi_a$, which in turn implies with the Cauchy-Schwarz inequality that
%$$
%\forall a \neq 0 \qquad \psi(a) >0.
%$$
%We can study the limit of $\psi$ when $|a| \longrightarrow + \infty$.
%From $(i)$ of Proposition \ref{prop:phitec}:
%$$
%| \langle \phi',\phi-\phi_a \rangle| \leq |\langle \phi',\phi \rangle| + | \langle \phi',\phi_a \rangle | \leq \|\phi'\| \|\phi_a\| = \|\phi'\|.
%$$
%Moreover, 
%$$
%\lim_{|a| \longrightarrow + \infty} \|\phi'\|^2 \|\phi-\phi_a\|^2 = 2 \|\phi'\|^2.
%$$
%Hence, we can find $A$ large enough such that:
%$$
%|a| \geq A \Longrightarrow \psi(a) \geq \frac{\|\phi'\|^2}{2} >0,
%$$
%and we conclude that:
%\begin{equation}\label{eq:minopsi}
%\min_{|a| \geq \eta} \psi(a) = m_{\eta} > 0.
%\end{equation}
%
%
%
%Equations \eqref{eq:D-N} and \eqref{eq:minopsi} show that an $\epsilon$
%exists such that:
%$$
%\forall |h| \leq \epsilon \qquad 
%D(a,-a+h)-N(a,-a+h) \geq \frac{m_{\eta}}{4 \|\phi-\phi_a\|^2}  h^2.
%$$
Since $\|\phi-\phi_a\|_2^2$ and $D$ are upper bounded by $2$, we deduce that:
$$
\forall (a,h) \in B(0,\eta)^c \times B(0,\epsilon)
 \qquad 
(D-N)(a,-a+h)  \geq \frac{m_{\eta}}{16} D(a,-a+h)  \|h\|^2.
$$
This inequality associated with 
$$\|\phi_h-\phi\|_2^2 = \| h \ps  \nabla \phi\|_2^2 + o(\|h\|^2) \leq \|h\|^2 \sup_{e \in \mathbb{S}^1} \|e \ps \nabla \phi\|_2^2 + o(\|h\|^2)$$ leads to
 the desired inequality \eqref{eq:cns} with $c=\frac{\tilde{m}_{\eta}}{32 \sup_{e \in \mathbb{S}^1} \|e \ps \nabla \phi\|_2^2}$.
\end{proof}

The next lemma concerns the behavior of $R$ around the origin $(0,0)$.
\begin{lemma}\label{lemma:0diagonal}
Two constants $(\eta,c_{\eta}) \in \RR_+^2$  exist such that:
$$
\|a\| \vee \|b\| \leq \eta \Longrightarrow R(a,b) \leq 1 - c_{\eta} \|\phi_{a+b}-\phi\|_2^2
$$
\end{lemma}
\begin{proof}
To  study $R$ around the origin, we write $a=r e$ and $b=\tr e$ with $(e,\te)\in \mathbb{S}^1 \times \mathbb{S}^1$ and remark that a third order Taylor expansion yields (below we skip the dependency in $\phi$ for the sake of convience and just write $d_e$ instead of $d_e[\phi]$):
\begin{align*}
\phi_a-\phi_{a+b} &=\tr d_{\te}-r \tr d_{e \te}- \frac{\tr^2}{2} d_{\te \te}
+\frac{r \tr^2}{2}d_{e\te\te}+\frac{\tr r^2}{2}d_{\te e e}+\frac{\tr^3}{6} d_{\te\te\te} + o((r\vee \tr) ^3),
\end{align*}
while 
$$
\phi-\phi_a = r d_{e}  - \frac{r^2}{2} d_{ee} +\frac{r^3}{6}d_{eee} +o(r^3).
$$
We can use these third order expansions in $N(a,b)$:
\begin{align*}
N(a,b)  &= r \tr \langle d_e ,d_{\te} \rangle - 
 \frac{r \tr^2}{2} \langle d_e ,d_{\te\te} \rangle - \frac{\tr r^2}{2} \langle d_{ee},d_{\te }\rangle  + \frac{r^2 \tr^2}{2}\langle d_{ee},d_{\te }\rangle + \frac{r^2 \tr^2}{2} \langle d_e,d_{e\te\te}\rangle\\
 & +\frac{\tr r^3}{2} \langle d_e,d_{\te e e}\rangle+ \frac{r \tr^3}{6} \langle d_e,d_{\te\te\te}\rangle + \frac{\tr r^3}{2}\langle d_{ee},d_{\te\te}\rangle+\frac{r^2 \tr^2}{4} \langle d_{ee},d_{\te\te}\rangle+\frac{\tr r^3}{6} \langle d_{\te},d_{eee}\rangle \\ &+ o((r\vee \tr) ^3)
\end{align*}
Now, using Proposition \ref{prop:phitec} $iii)$, we obtain  that
\begin{align*}
N(a,b) &= r \tr \left[ \langle d_e,d_{\te}\rangle - \frac{\tr}{2} \langle d_e,d_{\te\te}\rangle - \frac{r}{2} \langle d_{\te},d_{ee}\rangle + \frac{r \tr}{4} \langle d_e,d_{e\te\te}\rangle + \frac{\tr^2}{6} \langle d_e,d_{\te\te\te}\rangle + \frac{r^2}{6} \langle d_{\te},d_{eee}\rangle\right]\\ &+ o((r\vee \tr) ^3).
\end{align*}
Similar computations on $D(a,b)$ with $a=r e$ and $b=\tr \te$ yield:
$$
D(a,b) = r \tr \left[ \|d_e\|_2 \|d_{\te}\|_2 - \frac{1}{24} \left(
 \tr^2 \|d_{\te \te}\|_2^2 \frac{\|d_e\|_2}{\|d_{\te}\|_2}
 +r^2 \|d_{e e}\|_2^2 \frac{\|d_{\te}\|_2}{\|d_{e}\|_2}
\right)\right]+ o((r\vee \tr) ^3)
$$

We then consider the two possible situations: either $e=\te$ or $e \neq \te$.

\underline{Case $e = \te$:} in that situation, the expression of $N$ is simpler because of  Proposition \ref{prop:phitec} $ii)$ and we have
$$
N(a,b) = r \tr \left[ \|d_e\|_2^2- \frac{r \tr}{4} \|d_{ee}\|_2^2 - \frac{r^2+\tr^2}{6} \|d_{ee}\|_2^2\right]+ o((r\vee \tr) ^3)
$$
In that case, we then obtain
\begin{eqnarray*}
D(a,b)-|N(a,b)|& \geq&  r \tr \|d_{ee}\|_2^2 \left[ - \frac{r^2+\tr^2}{24} + \frac{r^2+\tr^2}{6} + \frac{r \tr}{4} \right] + o(r^2+\tr^2) \\
& = &\frac{3\|d_{ee}\|_2^2 }{24} r \tr  (r+\tr)^2o((r\vee \tr) ^3).
\end{eqnarray*}
Using the argument in Equation \eqref{eq:cns} again, we can check that:
$$
c \|\phi_{a+b}-\phi\|_2^2 D(a,b) \sim  c \underbrace{\|d_e\|_2^2 (r+\tr)^2}_{\|\phi_{a+b}-\phi\|_2^2} \times \underbrace{r \tr \|d_e\|_2^2}_{ D(a,b) } =c r \tr  (r+\tr)^2 \|d_e\|_2^4,
$$
which means that if $c< \min_{e \in \mathbb{S}^1}\frac{3\|d_{ee}[\phi]\|_2^2}{24 \|d_e[\phi]\|_2^4}$, then \eqref{eq:cns} holds for small enough $r$ and $\tr$. This is possible since for any vector $e$ in $\mathbb{S}^1$, $d_{ee}[\phi]$ does not vanish (otherwise $\phi$ would not be a density) and is a continuous function of $e$ on a compact space.

\underline{Case $e \neq \te$:} The situation is less intricate in that situation because the first order terms are not of the same size
$$
D(a,b)-N(a,b)=r \tr \left[ \|d_e\|_2 \|d_{\te}\|_2 - \langle d_e , d_{\te}\rangle \right] + o(r \vee \tr).
$$
Applying the Cauchy-Schwarz inequality, we check that
$\|d_e\|_2 \|d_{\te}\|_2 - \langle d_e , d_{\te}\rangle >0$ since $d_e$ and $d_{\te}$ are not proportional.
\end{proof}

The remaining lemma studies the behavior of $R$ outside the diagonal.
\begin{lemma}\label{lemma:out}
For any $\epsilon>0$, a constant $c_{\epsilon}$ exists such that: 
$$
\|a+b\|\geq \epsilon \Longrightarrow R(a,b) \leq 1- c_{\epsilon}.
$$
\end{lemma}
\begin{proof}

Consider the function $\varphi: h \longmapsto |\langle \phi,\phi_h \rangle| = \langle \phi,\phi_h\rangle$, the last equality resulting from the positivity of $\phi$ and $\phi_h$.
The dominated convergence theorem
shows that $\varphi$ is continuous and the Cauchy-Schwarz inequality implies that $\varphi$ is a bounded function whose values belong to $[0,1]$. From the identifiability result of Proposition \ref{prop:id},  we then have:
$$
\varphi(h)=1 \Longleftrightarrow h=0.
$$
Finally, Proposition \ref{prop:decorrelation} implies that $\lim_{\|h\| \longmapsto + \infty} \varphi(h)=0$. Taken together, these elements show that for any $\epsilon>0$, $\varphi$ attains its upper bound on 
$B(0,\epsilon)^{c}$. It yields:
\begin{equation}\label{eq:etaeps}
\forall \epsilon >0 \quad \exists \eta_{\epsilon}>0 \qquad 
\sup_{\|h\| \geq \epsilon} \varphi(h) \leq 1-\eta_{\epsilon}.
\end{equation}

$\bullet$
We first consider the case where $\|a\| \wedge \|b\|  \longrightarrow + \infty$ with $\epsilon \leq \|a+b\|$. In that case, if we denote $h=a+b$ and use
$\lim_{\|a\| \longrightarrow + \infty} \langle \phi,\phi_a \rangle = 0,$
 then we can find $M_{\epsilon}$ large enough such that:
\begin{eqnarray*}
\lefteqn{\|a\| \wedge \|b\|  \geq M_{\epsilon}} &\Longrightarrow  \\
&    
\frac{|N(a,b)|}{D(a,b)} = \frac{\left|1+\langle \phi,\phi_h\rangle-\langle \phi , \phi_a\rangle - \langle \phi,\phi_b\rangle\right| }{\|\phi-\phi_a\|_2 \|\phi-\phi_b\|_2} \leq \frac{1+\sup_{\epsilon \leq |h|} \varphi(h)}{2} \times \frac{1-\frac{\eta_{\epsilon}}{3}}{1-\frac{\eta_{\epsilon}}{2}} \leq 1-\frac{\eta_{\epsilon}}{3},
\end{eqnarray*}
where $\eta_{\epsilon}$ is defined in \eqref{eq:etaeps}.\\

$\bullet$
We now consider the case where $\|a\| \longrightarrow + \infty$ although $|b|$ remains bounded by $M_{\epsilon}$, so that $b \in B(0,M_{\epsilon}) \setminus \{0\}$. In that case, we compute:
$$
 N(a,b) =   \left| \langle \phi , \phi_{a+b} \rangle -  \langle \phi , \phi_{a} \rangle -  \langle \phi , \phi_{b} \rangle+ \|\phi_a\|_2^2 \right| \longrightarrow 1-\langle \phi , \phi_b \rangle \quad \text{if} \quad \|a\| \longrightarrow + \infty.
$$
At the same time, we also consider $D$ and check that:
$$
D(a,b) = \|\phi-\phi_a\|_2 \|\phi_{a+b}-\phi_a\|_2 \longrightarrow 2 \sqrt{1-\langle \phi, \phi_b \rangle }
 \quad \text{when} \quad \|a\| \longrightarrow + \infty.
$$
We then obtain:
$$
\lim_{\|a\|\longrightarrow + \infty} R(a,b) = \frac{\sqrt{1-\langle \phi,\phi_b\rangle}}{2} \leq \frac{1}{2}.
$$
Hence, we can find a constant $A_{\epsilon}$ sufficiently large such that:
$$
\forall \|a\| \geq A_{\epsilon} \quad \forall b \in B(0,M_{\epsilon}) \quad R(a,b) \leq \frac{3}{4}.
$$

$\bullet$
If $a$ and $b$ now belong to the compact set:
$$
\mathcal{E}_{\epsilon} := \left\{(a,b) \in \RR^2 \, : \, \|a\| \leq A_{\epsilon}, \, \|b\| \leq M_{\epsilon}, \, \|a+b\| \geq \epsilon\right\},
$$
we know that $R$ is a continuous function on $\mathcal{E}_{\epsilon,A,M}$ and attains its upper bound, which is strictly lower than $1$ by the Cauchy-Schwarz inequality. Consequently, 
$$
\exists \tilde{\eta}_{\epsilon} >0 \, \forall (a,b) \in \mathcal{E}_{\epsilon} \qquad 
R(a,b) \leq 1-\tilde{\eta}_{\epsilon}.
$$
Taking all the bounds obtained  outside of the diagonal together, we obtain the lemma with $c_\epsilon = (\tilde{\eta}_{\epsilon} \wedge \eta_{\epsilon}/3 \wedge 1/4)$.
\end{proof}

\section{Proofs of the lower bounds}\label{sect:LWB}
%Before stating intermediary technical results, we introduce a sub-class of densities $\phi$ that satisfy Assumption  $\HD$  introduced below.
%\paragraph{Assumption $\HD$} \textit{The density $\phi$ satisfies:}
%$$
%\cI_\phi := \sup_{1 \leq j \leq d} \int \{d_{j,j}\phi(x)\}^2 \phi^{-1}(x) dx < +\infty,
%$$
%where $d_{j,j}$ refers to the second derivative of $\phi$ with respect to the variable $j$.
%Note that Assumption $\HD$ is needed for our lower bound results (see Section \ref{s:lb})  but is not necessary to obtain good estimation properties. However, this assumption is very mild and is again satisfied for many probability distributions as pointed out in Remark \ref{rem:1}. Moreover, from the minimax paradigm, it is enough to obtain our lower bound results with a restricted subset of densities $\phi$.
%

\subsection{Asymmetric risk}
We begin by a useful lemma, which is a generalization of the Le Cam method for proving lower bounds if the loss involved in the statistical model is not symmetric, meaning that $\rho(\theta_1,\theta_2)$ is generally not equal to $\rho(\theta_2, \theta_1)$, but still satisfies a weak triangle inequality. Hence, the Le Cam Lemma requires a small modification in the spirit of the remark of \cite{Yu} (Example 2, Section 3).

In the sequel, $d_{\mbox{\tiny TV}}(\mathbb{P},\mathbb{Q})$ and $\KL(\mathbb{P},\mathbb{Q})$ denote the total variation distance and the Kullback-Leibler divergence between two measures, $\mathbb{P}$ and $\mathbb{Q}$, respectively. 
\begin{lemma}
\label{propgen-lowerbound}
Let $(\PP_\theta)_{\theta\in\Theta}$ be a family of measures indexed by $\Theta$ and assume that $\rho:(\theta_1,\theta_2)\in \Theta^2 \mapsto \rho(\theta_1,\theta_2)\in\RR^+$ satisfies the weak triangle inequality:
\begin{equation}
\label{condrho}
\forall (\theta_1,\theta_2,\theta_3)\in \Theta^3,\ 
\rho(\theta_1,\theta_3) + \rho(\theta_2,\theta_3) \geq \rho(\theta_1,\theta_2) \wedge \rho(\theta_2,\theta_1).
\end{equation}
Let $\Phi:\RR^+ \rightarrow  \RR^+$ be a non-decreasing function. 
Let $\delta>0$ and $(\theta_1,\theta_2)\in \Theta^2$ such that 
$\rho(\theta_1,\theta_2) \wedge \rho(\theta_2,\theta_1) \geq 2 \delta$. Then,
\begin{eqnarray*}
\underset{\hat \theta}{\inf}\ \underset{\theta\in\Theta}{\sup}\ \EE\left[ \Phi(\rho(\theta,\hat \theta)) \right] 
& \geq &  \frac{\Phi(\delta)}{2} \left\{1  - d_{\mbox{\tiny TV}}(\PP^{\otimes^n}_{\theta_1},\PP^{\otimes^n}_{\theta_2})\right\},\\
& \geq & 
\frac{\Phi(\delta)}{2} \left\{1 - \sqrt{\frac n 2 \KL(\PP_{\theta_1},\PP_{\theta_2})}\right\},
\end{eqnarray*}
where the infimum is taken over all estimators $\hat \theta$. 
\end{lemma}

\begin{proof}
First, we observe that:
$$
\EE[\Phi(\rho(\theta,\hat \theta))] \geq \Phi(\delta) \PP(\rho(\theta,\hat \theta) \geq \delta ),
$$
since $\Phi$ is a non-decreasing function. 
Let $\cV=\{1,2\}$ and $\Psi(\hat \theta)= \underset{v\in \cV}{\argmin}\ \rho(\theta_v,\hat \theta)$. \\
We can show that $\rho(\theta_v,\hat \theta) < \delta$ implies that $\Psi(\hat \theta) = v$.  According to Condition \eqref{condrho}, we have:
$$
\rho(\theta_v,\hat \theta) \geq \rho(\theta_v,\theta_{v'}) \wedge \rho(\theta_{v'},\theta_{v})  - \rho(\theta_{v'},\hat \theta) > 2 \delta - \rho(\theta_{v'},\hat \theta).
$$
Now, if $\rho(\theta_v,\hat \theta) < \delta$, then $\delta > 2 \delta - \rho(\theta_{v'},\hat \theta)$, so that
$\rho(\theta_{v'},\hat \theta) > \delta$, which is necessarily larger than $\rho(\theta_v,\hat \theta)$. Hence, we obtain $\Psi(\hat \theta) = v$.\\
Equivalently, for $v\in\{1,2\}$, we have $\Psi(\hat \theta) \neq v \Longrightarrow \rho(\theta_v,\hat \theta) > \rho(\theta_{v'},\hat \theta)$ since:
$$
2 \delta \leq \rho(\theta_v,\theta_{v'}) \wedge \rho(\theta_{v'},\theta_{v}) \leq \rho(\theta_v,\hat \theta) +\rho(\theta_{v'},\hat \theta) \leq 2 \rho(\theta_v,\hat \theta).
$$

The rest of the proof proceeds from the standard Le Cam argument:  $\Phi$ is non decreasing so that:
\begin{eqnarray*}
\underset{\theta\in\Theta}{\sup}\ \EE[\Phi(\rho(\theta,\hat \theta))] 
&\geq& \Phi(\delta)\ \underset{\theta\in\Theta}{\sup}\  \PP(\rho(\theta,\hat\theta) \geq \delta)\\
&\geq& \frac{\Phi(\delta)}{2} \{\PP(\rho(\theta_1,\hat \theta)\geq \delta) + \PP(\rho(\theta_2,\hat \theta)\geq \delta)\}\\
&\geq&\frac{\Phi(\delta)}{2} \{\PP^{\otimes^n}_{\theta_1}(\Psi(\hat \theta)\neq 1) + \PP^{\otimes^n}_{\theta_2}(\Psi(\hat \theta)\neq 2)\}.
\end{eqnarray*}
Taking an infimum over all tests $\Psi$ (see, \textit{e.g.}, \cite{LY}) we obtain:
\begin{eqnarray*}
\underset{\hat \theta}{\inf}\ \underset{\theta\in\Theta}{\sup}\ \EE[\Phi(\rho(\theta,\hat\theta))] 
&\geq & \frac{\Phi(\delta)}{2} \underset{\Psi}{\inf} \{\PP^{\otimes^n}_{\theta_1}(\Psi \neq 1) + \PP^{\otimes^n}_{\theta_2}(\Psi \neq 2)\}\\
&\geq&  \frac{\Phi(\delta)}{2} \left\{1  - d_{\mbox{\tiny TV}}(\PP^{\otimes^n}_{\theta_1}, \PP^{\otimes^n}_{\theta_2})\right\}.
\end{eqnarray*}
Pinsker's inequality:
$$
d_{\mbox{\tiny TV}}(\PP^{\otimes^n}_{\theta_1} ,\PP^{\otimes^n}_{\theta_2}) \leq \sqrt{\frac 1 2 \KL(\PP^{\otimes^n}_{\theta_1}, \PP^{\otimes^n}_{\theta_2})} = \sqrt{\frac n 2 \KL(\PP_{\theta_1},\PP_{\theta_2})}
$$ ends the proof.
\end{proof}

\subsection{Lower bound for the strong contamination model}\label{sub:LWBR1}

We now study the lower bounds in the first regime, namely when $\|\mu\|$ is lower bounded by a constant $m$ that is independent of $n$.\\

{\sc Proof of Theorem \ref{th:lb:R1}}

\paragraph{Item $(i)$} We apply Lemma \ref{propgen-lowerbound} with $\Phi(t)=t^2$ and the loss function $\rho$ defined as:
 $$
\forall (\theta_1,\theta_2) \in \Thh^2 \qquad 
 \rho(\theta_1,\theta_2)= \lambda_1 \|\mu_1 - \mu_2\|.$$
Remark that $\rho$ satisfies the weak triangle inequality \eqref{condrho}. Indeed, for all $ (\theta_1,\theta_2,\theta_3) \in \Thh^3$, we have:
\begin{eqnarray*} 
\rho(\theta_1,\theta_3) + \rho(\theta_2,\theta_3) &=& \lambda_1 \|\mu_1 - \mu_3\| + \lambda_2 \|\mu_2 - \mu_3\|\\
&\geq& \min(\lambda_1,\lambda_2) \|\mu_1 - \mu_2\| \\
&\geq & \rho(\theta_1,\theta_2) \wedge \rho(\theta_2,\theta_1).
\end{eqnarray*}

We introduce the subset
$$
\Thhb:= \left\{ \theta= (\lambda,\mu) : \frac{c}{\|\mu\|^2 \sqrt{n} } \leq \lambda \leq \bar{\lambda},\  m \leq \|\mu\| \leq  M   \right\}
$$
where $0<m<M$ and $0<\frac{c}{m^2 \sqrt{n} } < \bar\lambda <1$. Then, $\Thhb \subset \Thh$. 
We consider $\theta_1=(\lambda,\mu_1)$ and $\theta_2=(\lambda,\mu_2)$; their values will be chosen later  to ensure that $(\theta_1,\theta_2) \in \Thhb^2$.
According to Lemma~\ref{propgen-lowerbound} applied with $\delta  = \frac{\lambda \|\mu_1-\mu_2\|}{2}$, 
we can write:
\begin{eqnarray}
\underset{\hat\theta}{\inf}\  \underset{\theta\in\Thh}{\sup}\ \EE[\lambda^2 \|\hat \mu - \mu\|^2] 
&\geq&  \underset{\hat\theta}{\inf}\  \underset{\theta\in\Thhb}{\sup}\ \EE[\lambda^2 \| \hat \mu - \mu\|^2]\nonumber\\ 
&\geq& \frac{\delta^2}{2} \left\{1- \sqrt{\frac n 2 \KL(\PP_{\theta_1},\PP_{\theta_2})}\right\}.\label{eq:mino_minimax}
\end{eqnarray}

We can compute the Kullback-Leibler divergence between the two mixtures $\PP_{\theta_1}$ and $\PP_{\theta_2}$: if 
$f_1 =(1-\lambda) \phi + \lambda \phi_{\mu_1}$ (resp. $f_2 =(1-\lambda) \phi + \lambda \phi_{\mu_2}$) is the density of $\PP_{\theta_1}$ (resp. $\PP_{\theta_2}$) w.r.t. the Lebesgue measure, we have: 
\begin{eqnarray*}
\KL(\PP_{\theta_1},\PP_{\theta_2})
&=& \int \log\left[\frac{f_1(x)}{f_2(x)}\right] f_1(x) dx\\
&=& \int \log\left[1 + \frac{f_1(x) - f_2(x)}{f_2(x)}\right] f_1(x) dx\\
&\leq & \int  \frac{f_1(x) - f_2(x)}{f_2(x)} f_1(x)  dx,
\end{eqnarray*}
where we used the inequality $\log (1+t) \leq t$. If we once again write $f_1=f_2+f_1-f_2$, we obtain:
\begin{eqnarray*}
\KL(\PP_{\theta_1},\PP_{\theta_2}) &\leq& \int \frac{f_1(x) - f_2(x)}{f_2(x)}  \left[f_2(x)+f_1(x)-f_2(x)\right] dx\\
& = & \int \frac{[f_1(x) - f_2(x)]^2}{f_2(x)}dx \\
&\leq& \lambda^2  \int \frac{[\phi_{\mu_1}(x) - \phi_{\mu_2}(x)]^2}{(1- \lambda)\phi(x)+\lambda \phi_{\mu_{2}}(x)}  dx
\end{eqnarray*}
since $f_2(x) \geq (1-\lambda)\phi(x)$ and $f_1(x)-f_2(x) = \lambda [\phi_{\mu_1}(x) - \phi_{\mu_2}(x)]$. On the basis of Assumption $\HL$, we know that $|\phi_{\mu_1}-\phi_{\mu_2}| \leq \|\mu_1-\mu_2\| g$ and we obtain:
\begin{equation}\label{eq:majoKL}
\KL(\PP_{\theta_1},\PP_{\theta_2}) \leq  \frac{\lambda^2 \|\mu_1 - \mu_2\|^2 \cJ}{1- \bar\lambda},
\end{equation}
where $\cJ:= \|g \phi^{-1/2}\|_2^2$ is the constant involved in $\HL$. \\

We now choose $\lambda,\mu_1$ and $\mu_2$ so that we obtain the largest possible value in \eqref{eq:mino_minimax}, while satisfying the constraints given in $\Thhb$. Without loss of generality, we set $\mu_1^{(1)}<\mu_2^{(1)}$ and we need to find a choice of these parameters such that $m\leq \mu_{1}^{(1)}< \mu_2^{(1)}\leq M$ and $\frac{c}{(\mu_1^{(1)})^2 \sqrt{n}}\leq \lambda \leq \bar{\lambda}$. We set $\mu_1=(\mu_1^{(1)},0,\ldots,0)$ and $\mu_2=(\mu_2^{(1)},0,\ldots,0)$ so that
$$
\mu_1^{(1)}=m \qquad \text{and} \qquad \lambda=\frac{c}{m^2 \sqrt{n}}<\bar{\lambda}.
$$
For a given $\epsilon>0$, we choose $\mu_2^{(1)}$ such that
$
\frac{n}{2} \KL(\PP_{\theta_1},\PP_{\theta_2}) \leq 1-\epsilon.
$
Using \eqref{eq:majoKL}, we arrive at the calibration:
$$
\mu_2^{(1)}-\mu_1^{(1)}=\sqrt{\frac{2(1-\bar{\lambda}) (1-\epsilon)}{\lambda^2 \cJ n} }.
$$
It remains to check that $\mu_2^{(1)}\leq M$. From our choice of $\lambda$ and $\mu_1^{(1)}$, we see that:
$$
\mu_2^{(1)}= m \left[ 1+\sqrt{\frac{2 (1-\bar \lambda) m^2}{c^2 \cJ}(1-\epsilon)}\right] \leq m \left[ 1+\sqrt{\frac{2 m^2 (1-\epsilon)}{c^2 \cJ}}\right],
$$
which can be made smaller than $M$ if $1-\epsilon \leq \frac{c^2 \cJ (M-m)^2}{2 m^4}$. If we plug these choices of $\lambda,\mu_1$ and $\mu_2$ into \eqref{eq:mino_minimax}, we obtain:
\begin{eqnarray*}
	\underset{\hat\theta}{\inf}\  \underset{\theta\in\Thhb}{\sup}\ \EE[\lambda^2 \|\hat \mu - \mu\|^2] &\geq&
%\lambda^2 \times \frac{2 (1-\bar{\lambda}) (1-\epsilon)	}{8  \lambda^2 \cJ n} \times \left[1-\sqrt{1-\epsilon} \right]\\
%& \geq &
 \frac{(1-\bar{\lambda}) (1-\epsilon) \epsilon}{8  \cJ n},
	\end{eqnarray*}
	which is the desired lower bound of the minimax risk \eqref{eq:lbmu:R1}.
\paragraph{Item $(ii)$} We keep the same $\Phi$ and define $\rho(\theta_1,\theta_2) = |\lambda_1 - \lambda_2| = \rho(\theta_2,\theta_1)$. We consider $\theta_1=(\lambda_1,\mu)$ and $\theta_2=(\lambda_2,\mu)$ such that $|\lambda_1 - \lambda_2| = \frac{\epsilon}{\sqrt n}$ and 
$$\frac{c}{m^2 \sqrt n} = \lambda_1 < \lambda_2 \leq \bar\lambda,$$
$\mu$ and $\epsilon$ have to be chosen hereafter. Since $\lambda_2 = \lambda_1+\frac{\epsilon}{\sqrt n} \leq \bar{\lambda}$, we must choose $\epsilon$ such that:
\begin{equation}\label{eq:cond1eps} 
\epsilon \leq \bar{\lambda} \sqrt{n} - \frac{c}{m^2},
\end{equation}
which is possible since we assumed that $\frac{c}{m^2 \sqrt{n}} < \bar{\lambda}$.
From Lemma ~\ref{condrho},

\begin{eqnarray*}
\underset{\hat\theta}{\inf}\  \underset{\theta\in\Thh}{\sup}\ \EE[(\lambda-\hat\lambda)^2] 
&\geq& \underset{\hat\theta}{\inf}\  \underset{\theta\in\Thhb}{\sup}\ \EE[(\lambda-\hat\lambda)^2] \\
&\geq& \frac{\epsilon^2}{2 n} \left\{1 - \sqrt{\frac n 2 \KL(\PP_{\theta_1},\PP_{\theta_2})}\right\}.
\end{eqnarray*}

We can upper bound the Kullback-Leibler divergence as:
\begin{eqnarray*}
\KL(\PP_{\theta_1},\PP_{\theta_2})&\leq& \int  \left[f_1(x) - f_2(x)\right]^2 f_2(x)^{-1} dx \\
&\leq& (\lambda_1 - \lambda_2)^2 \int \left[\phi_{\mu}(x) - \phi(x)\right]^2  f_2(x)^{-1} dx \\
&\leq& \frac{  (\lambda_1 - \lambda_2)^2 \|\mu\|^2 }{1-\bar{\lambda}} \int g(x)^2 \phi(x)^{-1} dx\\
&\leq& \frac{\|\mu\|^2 \epsilon^2 \cJ}{(1-\bar\lambda) n}.
\end{eqnarray*}
By choosing $\mu=(\mu^{(1)},0,\ldots,0)$ with
\begin{equation}\label{eq:cond2eps} 
	\mu^{(1)}=\frac{m+M}{2} \qquad \text{and} \qquad  \epsilon \leq \sqrt{\frac{2 (1-\bar{\lambda})}{ \cJ (m+M)^2}},
\end{equation}
we obtain
$\frac n 2 \KL(\PP_{\theta_1},\PP_{\theta_2}) \leq \frac{1}{4}.$ Considering the minimal admissible value of $\epsilon$ in \eqref{eq:cond1eps} and \eqref{eq:cond2eps} now leads to a choice of the parameters $\theta_1$ and $\theta_2$ such that:
$$
\underset{\hat\theta}{\inf}\  \underset{\theta\in\Thh}{\sup}\ \EE[(\lambda-\hat\lambda)^2] \geq \frac{\epsilon^2}{4 n}. 
$$
This last inequality is the second lower bound \eqref{eq:lblambda:R1}.
\hfill $\square$

\subsection{Lower bound for the weak contamination model}\label{sub:LWBR2}

{\sc Proof of Theorem \ref{th:lb:R2}}

\paragraph{Point $(i)$} We consider $\Phi(t)=t^2$ and the loss function $\rho$ defined as:
$$\rho(\theta_1,\theta_2) = \|\mu_1\|^2 |\lambda_1-\lambda_2|.$$
 Note that $\rho$ satisfies \eqref{condrho} since $\forall (\theta_1,\theta_2,\theta_3) \in \Th^3$,
\begin{eqnarray*}
\rho(\theta_1,\theta_3) + \rho(\theta_2,\theta_3) &=& \|\mu_1\|^2 |\lambda_1 - \lambda_3| + \|\mu_2\|^2 |\lambda_2 - \lambda_3|\\
&\geq& \min(\|\mu_1\|^2,\|\mu_2\|^2) |\lambda_1 - \lambda_2| \\
&\geq & \rho(\theta_1,\theta_2) \wedge \rho(\theta_2,\theta_1).
\end{eqnarray*}

 To obtain a convenient lower bound, we need to use Lemma \ref{propgen-lowerbound} and find a couple of parameters $(\theta_1,\theta_2)$ that belongs to the admissible set and such that $ \KL(\PP_{\theta_1},\PP_{\theta_2})$ is small enough. In particular, the proximity between $\PP_{\theta_1}$ and $\PP_{\theta_2}$ will be obtained by a careful matching of the first moments of the two distributions, which is a good method for obtaining efficient lower bounds in mixture models (see, \textit{e.g.}, \cite{BG14} or \cite{KH15}). We give an example of this method below. First, remark that: 
$$
	\KL(\PP_{\theta_1},\PP_{\theta_2}) = \int \log \left[\frac{f_1(x)}{f_2(x)}\right] f_1(x) dx.
$$

Since $\phi$ satisfies $\HS$, then $\phi$ is a $\cC^3$ function on $\RR^d$, considering a shift $\mu = (\mu^{(1)},0,\ldots,0) =  o(1)$, we can write a third order Taylor expansion:
$$
\forall x \in \RR^d \quad \phi_{\mu}(x) = \phi(x) - \mu^{(1)} d_1\phi(x) + \frac{ \{\mu^{(1)}\}^2 d_{11}\phi(x)}{2}- \frac{\{\mu^{(1)}\}^{3}}{6} d_{111}\phi(\xi_{x,\mu}),
$$
where $\xi_{x,\mu}$ belongs to the interval defined by $x$ and $x-\mu$ and $d_1\phi$ (resp. $d_{11}\phi$ and $d_{111}\phi$) denotes the first (resp. second and third) partial derivative of $\phi$ w.r.t. the first coordinate of $x$.
 In particular, assuming that $d_{111}\phi$ is bounded on $\RR^d$ leads to:
$$
\forall x \in \RR^d \quad \phi_{\mu}(x) = \phi(x) - \mu^{(1)} d_1\phi(x) + \frac{\{\mu^{(1)}\}^2}{2} d_{11}\phi(x) +o(\|\mu\|^2).
$$
This Taylor expansion permits us to write, for small values of $\mu_1^{(1)}$:
{\small
\begin{eqnarray*}
\log [f_1(x)] &=& \log [(1-\lambda_1) \phi(x) + \lambda_1 \phi_{\mu_1}(x)]\\
&=& \log \left[(1-\lambda_1) \phi(x) + \lambda_1 \phi(x) - \lambda_1 \mu_1^{(1)} d_1\phi(x) + \frac 1 2 \lambda_1 \{\mu_1^{(1)}\}^2 d_{11}\phi(x) + o (\|\mu_1\|^2)\large\right]\\
&=& \log \left[\phi(x)\right] + \log \left[1 - \lambda_1 \mu_1^{(1)} \frac{d_1\phi(x)}{\phi(x)} + \frac 1 2 \lambda_1 \{\mu_1^{(1)}\}^2 \frac{d_{11}\phi(x)}{\phi(x)}+ o (\|\mu_1\|^2)\right]\\
&=& \log \left[\phi(x)\right] - \lambda_1 \mu_1^{(1)} \frac{d_1\phi(x)}{\phi(x)} + \frac 1 2 \lambda_1 \{\mu_1^{(1)}\}^2 \frac{d_{11}\phi(x)}{\phi(x)} \\ &&- \frac 1 2 \lambda_1^2 \{\mu_1^{(1)}\}^2 \left(\frac{d_1\phi(x)}{\phi(x)}\right)^2 + o(\|\mu_1\|^2).
\end{eqnarray*} 
In the same way, for small values of $\mu_2$:
\begin{eqnarray*}
\log [f_2(x)] &=& \log [(1-\lambda_2) \phi(x) + \lambda_2 \phi_{\mu_2}(x)]\\
&=& \log \left[\phi(x)\right] - \lambda_2 \mu_2^{(1)} \frac{d_1\phi(x)}{\phi(x)} + \frac 1 2 \lambda_2 \{\mu_2^{(1)}\}^2 \frac{d_{11}\phi(x)}{\phi(x)}  \\
&&- \frac 1 2 \lambda_2^2 \{\mu_2^{(1)}\}^2 \left(\frac{d_1\phi(x)}{\phi(x)}\right)^2 + o(\|\mu_2\|^2).
\end{eqnarray*} }
We thus obtain:

 \begin{eqnarray*}
 \lefteqn{
\log [f_1(x)] - \log [f_2(x)]} \\ &= (\lambda_2 \mu_2^{(1)}  - \lambda_1 \mu_1^{(1)} ) \frac{d_1\phi(x)}{\phi(x)} 
                                    +  \frac 1 2 (\lambda_1 \{\mu_1^{(1)}\}^2 -  \lambda_2 \{\mu_2^{(1)}\}^2 ) \frac{d_{11}\phi(x)}{\phi(x)} \\
                                 &   + \frac 1 2 ( \lambda_2^2 \{\mu_2^{(1)}\}^2 - \lambda_1^2\{\mu_1^{(1)}\}^2) \left(\frac{d_1\phi(x)}{\phi(x)}\right)^2+ o(\|\mu_1\|^2) + o(\|\mu_2\|^2).
\end{eqnarray*}
In particular, we observe that the term above can be considered as a ``second order term" if $\theta_1$ and $\theta_2$ are chosen such that $\lambda_1  \mu_1^{(1)}= \lambda_2 \mu_2^{(1)}$, which corresponds to the first moment of $\PP_{\theta_1}$ and $\PP_{\theta_2}$. If $\lambda_1  \mu_1^{(1)}= \lambda_2  \mu_2^{(1)}$, we obtain:
\begin{eqnarray*}
 \lefteqn{
             \log [f_1(x)] - \log [f_2(x)]} \\
                                 &=  \frac 1 2 (\lambda_1 \{\mu_1^{(1)}\}^2  -  \lambda_2 \{\mu_2^{(1)}\}^2) \frac{d_{11}\phi(x)}{\phi(x)} + o(\|\mu_1\|^2) + o(\|\mu_2\|^2).
\end{eqnarray*}

We deduce that:
{\small \begin{eqnarray*}
\lefteqn{
\KL(\PP_{\theta_1},\PP_{\theta_2})}\\
&=& \int \left[ \frac 1 2 (\lambda_1 \{\mu_1^{(1)}\}^2 -  \lambda_2 \{\mu_2^{(1)}\}^2 ) \frac{d_{11}\phi(x)}{\phi(x)} +o (\|\mu_1\|^2)+o (\|\mu_2\|^2)
  \right] f_1(x) dx\\
&=&\frac 1 2 (\lambda_1 \{\mu_1^{(1)}\}^2 -  \lambda_2 \{\mu_2^{(1)}\}^2 ) \left[ (1- \lambda_1) \int d_{11}\phi(x) dx   + \lambda_1 \int \frac{d_{11}\phi(x) \phi(x-\mu_1)}{\phi(x)} dx \right]\\ & &  + o (\|\mu_1\|^2)+o (\|\mu_2\|^2).
\end{eqnarray*}}
The smoothness of $\phi$ leads to $\int d_{11}\phi(x) dx =0$. We deduce that:
{\small
\begin{eqnarray*}
\lefteqn{\int \frac{d_{11}\phi(x) \phi(x-\mu_1)}{\phi(x)} dx  }\\
&=& \int \frac{d_{11}\phi(x) }{\phi(x)} [\phi(x) - \mu_1^{(1)} d_1\phi(x) + \frac{\{\mu_{1}^{(1)}\}^2}{ 2} d_{11}\phi(x) + o (\|\mu_1\|^2)]dx \\
&=& \int d_{11}\phi(x) dx - \mu_1^{(1)} \int \frac{d_{11}\phi(x) d_{1}\phi(x)}{\phi(x)}dx + \frac 1 2 \{\mu_1^{(1)}\}^2 \int \frac{\{d_{11}\phi(x)\}^2}{\phi(x)} dx + o(\mu_1^2)dx.
\end{eqnarray*}}
Now, we choose for the density $\phi$ an even function ($\phi(x)=\phi(-x)$ for all $x\in \RR^d$) and we obtain that
$$
\KL(\PP_{\theta_1},\PP_{\theta_2})
= \frac 1 2 \{\mu_1^{(1)}\}^2 \cI_\phi + \underset{n\to+\infty}{o} (\|\mu_1\|^2),$$
where the last line comes from the fact that $x\mapsto d_{11}\phi(x) d_1\phi(x) / \phi(x)$ is an odd function and the definition of $\cI_\phi$ (see \eqref{IHD}).
Finally, since $\lambda_1\mu_1^{(1)} = \lambda_2\mu_2^{(1)}$, we deduce that:
\begin{eqnarray}
\KL(\PP_{\theta_1},\PP_{\theta_2}) 
&=& \frac 1 4 (\lambda_1 \{\mu_1^{(1)}\}^2 -  \lambda_2 \{\mu_2^{(1)}\}^2 ) \lambda_1 \|\mu_1\|^2 \cI_\phi  + o(\|\mu_1\|^4)\nonumber\\
&=& \frac 1 4 \left(1 - \frac{\lambda_1}{\lambda_2}  \right) \lambda_1^2 \|\mu_1\|^4 \cI_\phi + o(\|\mu_1\|^4).\label{eq:kldl}
\end{eqnarray}
Next, let $\bar\lambda\in(0,1)$. Choosing $\lambda_2= \frac{\bar\lambda}{2} < \bar \lambda$ and $\lambda_1 =  \frac 1 \alpha \lambda_2$ with $\alpha=\frac{1+\sqrt{5}}{2}$, we have: 
$$
\left(1-\frac{\lambda_1}{\lambda_2}\right)\lambda_1^2 =(\lambda_1-\lambda_2)^2.
$$
Thus,
$$
\KL(\PP_{\theta_1},\PP_{\theta_2})  = \frac 1 4 (\lambda_2 - \lambda_1)^2 \|\mu_1\|^4 \cI_\phi + o(\|\mu_1\|^4).
$$

In order to apply Lemma \ref{propgen-lowerbound}, let $\delta>0$ such that $2 \delta = \rho(\theta_1,\theta_2)\wedge \rho(\theta_2,\theta_1)$. According to our constraint $\lambda_1 \mu_1^{(1)} = \lambda_2 \mu_2^{(1)}$ and $\lambda_2 = \alpha \lambda_1 > \lambda_1$, we observe that $\mu_2^{(1)} < \mu_1^{(1)}$ so that: 
$$
 	2 \delta = \|\mu_2\|^2 |\lambda_1-\lambda_2|.
$$
We deduce that: 
$$
	|\lambda_1-\lambda_2| \|\mu_1\|^2 = |\lambda_1-\lambda_2| \left(\frac{\lambda_2}{\lambda_1}\right)^2 \|\mu_2\|^2 = 2 \delta \alpha^2
$$
and 
$$
	\|\mu_1\|^2 = \left(\frac{\lambda_2}{\lambda_1}\right)^2 \|\mu_2\|^2 = \alpha^2 \frac{4\alpha}{(\alpha-1)\bar\lambda} \delta.
$$
Thus,
$$
	\KL(\PP_{\theta_1},\PP_{\theta_2})  = \delta^2 \alpha^4 \cI_\phi + o(\delta^2),
$$
and according to Lemma \ref{propgen-lowerbound}, we obtain:
$$
\underset{\hat\theta}{\inf}\  \underset{\theta\in\Th}{\sup}\ \EE[\|\mu\|^4 (\lambda-\hat\lambda)^2] \geq \frac{\delta^2}{2} \left\{1 -  \sqrt{\frac n 2  \delta^2 \left[ \alpha^4 \cI_\phi + o(1)  \right]}\right\}.
$$
The choice of $\delta$ is determined by the right brackets that should be non-negative. We can choose:
$$
\delta = \left[2 n \alpha^4 \cI_\phi\right]^{-\frac 1 2},
$$
so that  $  \frac n 2  \delta^2 \left[ \alpha^4 \cI_\phi + o(1)  \right] = \frac 1 4 (1 + o(1)). 
$
Thus, an integer $N$ exists such that: 
$$
\forall n \geq N \qquad 
\underset{\hat\theta}{\inf}\  \underset{\theta\in\Th}{\sup}\ \EE[\|\mu\|^4 (\lambda-\hat\lambda)^2]  \geq  \frac{\delta^2}{6}  = \frac{1}{12 \alpha^4 \cI_\phi n}. 
$$
This ends the proof of the first point.

\paragraph{Point $(ii)$}

We define the loss function $\rho(\theta_1,\theta_2) = \lambda_1 \|\mu_1\|  \|\mu_1 -  \mu_2\|$ and $\Phi(t)=t^2$. 
The function $\rho$ satisfies the weak triangle inequality \eqref{condrho}:
\begin{eqnarray*}
\lefteqn{\forall (\theta_1,\theta_2,\theta_3) \in \Th^3:} \\
& \rho(\theta_1,\theta_3) + \rho(\theta_2,\theta_3) =& \lambda_1 \|\mu_1\| \|\mu_1 - \mu_3\| + \lambda_2 \|\mu_2\| \|\mu_2 - \mu_3\|\\
& &\geq \min(\lambda_1 \|\mu_1\|,\lambda_2\|\mu_2\|) \|\mu_1 - \mu_2\| \\
&&\geq  \rho(\theta_1,\theta_2) \wedge \rho(\theta_2,\theta_1).
\end{eqnarray*}

The proof follows the same lines as the ones of $(i)$ and our starting point is once again the Kullback-Leibler divergence asymptotics given in Equation \eqref{eq:kldl}. Our baseline relationship $\lambda_1 \mu_1 = \lambda_2 \mu_2$ is still necessary and we obtain while choosing $\mu_1=(\mu_1^{(1)},0,\ldots,0)$ and  $\mu_2=(\mu_2^{(1)},0,\ldots,0)$:
$$
\KL(\PP_{\theta_1},\PP_{\theta_2}) 
= \frac{\cI_\phi}{4} \left(1 - \frac{\lambda_2}{\lambda_1}  \right) \lambda_1^2 \mu_1^4 + o(\|\mu_1\|^4).
$$
We choose $\mu_1=2 \mu_2$ so that $\lambda_2=2 \lambda_1$ and:
$$
\rho(\theta_1,\theta_2) \wedge \rho(\theta_2,\theta_1) = \lambda_1 \|\mu_1\| \|\mu_1-\mu_2\| =  \frac 1 2 \lambda_1 \|\mu_1\|^2 := 2 \delta.
$$
The coefficients $\lambda_1$ and $\lambda_2$ can be made explicit, e.g., $\lambda_1=\bar{\lambda}/2$ and $\lambda_2=\bar{\lambda}$.  This choice implies that $\mu_1^{(1)} = 2 \sqrt{2 \delta/\bar{\lambda}}$. These settings can be used in the result of Lemma \ref{propgen-lowerbound} and we obtain:
$$
\underset{\hat\theta}{\inf}\  \underset{\theta\in\Th}{\sup}\ \EE[\lambda^2 \mu^2 (\mu-\hat\mu)^2]   
\geq \frac{\delta^2}{2} \left\{1 - \sqrt{\frac{n\delta^2}{2}   \left[2 \cI_\phi+ o(1)\right]}\right\}.
$$
We can obtain an efficient lower bound by choosing: 
$$
\delta_n :=\frac{1}{2 \sqrt{n \cI_\phi}},
$$
which implies, of course, that $\mu_1=o(1)$ and $\mu_2=o(1)$.
According to this choice, an integer $N$ exists such that $\forall n \geq N$:
\begin{eqnarray*}
\underset{\hat\theta}{\inf}\  \underset{\theta\in\Th}{\sup}\ \EE[\lambda^2 \|\mu\|^2 \|\mu-\hat\mu\|^2] 
  & \geq & \frac{1}{8 n \cI_\phi} \times (1-\frac{1}{2})/2 =  \frac{1}{32 n \cI_\phi}.
\end{eqnarray*} 
This ends the proof of the second point.
\hfill
$\square$

\section*{Acknowledgments}
This work was partially supported by the French Agence Nationale de la Recherche (ANR-
13-JS01-0001-01, project MixStatSeq).

\bibliographystyle{siam}         %alpha
\bibliography{GMM_bib}

\end{document}